\documentclass[11pt]{amsart}
\usepackage[top=1.5in, bottom=1.5in, left=1.4in, right=1.4in]{geometry}
\usepackage{color}
\usepackage{graphicx}
\usepackage{epstopdf}
\usepackage[table,dvipsnames]{xcolor}
\usepackage{hyperref}

\usepackage{amssymb}
\usepackage{amsthm}
\usepackage{array}
\usepackage{mathtools}
\usepackage{enumerate}

\allowdisplaybreaks

\usepackage[all]{xy}
\usepackage{graphicx} 
\usepackage{tikz}
\usetikzlibrary{patterns}
\usetikzlibrary{decorations.pathreplacing}
\usepackage{tkz-graph}
\usetikzlibrary{arrows}
\usetikzlibrary{decorations.markings}
\usepackage{cellspace}
\newtheorem{thm}{Theorem}[section]
\newtheorem{lemma}[thm]{Lemma}
\newtheorem{cor}[thm]{Corollary}
\newtheorem{prop}[thm]{Proposition}
\newtheorem{conj}[thm]{Conjecture}

\newtheorem{prob}[thm]{Problem}

\newtheorem{Definition}[thm]{Definition}
\newenvironment{definition}
  {\begin{Definition}\rm}{\end{Definition}}

\newtheorem{Example}[thm]{Example}
\newenvironment{example}
  {\begin{Example}\rm}{\end{Example}}

\newtheorem{Remark}[thm]{Remark}
\newenvironment{remark}
  {\begin{Remark}\rm}{\end{Remark}}

\numberwithin{equation}{section}

\def\SetFancyGraph {
	\SetVertexMath
	\GraphInit[vstyle=Art]
	\SetUpVertex[MinSize=2pt]
	\SetVertexLabel
	\tikzset{VertexStyle/.style = {shape = circle,shading = ball,ball color = black,inner sep = 1.5pt}}
	\SetUpEdge[color=black]
	\tikzset{->-/.style={decoration={ markings, mark=at position 0.8 with {\arrow{>}}},postaction={decorate}}}
	\tikzset{->--/.style={decoration={ markings, mark=at position 0.55 with {\arrow{>}}},postaction={decorate}}}
}

\usepackage{etoolbox}
\apptocmd{\sloppy}{\hbadness 10000\relax}{}{}

\newcommand{\emailhref}[1]{\email{\href{#1}{#1}}}

\title[Minuscule doppelg\"{a}ngers, the CDE property, and rowmotion]{Minuscule doppelg\"{a}ngers, the coincidental down-degree expectations property, and rowmotion}
\author[S. Hopkins]{Sam Hopkins}\emailhref{shopkins@umn.edu}
\address{School of Mathematics, University of Minnesota, Minneapolis, MN 55455}

\date{\today}

\begin{document}

\begin{abstract}
We relate Reiner, Tenner, and Yong's \emph{coincidental down-degree expectations (CDE)} property of posets to the \emph{minuscule doppelg\"{a}nger pairs} studied by Hamaker, Patrias, Pechenik, and Williams. Via this relation, we put forward a series of conjectures which suggest that the minuscule doppelg\"{a}nger pairs behave ``as if'' they had isomorphic comparability graphs, even though they do not. We further explore the idea of minuscule doppelg\"{a}nger pairs pretending to have isomorphic comparability graphs by considering the \emph{rowmotion} operator on order ideals. We conjecture that the members of a minuscule doppelg\"{a}nger pair behave the same way under rowmotion, as they would if they had isomorphic comparability graphs. Moreover, we conjecture that these pairs continue to behave the same way under the \emph{piecewise-linear} and \emph{birational} liftings of rowmotion introduced by Einstein and Propp. This conjecture motivates us to study the \emph{homomesies} (in the sense of Propp and Roby) exhibited by birational rowmotion. We establish the birational analog of the antichain cardinality homomesy for the major examples of posets known or conjectured to have finite birational rowmotion order (namely: \emph{minuscule posets} and \emph{root posets of coincidental type}).
\end{abstract}

\maketitle

\section{Introduction} \label{sec:intro}

Let $P$ be a finite poset. The \emph{down-degree} of $p\in P$ is the number of elements of $P$ which $p$ covers. Consider two probability distributions on $P$: the uniform distribution; and the distribution where $p\in P$ occurs proportional to the number of maximal chains of $P$ containing $p$. We say that $P$ has the \emph{coincidental down-degree expectations (CDE)} property if the expected value of the down-degree statistic is the same for these two distributions. We also say that $P$ \emph{is CDE} for short.

Most posets are not CDE, but, perhaps surprisingly, many posets of interest in algebraic combinatorics are CDE. In~\cite{reiner2018poset}, Reiner, Tenner, and Yong introduced the CDE property and explained its connection to the theory of symmetric functions, tableaux, reduced words, et cetera. They proved a number of results about CDE posets, and also made a number of intriguing conjectures concerning the CDE property and related matters. A steady stream of subsequent work~\cite{hopkins2017cde, rush2016minuscule, fan2019proof, hopkins2018vexillary} resolved most of these conjectures (always in the affirmative), to the point where there now remains only one conjecture from~\cite{reiner2018poset} which has not been resolved. The sole remaining conjecture is that the distributive lattice of order ideals of the ``trapezoid'' poset~$T_{k,n}$ is CDE. The trapezoid poset $T_{3,7}$ is depicted on the right in Figure~\ref{fig:rect_trap}. In this paper, we do not resolve this final conjecture of Reiner-Tenner-Yong, but we do situate it in what we believe is the correct context, and suggest a program which could resolve it. 

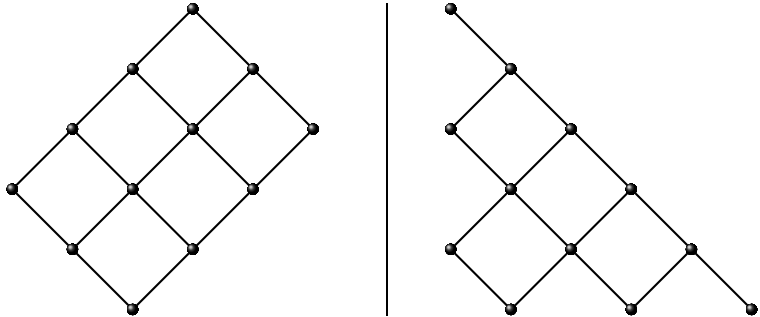
\begin{figure}
\begin{center}
 \begin{tikzpicture}[scale=0.8]
	\SetFancyGraph
	\Vertex[NoLabel,x=0,y=0]{1}
	\Vertex[NoLabel,x=1,y=1]{2}
	\Vertex[NoLabel,x=2,y=2]{3}
	\Vertex[NoLabel,x=3,y=3]{4}
	\Vertex[NoLabel,x=-1,y=1]{5}
	\Vertex[NoLabel,x=0,y=2]{6}
	\Vertex[NoLabel,x=1,y=3]{7}
	\Vertex[NoLabel,x=2,y=4]{8}
	\Vertex[NoLabel,x=-2,y=2]{9}
	\Vertex[NoLabel,x=-1,y=3]{10}
	\Vertex[NoLabel,x=0,y=4]{11}
	\Vertex[NoLabel,x=1,y=5]{12}
	\Edges[style={thick}](1,2)
	\Edges[style={thick}](2,3)
	\Edges[style={thick}](3,4)
	\Edges[style={thick}](1,5)
	\Edges[style={thick}](2,6)
	\Edges[style={thick}](3,7)
	\Edges[style={thick}](4,8)
	\Edges[style={thick}](5,6)
	\Edges[style={thick}](6,7)
	\Edges[style={thick}](7,8)
	\Edges[style={thick}](5,9)
	\Edges[style={thick}](6,10)
	\Edges[style={thick}](7,11)
	\Edges[style={thick}](8,12)
	\Edges[style={thick}](9,10)
	\Edges[style={thick}](10,11)
	\Edges[style={thick}](11,12)
\end{tikzpicture} \qquad \vline \qquad \begin{tikzpicture}[scale=0.8]
	\SetFancyGraph
	\Vertex[NoLabel,x=0,y=0]{1}
	\Vertex[NoLabel,x=2,y=0]{2}
	\Vertex[NoLabel,x=4,y=0]{3}
	\Vertex[NoLabel,x=-1,y=1]{4}
	\Vertex[NoLabel,x=1,y=1]{5}
	\Vertex[NoLabel,x=3,y=1]{6}
	\Vertex[NoLabel,x=0,y=2]{7}
	\Vertex[NoLabel,x=2,y=2]{8}
	\Vertex[NoLabel,x=-1,y=3]{9}
	\Vertex[NoLabel,x=1,y=3]{10}
	\Vertex[NoLabel,x=0,y=4]{11}
	\Vertex[NoLabel,x=-1,y=5]{12}
	\Edges[style={thick}](1,4)
	\Edges[style={thick}](1,5)
	\Edges[style={thick}](2,5)
	\Edges[style={thick}](2,6)
	\Edges[style={thick}](3,6)
	\Edges[style={thick}](4,7)
	\Edges[style={thick}](5,7)
	\Edges[style={thick}](5,8)
	\Edges[style={thick}](6,8)
	\Edges[style={thick}](7,9)
	\Edges[style={thick}](7,10)
	\Edges[style={thick}](8,10)
	\Edges[style={thick}](9,11)
	\Edges[style={thick}](10,11)
	\Edges[style={thick}](11,12)
\end{tikzpicture} 
\end{center}

\caption{On the left: the $3 \times 4$ rectangle. On the right: the trapezoid $T_{3,7}$.} \label{fig:rect_trap}
\end{figure}

Our starting point is the observation that the trapezoid $T_{k,n}$ is the doppelg\"{a}nger of the $k \times (n-k)$ rectangle. The $3 \times 4$ rectangle is depicted on the left in Figure~\ref{fig:rect_trap}. Here two posets  are said to be \emph{doppelg\"{a}ngers} if they have the same number of $P$-partitions of height $\ell$ for all $\ell\geq 1$; that is, $P$ and $Q$ are doppelg\"{a}ngers if the number of order-preserving maps from $P$ to the chain of length~$\ell$ is the same as the number of order-preserving maps from $Q$ to the chain of length~$\ell$ for all~$\ell\geq 1$. This terminology derives from Hamaker, Patrias, Pechenik, and Williams~\cite{hamaker2018doppelgangers}, who highlighted several pairs of doppelg\"{a}ngers which arise in the $K$-theoretic Schubert calculus of minuscule flag varieties. In particular, they gave a bijective proof that these pairs are doppelg\"{a}ngers using a $K$-theoretic version of jeu-de-taquin. The rectangle and trapezoid pair is a prominent example of such a \emph{minuscule doppelg\"{a}nger pair}. In each minuscule doppelg\"{a}nger pair, one poset is a \emph{minuscule poset} (like the rectangle), and the other is a \emph{root poset of coincidental type} or an order filter in one of these root posets (like the trapezoid).

The distributive lattice of order ideals of the rectangle is the most fundamental example of a CDE poset. Indeed, this example precedes the work of Reiner-Tenner-Yong. It was first discovered in the context of the algebraic geometry of curves: establishing that the distributive lattice of order ideals of the rectangle is CDE was the key combinatorial result that Chan, L\'{o}pez Mart\'{i}n, Pflueger, and Teixidor i Bigas~\cite{chan2018genera} needed to reprove a product formula for the genus of one-dimensional Brill-Noether loci. Later, Chan, Haddadan, Hopkins, and Moci~\cite{chan2017expected} generalized the work of~\cite{chan2018genera} by introducing the ``toggle perspective.'' The toggle perspective, based on the notion of ``toggling'' elements into and out of order ideals, is the principal tool we have for establishing that posets are CDE. Essentially all interesting examples of posets known to be CDE can be proved to be CDE via the toggle perspective (see~\cite{hopkins2017cde, rush2016minuscule, hopkins2018vexillary}). For example, as we will see later, the toggle perspective can be used to prove that all distributive lattices corresponding to minuscule posets and to root posets of coincidental type are CDE. (The distributive lattice corresponding to an arbitrary root poset need not be CDE.)

Already in~\cite[Example 4.7]{hopkins2017cde} it was observed that the toggle perspective cannot be applied to the trapezoid. This is because, when it works, the toggle perspective proves something stronger than that a distributive lattice is CDE: it proves that the lattice in question is \emph{toggle CDE (tCDE)}. And the distributive lattice of order ideals of the trapezoid is not tCDE! So some approach beyond the toggle perspective is needed to handle the trapezoid. We suggest such an approach: extend the jeu-de-taquin style bijection of Hamaker et al.~\cite{hamaker2018doppelgangers} from $P$-partitions to \emph{set-valued} $P$-partitions. That is, we conjecture that the minuscule doppelg\"{a}nger pairs don't just have the same structure of $P$-partitions, they in fact have the same structure of set-valued $P$-partitions.

This conjecture asserts that the minuscule doppelg\"{a}nger pairs are similar in ways beyond those which are implied by their being doppelg\"{a}ngers, but which would follow from their having isomorphic comparability graphs. In other words, the minuscule doppelg\"{a}nger pairs behave ``as if'' they had isomorphic comparability graphs. However, for the most part, the minuscule doppelg\"{a}nger pairs do not actually have isomorphic comparability graphs. (Stanley~\cite{stanley1986two} proved that posets with isomorphic comparability graphs are necessarily doppelg\"{a}ngers, but the converse is not true.)

We also put forward another related conjecture which says that the minuscule doppelg\"{a}nger pairs have the same $P$-partition down-degree generating functions. This need not happen for arbitrary doppelg\"{a}ngers, but does happen for posets with isomorphic comparability graphs, and hence is another way in which the minuscule doppelg\"{a}nger pairs behave ``as if'' they had isomorphic comparability graphs.

\begin{prob} \label{prob:main}
Give a conceptual explanation for why the minuscule doppelg\"{a}nger pairs behave ``as if'' they had isomorphic comparability graphs.
\end{prob}

In the second half of this paper we further explore this idea of minuscule doppelg\"{a}nger pairs pretending to have isomorphic comparability graphs by considering rowmotion. \emph{Rowmotion} is an invertible operator acting on the set of order ideals of any poset. It has been the subject of a significant amount of research, with a renewed interest especially in the last ten years~\cite{brouwer1974period, fonderflaass1993orbits, cameron1995orbits, panyushev2009orbits, armstrong2013uniform, striker2012promotion}. The poset on which the action of rowmotion has been studied the most is the rectangle.

We conjecture that the members of a minuscule doppelg\"{a}nger pair have the same orbit structure under rowmotion, and that they have the same down-degree orbit averages. Again, this is not something that automatically happens for doppelg\"{a}ngers; but, as we explain below, it does occur for posets with isomorphic comparability graphs. This suggests further study of rowmotion could help to address Problem~\ref{prob:main}.

Originally the major goal in studying rowmotion was to understand its orbit structure and, in particular, compute its order. More recently, another goal has been to exhibit homomesies for rowmotion. We recall this terminology from Propp and Roby~\cite{propp2015homomesy}: a combinatorial statistic $f$ on a set $X$ is said to be \emph{homomesic} with respect to the action of an invertible operator $\Psi$ on $X$ if the average value of $f$ along every $\Psi$-orbit of $X$ is the same. Homomesies of a dynamical system are in some sense ``dual'' to invariant quantities.

An important example of a statistic that often is homomesic with respect to the action of rowmotion is the \emph{antichain cardinality} statistic, which is in fact the same as down-degree in the lattice of order ideals. (This example of homomesy traces back to a series of influential conjectures of Panyushev~\cite{panyushev2009orbits} and was the main motivating example for Propp-Roby's introduction of the homomesy paradigm.) Whenever one shows that a distributive lattice is tCDE, it follows that the antichain cardinality statistic is homomesic with respect to the action of rowmotion on this lattice. This is a fundamental connection between the study of CDE posets and the study of homomesies for rowmotion.

Rowmotion has been extensively investigated for both minuscule posets and root posets. For instance, Rush and Shi~\cite{rush2013orbits} proved that the action of rowmotion on the set of order ideals of a minuscule poset exhibits the \emph{cyclic sieving phenomenon}. We will explain the cyclic sieving phenomenon precisely later, but in short this means that the order and orbit structure of rowmotion for minuscule posets are completely understood. And even before the work of Rush-Shi, Armstrong, Stump, and Thomas~\cite{armstrong2013uniform} showed that rowmotion acting on the set of order ideals of a root poset (not necessarily of coincidental type) exhibits the cyclic sieving phenomenon. But, as far as we know, before our work it was an open problem to describe the orbit structure of rowmotion acting on the order ideals of the trapezoid, or even to compute its order.\footnote{Very recently, the part of our conjecture concerning the orbit structure of rowmotion for minuscule doppelg\"{a}nger pairs was resolved in the affirmative; see Remark~\ref{rem:reu}.}

Minuscule posets are also known to exhibit the antichain cardinality homomesy for rowmotion~\cite{rush2015orbits, rush2016minuscule} (one way to see this is via the aforementioned ``tCDE implies antichain cardinality homomesy'' fact). Thus, another consequence of our conjecture would be that both members of a minuscule doppelg\"{a}nger pair exhibit this homomesy.

After considering rowmotion of order ideals, we go on to study the way our favorite families of posets (minuscule posets, root posets, and the trapezoid) behave under the \emph{piecewise-linear} and \emph{birational} liftings of rowmotion introduced by Einstein and Propp~\cite{einstein2013combinatorial, einstein2014piecewise}.

We conjecture that the minuscule doppelg\"{a}nger pairs continue to behave the same way with respect to piecewise-linear and birational rowmotion. For most posets, these invertible operators do not even have finite order, but (by results of Grinberg and Roby~\cite{grinberg2016birational1, grinberg2015birational2}) for minuscule posets they do. On the other hand, showing that the trapezoid has finite piecewise-linear and birational rowmotion order remains a significant open problem.

By restricting piecewise-linear rowmotion to the rational points in the order polytope of a poset $P$, we obtain an action of rowmotion on the set of $P$-partitions of height $\ell$. We conjecture that for a minuscule poset $P$ this action of rowmotion on $P$-partitions exhibits cyclic sieving. This cyclic sieving conjecture extends the result of Rush-Shi, and is known to be true in the case of the rectangle by work of Rhoades~\cite{rhoades2010cyclic}. The orbit structure of rowmotion acting on the $P$-partitions of an arbitrary root poset seems unpredictable. But, at least according to our conjectures, for root posets of coincidental type we again get a cyclic sieving result. This is yet another way in which minuscule posets and root posets of coincidental type are well-behaved.

In focusing on the trapezoid in this paper, we have emphasized the limitations of the toggle perspective. But at the end of the paper we also offer a new demonstration of the robustness of the toggle perspective: we show that the ``tCDE implies antichain cardinality homomesy for rowmotion'' fact extends to the piecewise-linear and birational levels. In this way, we obtain birational homomesies for the major examples of posets known or conjectured to have finite birational rowmotion order (namely: minuscule posets and root posets of coincidental type). Furthermore, one part of our conjecture that the minuscule doppelg\"{a}nger pairs behave the same way under birational rowmotion asserts that both members of a minuscule doppelg\"{a}nger pair should exhibit the birational antichain cardinality homomesy.

One surprising aspect of our proof of the birational rowmotion antichain cardinality homomesy for tCDE distributive lattices is that it is an instance where we can deduce a certain result at the piecewise-linear and birational levels from the combinatorial analog of the result in question. If an identity holds at the birational level, then via tropicalization it holds at the piecewise-linear level, and via specialization it also holds at the combinatorial level. Normally implications do not go in the other direction; but in this case we show that they do. 

Here is the structure of the rest of the paper: Section~\ref{sec:background} contains background material; Section~\ref{sec:cde} discusses the CDE property; and Section~\ref{sec:rowmotion} discusses rowmotion.

\medskip

\noindent {\bf Acknowledgements}: I thank Vic Reiner for many useful discussions throughout, and in particular for helpful references. I thank Nathan Williams for useful comments about doppelg\"{a}ngers and bijections, and for making me aware of the paper~\cite{ceballos2014subword}. I thank Joel Kamnitzer, Brendon Rhoades, and Hugh Thomas for useful comments about Conjecture~\ref{conj:minuscule_row_cyc_siev}. Hugh Thomas informed me that he and his collaborators were also thinking about cyclic sieving for minuscule $P$-partitions. I thank Jim Propp for informing me that he had independently and earlier (around 2016) conjectured the Type~A case of Conjecture~\ref{conj:root_poset_cyc_siev}. I thank David Einstein for explaining to me the bounded nature of piecewise-linear rowmotion (see Remark~\ref{rem:bounded}). I thank Bruce Westbury for explaining some possible connections to crystals and cactus group actions (see Remark~\ref{rem:invariant_tensors}). Finally, I thank the anonymous referee for careful attention to the manuscript and useful comments. I was supported by NSF grant~$\#1802920$. Sage mathematics software~\cite{sagemath, Sage-Combinat} was indispensable for testing conjectures. The Sage code used for these tests is available upon request and is included with the arXiv submission of this paper. 

\section{Background} \label{sec:background}

In this section we review the background we will need to state the conjectures and prove the results of the next two sections.

\subsection{Posets, distributive lattices, and doppelg\"{a}ngers} \label{sec:doppelganger_defs}

We assume the reader is familiar with standard terminology and notation for posets as described in, e.g.,~\cite[Chapter 3]{stanley2012ec1}. Throughout the paper {\bf all posets are assumed to be finite}. We represent posets via their Hasse diagrams. We use the notation $[n]\coloneqq \{1,2,\ldots,n\}$, which we also view as a chain poset with the obvious total order.

Since we will be dealing so often with distributive lattices, let us briefly review these. Let $P$ be a poset. We recall that an \emph{order ideal} of $P$ is a subset $I\subseteq P$ for which $y \in I$ and $x \in P$ with $x \leq y$ implies $x \in I$. (A subset $F\subseteq P$ satisfying the dual condition, that $x \in F$ and $y \in P$ with $x \leq y$ implies $y \in F$, is called an \emph{order filter} of $P$.) The set of order ideals of $P$ partially ordered by containment is denoted $J(P)$. This poset~$J(P)$ is in fact a \emph{distributive lattice}, and every finite distributive lattice arises as $J(P)$ for some (unique up to isomorphism) poset $P$: we can recover $P$ from $J(P)$ as the set of ``join-irreducible elements.''

Recall that a poset is \emph{graded of rank $r$} if every maximal chain has length $r$. (We use the convention that the length of a chain $p_0 < p_1 < \cdots < p_\ell$ is $\ell$.) We denote the rank of a graded poset~$P$ by $r(P)$. For a graded poset $P$ we also use $r\colon P \to \mathbb{N}$ to denote the \emph{rank function}: the unique function with $r(q) = r(p)+1$ if $q$ covers $p$, normalized so that minimal elements have rank~$0$. Note that with this normalization convention maximal elements have rank $r(P)$. Distributive lattices are always graded, and~$r(I)=\#I$ for all~$I \in J(P)$. 

A \emph{$P$-partition of height $\ell$} is a weakly order-preserving map $T\colon P \to \{0,1,\ldots,\ell\}$, that is, one for which $p \leq q \in P$ implies $T(p)\leq T(q)$.\footnote{Traditionally (as in~\cite[Chapter 3]{stanley2012ec1}) a $P$-partition is defined to be order-reversing rather than order-preserving; but we follow~\cite{hamaker2018doppelgangers} here in defining it to be order-preserving.} We denote the set of $P$-partitions of height $\ell$ by $\mathrm{PP}^{\ell}(P)$. 

Note that there is a natural bijection between $P$-partitions $T\in\mathrm{PP}^{\ell}(P)$ of height~$\ell$ and nested sequences $I_0 \subseteq I_1 \subseteq \cdots \subseteq I_{\ell-1} \in J(P)$ of $\ell$ order ideals of $P$ given by setting $I_i \coloneqq  T^{-1}(\{0,1,\ldots,i\})$. In other words, $P$-partitions of height~$\ell$ are the same as multichains of $J(P)$ of length $(\ell-1)$. In particular, $P$-partitions of height~$1$ are exactly the same thing as order ideals.

It is well-known that the number $\#\mathrm{PP}^{\ell}(P)$ of $P$-partitions of height $\ell$ is given by a polynomial $\Omega_P(\ell)$ in $\ell$ called the \emph{order polynomial} of $P$.\footnote{The conventional indexing (as in~\cite[Chapter 3]{stanley2012ec1}) would define the order polynomial to be what is in our notation~$\Omega_P(\ell-1)$, but this is immaterial; again, we follow~\cite{hamaker2018doppelgangers}.} The degree of $\Omega_P(\ell)$ is $\#P$ and the leading coefficient is the number of linear extensions of $P$ divided by $\#P!$.

We now come to one of the central definitions in the paper:

\begin{definition}
We say that posets $P$ and $Q$ are \emph{doppelg\"{a}ngers} if $\Omega_P(\ell)=\Omega_Q(\ell)$.
\end{definition}

In some sense the study of doppelg\"{a}ngers goes back to Stanley's introduction of the order polynomial~\cite{stanley1972ordered} (see also the contemporaneous work of Johnson~\cite{johnson1971real}). It was also Stanley in his ``Two poset polytopes'' paper~\cite{stanley1986two} who gave the first interesting sufficient condition for posets to be doppelg\"{a}ngers; we will discuss this in a moment. The ``doppelg\"{a}nger'' terminology, however, is much more recent, first appearing in~\cite{hamaker2018doppelgangers}. Another recent paper concerning doppelg\"{a}ngers is~\cite{browning2017doppelgangers}.

Let us also mention that there are a number of other well-studied problems which are closely related to the problem of understanding when posets are doppelg\"{a}ngers. One is understanding when posets have the same quasisymmetric $P$-partition (or $(P,\omega)$-partition) generating function. This problem has received some attention lately~\cite{mcnamara2014equality, liu2018ppartition}; it is a generalization of the problem of understanding equality of skew Schur functions, which has also received significant attention~\cite{billera2006decomposable, reiner2007coincidences, mcnamara2009towards, mcnamara2014comparing}. Another related problem is: when (and why) do two posets with the same number of elements have the same number of order ideals? For instance, the coincidence of the number of alternating sign matrices and of totally symmetric self-complementary plane partitions can be cast in this language~\cite{striker2011unifying}. 

There is no known exact criterion for when two posets are doppelg\"{a}ngers. In terms of necessary conditions, we have the following (see~\cite[\S3.15]{stanley2012ec1}):

\begin{prop} \label{prop:doppelganger_basics}
Let $P$ and $Q$ be doppelg\"{a}ngers. Then:
\begin{itemize}
\item $\#P=\#Q$;
\item the number of linear extensions of $P$ is the number of linear extensions of $Q$;
\item the length of the longest chain of $P$ is the length of the longest chain of $Q$;
\item $P$ is graded if and only if $Q$ is graded.
\end{itemize}
\end{prop} 

What about sufficient conditions for posets to be doppelg\"{a}ngers? As mentioned in~\cite{hamaker2018doppelgangers}, it is trivial that a poset $P$ and its dual poset~$P^{*}$ are doppelg\"{a}ngers. But what is not mentioned in~\cite{hamaker2018doppelgangers} is that there is a nontrivial extension of this ``a poset and its dual are doppelg\"{a}ngers'' observation based on the notion of comparability graphs. 

The \emph{comparability graph} of a poset $P$, denoted $\mathrm{com}(P)$, is the (undirected, simple) graph with vertices the elements of $P$ and with $p,q \in P$ joined by an edge if and only if~$p$ and $q$ are comparable (i.e., either $p\leq q$ or $q\leq p$). Clearly a poset is not determined by its comparability graph: for instance $P$ and $P^*$ have isomorphic comparability graphs. But nevertheless, the comparability graph does contain a lot of information about the poset, and especially starting in the 1970s there was significant interest in studying \emph{comparability invariants} of posets, i.e., properties of posets that depend only on the comparability graph. 

It turns out that the order polynomial is a comparability invariant.

\begin{thm} \label{thm:com_graph_dop}
Let $P$ and $Q$ be posets with $\mathrm{com}(P)\simeq\mathrm{com}(Q)$. Then $P$ and $Q$ are doppelg\"{a}ngers.
\end{thm}

Theorem~\ref{thm:com_graph_dop} was first proved by Stanley~\cite{stanley1986two}. 

Let us explain two different proofs of Theorem~\ref{thm:com_graph_dop}, as both will be useful for us in our further study of graphs with isomorphic comparability graphs. First let us explain Stanley's original argument, using polytopes. We use $\mathbb{R}^P$ to denote the real vector space of functions $P\to \mathbb{R}$. The \emph{order polytope} of a poset~$P$, denoted $\mathcal{O}(P)$ is the polytope in $\mathbb{R}^{P}$ defined by the inequalities:
\begin{align*}
0 \leq f(p) &\leq 1 &\textrm{ for all $p\in P$};\\
f(p) &\leq f(q) &\textrm{ for all $p\leq q\in P$}.
\end{align*}
The \emph{chain polytope} of $P$, denoted $\mathcal{C}(P)$, is the polytope in $\mathbb{R}^P$ defined by the inequalities:
\begin{align*}
0 \leq f(p) &\leq 1 &\textrm{ for all $p\in P$};\\
\sum_{p\in C} f(p) &\leq 1 &\textrm{ for any chain $C$ of $P$}.
\end{align*}
Stanley proved~\cite[Corollary 1.3, Theorem 2.2]{stanley1986two} that the vertices of the order polytope are the indicator functions of order filters and the vertices of the chain polytope are the indicator functions of antichains. And he defined a \emph{transfer map}\footnote{This transfer map is slightly different than the one in~\cite{stanley1986two}: the difference is essentially given by replacing $P$ by $P^*$. The differences are immaterial and this definition of $\phi$ is cleaner for our later applications.} $\phi\colon \mathbb{R}^{P} \to \mathbb{R}^P$ by
\[\phi(f)(p) = \begin{cases} 1-f(p) &\textrm{if $p$ is maximal in $P$}; \\ \mathrm{min}\{f(q)\colon \textrm{$q\in P$ covers $p$}\} - f(p) &\textrm{otherwise}.\end{cases} \]
Stanley~\cite[Theorem 3.2]{stanley1986two} proved two important things about $\phi$:
\begin{itemize}
\item $\phi$ is bijection from $\mathcal{O}(P)$ to $\mathcal{C}(P)$ (and in fact it is easy to write down the inverse explicitly);
\item $\phi(\frac{1}{\ell}\mathbb{Z}^P \cap \mathcal{O}(P)) = \frac{1}{\ell}\mathbb{Z}^P \cap \mathcal{C}(P)$ for all $\ell \geq 1$.
\end{itemize}
With the transfer map in hand, it is easy to prove Theorem~\ref{thm:com_graph_dop}. Indeed, first note that $\mathrm{PP}^{\ell}(P)$ is in bijection with $\frac{1}{\ell}\mathbb{Z}^P\cap\mathcal{O}(P)$ via the map $T \mapsto \frac{1}{\ell}T$. (In other words $\Omega_P(\ell)$ is the \emph{Ehrhart polynomial} of $\mathcal{O}(P)$.) But then observe that $\mathcal{C}(P)$ only depends on the comparability graph of~$P$. So the transfer map tells us that $\#\mathrm{PP}^{\ell}(P)$ only depends on the comparability graph of $P$ as well.

Next let us explain a different way to prove Theorem~\ref{thm:com_graph_dop} (which Stanley also explained in his original paper~\cite{stanley1986two}). This second approach is based on the fact that there is a useful exact criterion for two posets to have isomorphic comparability graphs. To explain this criterion we need a little terminology. Let $P$ be a poset. A subset $A\subseteq P$ is said to be \emph{autonomous} if each element in $P\setminus A$ has the same order relation to all the elements in~$A$; that is, if
\[\textrm{($x \leq y$ if and only if $x' \leq y$) and ($y \leq x$ if and only if $y \leq x'$) for $x,x' \in A$, $y \in P\setminus A$}.\]
Let $A\subseteq P$ be an autonomous subset. When we say that $Q$ is \emph{obtained from $P$ by dualizing $A$} we mean the obvious thing: that $Q$ is a poset with the same set of elements as $P$ and
\begin{itemize}
\item $x \leq_Q y$ if and only if $x \leq_P y$ for $x,y \in P\setminus A$;
\item ($x \leq_Q y$ if and only if $x \leq_P y$) and ($y \leq_Q x$ if and only if $y \leq_P x$) for $x \in P\setminus A$ and $y \in A$;
\item $x \leq_Q y$ if and only if $y \leq_P x$ for $x,y \in A$.
\end{itemize} 
We can now state the condition for posets to have isomorphic comparability graphs.

\begin{lemma} \label{lem:com_graph_criterion}
Posets $P$ and $Q$ satisfy $\mathrm{com}(P)\simeq\mathrm{com}(Q)$ if and only if there is a sequence of posets $P=P_0,P_1,\ldots,P_k=Q$ such that $P_{i}$ is obtained from $P_{i-1}$ by dualizing an autonomous subset of $P_{i-1}$ for $1\leq i\leq k$.
\end{lemma}

Lemma~\ref{lem:com_graph_criterion} was implicit in the work of a number of authors, but the first published proofs appear in~\cite{dreesen1985comparability} and~\cite{kelly1986invariants}; see~\cite[Chapter 3, exercise 143(a)]{stanley2012ec1}. With Lemma~\ref{lem:com_graph_criterion} it is easy to prove Theorem~\ref{thm:com_graph_dop}. Indeed, if we are free to assume that $Q$ is obtained from~$P$ by dualizing an autonomous subset $A$ then there is a simple bijection (``reflecting the values in $A$'') between the $P$-partitions in question.

It is not the case that doppelg\"{a}ngers must have isomorphic comparability graphs. In the remainder of this section we will introduce some specific families of doppelg\"{a}nger pairs which do not (in general) have isomorphic comparability graphs. These families of doppelg\"{a}nger pairs, which arise in the combinatorics of root systems, will be the focus of the rest of our paper.

\subsection{Root posets, minuscule posets, and minuscule doppelg\"{a}nger pairs}

We use standard terminology and notation for root systems; see, e.g.,~\cite{humphreys1972lie, bourbaki2002lie, bjorner2005coxeter} for detailed presentations. Here we only very briefly review the basic facts about root systems that we will need to define the posets we care about. 

Let $V$ be an $n$-dimensional real vector space with inner product~$\langle \cdot,\cdot \rangle$. For any nonzero vector $v \in V\setminus\{0\}$ we use $s_v\colon V\to V$ to denote the \emph{orthogonal reflection} across the hyperplane perpendicular to $v$: $s_v(w) \coloneqq  w-2\frac{\langle w,v\rangle}{\langle v,v\rangle} v$. A \emph{root system} in $V$ is a finite subset of nonzero vectors~$\Phi\subseteq V\setminus\{0\}$ such that:
\begin{itemize}
\item $\mathrm{Span}_{\mathbb{R}}(\Phi) = V$;
\item $\mathrm{Span}_{\mathbb{R}}(\alpha)\cap \Phi = \{\pm\alpha\}$ for all $\alpha\in \Phi$;
\item $s_{\alpha}(\Phi) = \Phi$ for all $\alpha\in \Phi$.
\end{itemize}
The elements of $\Phi$ are called \emph{roots}. If additionally we have $2\frac{\langle\beta,\alpha\rangle}{\langle\alpha,\alpha\rangle} \in \mathbb{Z}$ for all $\alpha,\beta \in \Phi$, then~$\Phi$ is called \emph{crystallographic}. Root systems arose first in Lie theory: the crystallographic root systems correspond bijectively to semisimple Lie algebras over $\mathbb{C}$, and hence also (essentially) to semisimple Lie groups over~$\mathbb{C}$.

Let $\Phi$ be a root system in $V$. It is well-known that we can choose a collection $\alpha_1,\ldots,\alpha_n \in \Phi$ of \emph{simple roots} with the property that every root can be expressed as linear combination of simple roots with either all nonnegative, or all nonpositive, integral coefficients. Such a choice is equivalent to a choice of \emph{positive roots}~$\Phi^+\subseteq\Phi$, which are those that expand nonnegatively into simple roots.

If $\Phi$ is crystallographic, then $\{v \in V\colon 2\frac{\langle v,\alpha\rangle}{\langle\alpha,\alpha\rangle} \in\mathbb{Z} \textrm{ for all~$\alpha\in\Phi$}\}$ is its \emph{weight lattice}, the elements of which are called \emph{(integral) weights}. This lattice is generated by the \emph{fundamental weights} $\omega_1,\ldots,\omega_n$ defined by $2\frac{\langle \omega_i,\alpha_j\rangle}{\langle\alpha_j,\alpha_j\rangle}=\delta_{i,j}$. A weight is \emph{dominant} if it is a nonnegative combination of fundamental weights.

If $\Phi$ decomposes as a disjoint union of two nonempty subsets which span orthogonal subspaces, then it is called \emph{reducible}; otherwise it is called \emph{irreducible}. The irreducible root systems have been classified. This is the famous \emph{Cartan-Killing classification} into \emph{types}. The crystallographic types are $A_n$, $B_n$, $C_n$, $D_n$, $E_6$, $E_7$, $E_8$, $F_4$, and $G_2$. The non-crystallographic types are $H_3$, $H_4$, and $I_2(m)$ (for $m\notin\{2,3,4,6\}$).

The subgroup of $GL(V)$ generated by $s_{\alpha}$ for $\alpha\in \Phi$ is (essentially by definition) a finite real reflection group. (If $\Phi$ is crystallographic, then this group is a \emph{Weyl group}.) Among all the complex reflection groups, there is a special collection called the \emph{coincidental types}; see~\cite[Theorem 14]{miller2015foulkes}~\cite{miller2018walls} for various equivalent definitions of the coincidental types. We say that $\Phi$ is of \emph{coincidental type} if its corresponding reflection group is. The root systems of coincidental type are precisely $A_n$, $B_n$, $C_n$, $G_2$, $H_3$, and $I_2(m)$. (Since $G_2$ is a special case of $I_2(m)$, we will omit it from future lists of the coincidental types.) The coincidental types tend to enjoy better enumerative properties than the non-coincidental types (see, e.g., the list at the beginning of~\cite[\S8]{hamaker2018doppelgangers}).

Now on to the posets. We first define root posets, and we start with the crystallographic case. So let~$\Phi$ be an irreducible, crystallographic root system. There is a natural partial order on the set $\Phi^+$ of positive roots whereby for two roots~$\alpha,\beta \in \Phi^+$ we have  $\alpha \leq \beta$ if and only if $\beta-\alpha=\sum_{i=1}^{n}a_i \alpha_i$ with all $a_i\geq 0$. The resulting poset is called the \emph{root poset} of $\Phi$. We denote this poset by $\Phi^+(X)$ where $X$ is the type of the root system $\Phi$. For instance, the posets $\Phi^+(A_n)$, $\Phi^+(B_n)\simeq\Phi^+(C_n)$, and $\Phi^+(D_4)$ are depicted in Figure~\ref{fig:crystal_root_posets}.

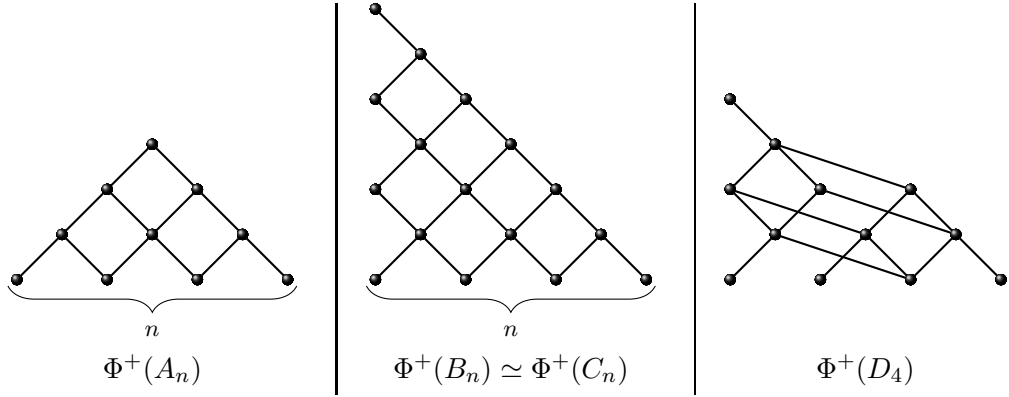
\begin{figure}
\begin{center}
\begin{tikzpicture}[scale=1.2]
	\SetFancyGraph
	\Vertex[NoLabel,x=0,y=0]{1}
	\Vertex[NoLabel,x=1,y=0]{2}
	\Vertex[NoLabel,x=2,y=0]{3}
	\Vertex[NoLabel,x=3,y=0]{4}
	\Vertex[NoLabel,x=0.5,y=0.5]{5}
	\Vertex[NoLabel,x=1.5,y=0.5]{6}
	\Vertex[NoLabel,x=2.5,y=0.5]{7}
	\Vertex[NoLabel,x=1,y=1]{8}
	\Vertex[NoLabel,x=2,y=1]{9}
	\Vertex[NoLabel,x=1.5,y=1.5]{10}
	\Edges[style={thick}](1,5)
	\Edges[style={thick}](2,5)
	\Edges[style={thick}](2,6)
	\Edges[style={thick}](3,6)
	\Edges[style={thick}](3,7)
	\Edges[style={thick}](4,7)
	\Edges[style={thick}](5,8)
	\Edges[style={thick}](6,8)
	\Edges[style={thick}](6,9)
	\Edges[style={thick}](7,9)
	\Edges[style={thick}](8,10)
	\Edges[style={thick}](9,10)
	\draw [decorate,decoration={brace,amplitude=10pt,mirror},yshift=-2pt] (-0.1,0) -- (3.1,0) node [black,midway,yshift=-0.6cm]  {\footnotesize $n$};
	\node at (1.5,-1) {$\Phi^+(A_n)$};
\end{tikzpicture} \quad \vline \quad \begin{tikzpicture}[scale=1.2]
	\SetFancyGraph
	\Vertex[NoLabel,x=0,y=0]{1}
	\Vertex[NoLabel,x=1,y=0]{2}
	\Vertex[NoLabel,x=2,y=0]{3}
	\Vertex[NoLabel,x=3,y=0]{4}
	\Vertex[NoLabel,x=0.5,y=0.5]{5}
	\Vertex[NoLabel,x=1.5,y=0.5]{6}
	\Vertex[NoLabel,x=2.5,y=0.5]{7}
	\Vertex[NoLabel,x=1,y=1]{8}
	\Vertex[NoLabel,x=2,y=1]{9}
	\Vertex[NoLabel,x=1.5,y=1.5]{10}
	\Vertex[NoLabel,x=0,y=1]{11}
	\Vertex[NoLabel,x=0.5,y=1.5]{12}
	\Vertex[NoLabel,x=0,y=2]{13}
	\Vertex[NoLabel,x=1,y=2]{14}
	\Vertex[NoLabel,x=0.5,y=2.5]{15}
	\Vertex[NoLabel,x=0,y=3]{16}
	\Edges[style={thick}](1,5)
	\Edges[style={thick}](2,5)
	\Edges[style={thick}](2,6)
	\Edges[style={thick}](3,6)
	\Edges[style={thick}](3,7)
	\Edges[style={thick}](4,7)
	\Edges[style={thick}](5,8)
	\Edges[style={thick}](6,8)
	\Edges[style={thick}](6,9)
	\Edges[style={thick}](7,9)
	\Edges[style={thick}](8,10)
	\Edges[style={thick}](9,10)
	\Edges[style={thick}](5,11)
	\Edges[style={thick}](11,12)
	\Edges[style={thick}](8,12)
	\Edges[style={thick}](12,13)
	\Edges[style={thick}](12,14)
	\Edges[style={thick}](10,14)
	\Edges[style={thick}](13,15)
	\Edges[style={thick}](14,15)
	\Edges[style={thick}](15,16)
	\draw [decorate,decoration={brace,amplitude=10pt,mirror},yshift=-2pt] (-0.1,0) -- (3.1,0) node [black,midway,yshift=-0.6cm]  {\footnotesize $n$};
	\node at (1.5,-1) {$\Phi^+(B_n)\simeq \Phi^+(C_n)$};
\end{tikzpicture} \quad \vline \quad \begin{tikzpicture}[scale=1.2]
	\SetFancyGraph
	\Vertex[NoLabel,x=0,y=0]{1}
	\Vertex[NoLabel,x=1,y=0]{2}
	\Vertex[NoLabel,x=2,y=0]{3}
	\Vertex[NoLabel,x=3,y=0]{4}
	\Vertex[NoLabel,x=0.5,y=0.5]{5}
	\Vertex[NoLabel,x=1.5,y=0.5]{6}
	\Vertex[NoLabel,x=2.5,y=0.5]{7}
	\Vertex[NoLabel,x=1,y=1]{8}
	\Vertex[NoLabel,x=2,y=1]{9}
	\Vertex[NoLabel,x=0,y=1]{11}
	\Vertex[NoLabel,x=0.5,y=1.5]{12}
	\Vertex[NoLabel,x=0,y=2]{13}
	\Edges[style={thick}](1,5)
	\Edges[style={thick}](2,6)
	\Edges[style={thick}](3,6)
	\Edges[style={thick}](3,7)
	\Edges[style={thick}](4,7)
	\Edges[style={thick}](5,8)
	\Edges[style={thick}](6,9)
	\Edges[style={thick}](7,9)
	\Edges[style={thick}](5,11)
	\Edges[style={thick}](11,12)
	\Edges[style={thick}](8,12)
	\Edges[style={thick}](12,13)
	\Edges[style={thick}](3,5)
	\Edges[style={thick}](6,11)
	\Edges[style={thick}](7,8)
	\Edges[style={thick}](9,12)
	\node at (1.5,-1) {$\Phi^+(D_4)$};
\end{tikzpicture}
\end{center}
\caption{Some of the root posets for crystallographic root systems.} \label{fig:crystal_root_posets}
\end{figure}

The root poset $\Phi^+$ of $\Phi$ is graded, and its rank function is given by $r(\alpha)=\mathrm{ht}(\alpha)-1$, where the \emph{height} of a positive root $\alpha =\sum_{i=1}^{n} a_i \alpha_i \in \Phi^{+}$ is $\mathrm{ht}(\alpha) := \sum_{i=1}^{n}a_i$. The minimal elements of the root poset are precisely the simple roots, and there is a unique maximal element called the \emph{highest root}. Beyond enjoying these nice combinatorial properties, the root poset also contains important numerical information about the root system. For example, recall that the \emph{degrees} $d_1 \leq d_2 \leq \ldots \leq d_n$ of the Weyl group of~$\Phi$ are certain very significant parameters in the theory of finite reflection groups (for instance because of the way they enter into the Chevalley-Shephard-Todd theorem). These can be read off from the rank function of the root poset as follows:
\[\#\{\alpha \in \Phi^+\colon r(\alpha)=i\} = \#\{1\leq j \leq n\colon d_j > i+1\} \textrm{ for all $i=0,\ldots,r(\Phi^+)+1$}.\]
This was first proved uniformly by Kostant~\cite{kostant1959principal}; see also~\cite[Theorem 3.20]{humphreys1990reflection}. In particular, the \emph{Coxeter number} $h$ of the Weyl group of $\Phi$ is $h=d_n = r(\Phi^+)+2$.

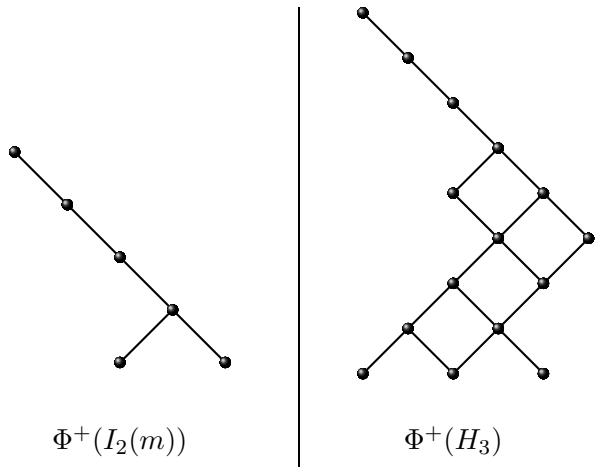
\begin{figure}
\begin{center}
 \begin{tikzpicture}[scale=1.4]
	\SetFancyGraph
	\Vertex[NoLabel,x=0,y=0]{1}
	\Vertex[NoLabel,x=1,y=0]{2}
	\Vertex[NoLabel,x=0.5,y=0.5]{3}
	\Vertex[NoLabel,x=0,y=1]{4}
	\Vertex[NoLabel,x=-0.5,y=1.5]{5}
	\Vertex[NoLabel,x=-1,y=2]{6}
	\Edges[style={thick}](1,3)
	\Edges[style={thick}](2,3)
	\Edges[style={thick}](3,4)
	\Edges[style={thick}](4,5)
	\Edges[style={thick}](5,6)
	\rotatebox{-45}{\draw [decorate,decoration={brace,amplitude=10pt},yshift=7pt] (-2.2,0.5) -- (0.8,0.5) node [black,midway,yshift=0.6cm]  {\footnotesize $(m-1)$};}
	\node at (0,-0.75) {$\Phi^+(I_2(m))$};
\end{tikzpicture}  \qquad  \vline \qquad \begin{tikzpicture}[scale=1.2]
	\SetFancyGraph
	\Vertex[NoLabel,x=0,y=0]{1}
	\Vertex[NoLabel,x=1,y=0]{2}
	\Vertex[NoLabel,x=2,y=0]{3}
	\Vertex[NoLabel,x=0.5,y=0.5]{4}
	\Vertex[NoLabel,x=1.5,y=0.5]{5}
	\Vertex[NoLabel,x=1,y=1]{6}
	\Vertex[NoLabel,x=2,y=1]{7}
	\Vertex[NoLabel,x=1.5,y=1.5]{8}
	\Vertex[NoLabel,x=2.5,y=1.5]{9}
	\Vertex[NoLabel,x=1,y=2]{10}
	\Vertex[NoLabel,x=2,y=2]{11}
	\Vertex[NoLabel,x=1.5,y=2.5]{12}
	\Vertex[NoLabel,x=1,y=3]{13}
	\Vertex[NoLabel,x=0.5,y=3.5]{14}
	\Vertex[NoLabel,x=0,y=4]{15}
	\Edges[style={thick}](1,4)
	\Edges[style={thick}](2,4)
	\Edges[style={thick}](2,5)
	\Edges[style={thick}](3,5)
	\Edges[style={thick}](4,6)
	\Edges[style={thick}](5,6)
	\Edges[style={thick}](5,7)
	\Edges[style={thick}](6,8)
	\Edges[style={thick}](7,8)
	\Edges[style={thick}](7,9)
	\Edges[style={thick}](8,10)
	\Edges[style={thick}](8,11)
	\Edges[style={thick}](9,11)
	\Edges[style={thick}](10,12)
	\Edges[style={thick}](11,12)
	\Edges[style={thick}](12,13)
	\Edges[style={thick}](13,14)
	\Edges[style={thick}](14,15)
	\node at (1,-0.75) {$\Phi^+(H_3)$};
\end{tikzpicture}
\end{center}
\caption{The root posets for non-crystallographic root systems of coincidental type.} \label{fig:non_crystal_root_posets}
\end{figure}

The same definition of partial order could be applied to the positive roots of a non-crystallographic root system, but this naive generalization would not have the desirable properties that the crystallographic case has. For the non-crystallographic root systems of coincidental type, $H_3$ and $I_2(m)$, there are ad hoc constructions of root posets (due to Armstrong~\cite{armstrong2009generalized}) which have all the desirable numerological and combinatorial properties. These are depicted in Figure~\ref{fig:non_crystal_root_posets}. Note that this definition of the root poset for $I_2(m)$ is consistent with the special cases $I_2(2)=A_1 \oplus A_1$, $I_2(3)=A_2$, $I_2(4)=B_2$, and $I_2(6)=G_2$. For $H_4$, the unique non-crystallographic root system which is not of coincidental type, it turns out that there is no way to define a root poset that has all the same desirable properties; see~\cite{cuntz2015root}.

We will particularly focus in this paper on the root posets of coincidental type, namely $\Phi^+(A_n)$, $\Phi^+(B_n)\simeq\Phi^+(C_n)$, $\Phi^+(H_3)$ and  $\Phi^+(I_2(m))$. One feature of the root posets of coincidental type that will turn out to be very important for us is that each element covers, and is covered by, at most two elements. In fact, all the posets we will be interested in have such a $2$-dimensional grid-like structure. The reader can check that the root poset $\Phi^+(D_4)$, depicted in Figure~\ref{fig:crystal_root_posets}, lacks this structure.

Next we define the minuscule posets. We give only a very cursory account here; for details, with a particular focus on the representation theoretic significance of minuscule posets, consult the recent book~\cite{green2013minuscule}. Let $\Phi$ be an irreducible, crystallographic root system. A nonzero, dominant weight $\omega$ is called \emph{minuscule} if $2\frac{\langle \omega,\alpha\rangle }{\langle \alpha,\alpha\rangle}\in \{-1,0,1\}$ for all roots~$\alpha \in \Phi$. Let $\omega$ be a minuscule weight. Such a minuscule weight must be fundamental: we have $\omega = \omega_i$ for some $i=1,\ldots,n$. If $G$ is a simple Lie group over $\mathbb{C}$ whose root system is $\Phi$, and $P_i\subseteq G$ is the \emph{maximal parabolic subgroup} corresponding to the simple root~$\alpha_i$, then we call the homogeneous space $G/P_i$ a \emph{minuscule variety} corresponding to the minuscule weight~$\omega_i$. The \emph{minuscule poset} for this minuscule weight, which we denote by~$\Lambda_{G/P_i}$ where $G/P_i$ is the corresponding minuscule variety, is the order filter generated by the simple root $\alpha_i$ in the root poset $\Phi^+$:
\[ \Lambda_{G/P_i} \coloneqq  \{\alpha\in \Phi^+\colon \alpha \geq \alpha_i\}.\]
Since root posets are graded, minuscule posets are graded. Furthermore, it is clear that each minuscule poset has a unique minimal element and a unique maximal element.

The minuscule posets have been classified. All the minuscule posets (up to poset isomorphism) are depicted in Figure~\ref{fig:minuscule_posets}. These are: the \emph{rectangle}  $\Lambda_{\mathrm{Gr}(k,n)}$  (a.k.a., product of chains $[k]\times [n-k]$) associated to the Grassmannian of $k$-planes in $\mathbb{C}^n$; the \emph{shifted staircase} $\Lambda_{\mathrm{OG}(n,2n)}$ associated to the even orthogonal Grassmannian; the \emph{``propeller poset''} $\Lambda_{\mathbb{Q}^{2n}}$ associated to the even dimensional quadric; and the exceptional posets $\Lambda_{\mathbb{OP}^2}$ associated to the Cayley plane (coming from $E_6$) and $\Lambda_{G_{\omega}(\mathbb{O}^3,\mathbb{O}^6)}$ associated to the Freudenthal variety (coming from $E_7$).

\begin{figure}
\begin{center}
 \begin{tikzpicture}[scale=0.75]
	\SetFancyGraph
	\Vertex[NoLabel,x=0,y=0]{1}
	\Vertex[NoLabel,x=1,y=1]{2}
	\Vertex[NoLabel,x=2,y=2]{3}
	\Vertex[NoLabel,x=3,y=3]{4}
	\Vertex[NoLabel,x=-1,y=1]{5}
	\Vertex[NoLabel,x=0,y=2]{6}
	\Vertex[NoLabel,x=1,y=3]{7}
	\Vertex[NoLabel,x=2,y=4]{8}
	\Vertex[NoLabel,x=-2,y=2]{9}
	\Vertex[NoLabel,x=-1,y=3]{10}
	\Vertex[NoLabel,x=0,y=4]{11}
	\Vertex[NoLabel,x=1,y=5]{12}
	\Edges[style={thick}](1,2)
	\Edges[style={thick}](2,3)
	\Edges[style={thick}](3,4)
	\Edges[style={thick}](1,5)
	\Edges[style={thick}](2,6)
	\Edges[style={thick}](3,7)
	\Edges[style={thick}](4,8)
	\Edges[style={thick}](5,6)
	\Edges[style={thick}](6,7)
	\Edges[style={thick}](7,8)
	\Edges[style={thick}](5,9)
	\Edges[style={thick}](6,10)
	\Edges[style={thick}](7,11)
	\Edges[style={thick}](8,12)
	\Edges[style={thick}](9,10)
	\Edges[style={thick}](10,11)
	\Edges[style={thick}](11,12)
	\rotatebox{-45}{\draw [decorate,decoration={brace,mirror,amplitude=10pt},yshift=-7pt] (-2.8,0) -- (0,0) node [black,midway,yshift=-0.6cm]  {\footnotesize $k$};}
	\rotatebox{45}{\draw [decorate,decoration={brace,mirror,amplitude=10pt},yshift=-7pt] (-0.2,0) -- (4.3,0) node [black,midway,yshift=-0.6cm]  {\footnotesize $(n-k)$};}
	\node at (1,-1) {$\Lambda_{\mathrm{Gr}(k,n)}$};
\end{tikzpicture}  \vline \quad \begin{tikzpicture}[scale=0.75]
	\SetFancyGraph
	\Vertex[NoLabel,x=0,y=0]{1}
	\Vertex[NoLabel,x=1,y=1]{2}
	\Vertex[NoLabel,x=2,y=2]{3}
	\Vertex[NoLabel,x=3,y=3]{4}
	\Vertex[NoLabel,x=0,y=2]{6}
	\Vertex[NoLabel,x=1,y=3]{7}
	\Vertex[NoLabel,x=2,y=4]{8}
	\Vertex[NoLabel,x=0,y=4]{11}
	\Vertex[NoLabel,x=1,y=5]{12}
	\Vertex[NoLabel,x=0,y=6]{13}
	\Edges[style={thick}](1,2)
	\Edges[style={thick}](2,3)
	\Edges[style={thick}](3,4)
	\Edges[style={thick}](2,6)
	\Edges[style={thick}](3,7)
	\Edges[style={thick}](4,8)
	\Edges[style={thick}](6,7)
	\Edges[style={thick}](7,8)
	\Edges[style={thick}](7,11)
	\Edges[style={thick}](8,12)
	\Edges[style={thick}](11,12)
	\Edges[style={thick}](12,13)
	\rotatebox{45}{\draw [decorate,decoration={brace,mirror,amplitude=10pt},yshift=-7pt] (-0.2,0) -- (4.3,0) node [black,midway,yshift=-0.6cm]  {\footnotesize $(n-1)$};}
	\node at (2,-1) {$\Lambda_{\mathrm{OG}(n,2n)}$};
\end{tikzpicture}  \vline \quad  \begin{tikzpicture}[scale=0.75]
	\SetFancyGraph
	\Vertex[NoLabel,x=0,y=0]{1}
	\Vertex[NoLabel,x=-1,y=1]{2}
	\Vertex[NoLabel,x=-2,y=2]{3}
	\Vertex[NoLabel,x=-3,y=3]{4}
	\Vertex[NoLabel,x=-1,y=3]{5}
	\Vertex[NoLabel,x=-2,y=4]{6}
	\Vertex[NoLabel,x=-3,y=5]{7}
	\Vertex[NoLabel,x=-4,y=6]{8}
	\Edges[style={thick}](1,2)
	\Edges[style={thick}](2,3)
	\Edges[style={thick}](3,4)
	\Edges[style={thick}](3,5)
	\Edges[style={thick}](4,6)
	\Edges[style={thick}](5,6)
	\Edges[style={thick}](6,7)
	\Edges[style={thick}](7,8)
	\draw [decorate,decoration={brace,amplitude=10pt},yshift=7pt] (-4,6) -- (-0.9,2.9) node [black,midway,yshift=0.4cm,xshift=0.4cm]  {\rotatebox{-45}{\footnotesize $n$}};
	\rotatebox{-45}{\draw [decorate,decoration={brace,mirror,amplitude=10pt},yshift=-7pt] (-4.2,0) -- (0,0) node [black,midway,yshift=-0.6cm]  {\footnotesize $n$};}
	\node at (-2,-1) {$\Lambda_{\mathbb{Q}^{2n}}$};
\end{tikzpicture} \\ \hrulefill \\ \vspace{0.3cm}
\begin{tikzpicture}[scale=0.75]
	\SetFancyGraph
	\Vertex[NoLabel,x=3,y=-3]{-4}
	\Vertex[NoLabel,x=2,y=-2]{-3}
	\Vertex[NoLabel,x=1,y=-1]{-2}
	\Vertex[NoLabel,x=2,y=0]{-1}
	\Vertex[NoLabel,x=-1,y=1]{0}
	\Vertex[NoLabel,x=0,y=0]{1}
	\Vertex[NoLabel,x=1,y=1]{2}
	\Vertex[NoLabel,x=2,y=2]{3}
	\Vertex[NoLabel,x=3,y=3]{4}
	\Vertex[NoLabel,x=0,y=2]{6}
	\Vertex[NoLabel,x=1,y=3]{7}
	\Vertex[NoLabel,x=2,y=4]{8}
	\Vertex[NoLabel,x=0,y=4]{11}
	\Vertex[NoLabel,x=1,y=5]{12}
	\Vertex[NoLabel,x=0,y=6]{13}
	\Vertex[NoLabel,x=-1,y=7]{14}
	\Edges[style={thick}](-4,-3)
	\Edges[style={thick}](-3,-2)
	\Edges[style={thick}](-2,-1)
	\Edges[style={thick}](-2,1)
	\Edges[style={thick}](-1,2)
	\Edges[style={thick}](0,6)
	\Edges[style={thick}](0,1)
	\Edges[style={thick}](1,2)
	\Edges[style={thick}](2,3)
	\Edges[style={thick}](3,4)
	\Edges[style={thick}](2,6)
	\Edges[style={thick}](3,7)
	\Edges[style={thick}](4,8)
	\Edges[style={thick}](6,7)
	\Edges[style={thick}](7,8)
	\Edges[style={thick}](7,11)
	\Edges[style={thick}](8,12)
	\Edges[style={thick}](11,12)
	\Edges[style={thick}](12,13)
	\Edges[style={thick}](13,14)
	\node at (0,-5) {$\Lambda_{\mathbb{OP}^2}$};
\end{tikzpicture} \qquad  \vline \qquad \begin{tikzpicture}[scale=1.2]
	\SetFancyGraph
	\Vertex[NoLabel,x=0.5,y=-0.5]{-4}
	\Vertex[NoLabel,x=1.5,y=-0.5]{-5}
	\Vertex[NoLabel,x=1,y=-1]{-6}
	\Vertex[NoLabel,x=2,y=-1]{-7}
	\Vertex[NoLabel,x=1.5,y=-1.5]{-8}
	\Vertex[NoLabel,x=2.5,y=-1.5]{-9}
	\Vertex[NoLabel,x=1,y=-2]{-10}
	\Vertex[NoLabel,x=2,y=-2]{-11}
	\Vertex[NoLabel,x=1.5,y=-2.5]{-12}
	\Vertex[NoLabel,x=1,y=-3]{-13}
	\Vertex[NoLabel,x=0.5,y=-3.5]{-14}
	\Vertex[NoLabel,x=0,y=-4]{-15}
	\Vertex[NoLabel,x=0,y=0]{1}
	\Vertex[NoLabel,x=1,y=0]{2}
	\Vertex[NoLabel,x=2,y=0]{3}
	\Vertex[NoLabel,x=0.5,y=0.5]{4}
	\Vertex[NoLabel,x=1.5,y=0.5]{5}
	\Vertex[NoLabel,x=1,y=1]{6}
	\Vertex[NoLabel,x=2,y=1]{7}
	\Vertex[NoLabel,x=1.5,y=1.5]{8}
	\Vertex[NoLabel,x=2.5,y=1.5]{9}
	\Vertex[NoLabel,x=1,y=2]{10}
	\Vertex[NoLabel,x=2,y=2]{11}
	\Vertex[NoLabel,x=1.5,y=2.5]{12}
	\Vertex[NoLabel,x=1,y=3]{13}
	\Vertex[NoLabel,x=0.5,y=3.5]{14}
	\Vertex[NoLabel,x=0,y=4]{15}
	\Edges[style={thick}](1,4)
	\Edges[style={thick}](2,4)
	\Edges[style={thick}](2,5)
	\Edges[style={thick}](3,5)
	\Edges[style={thick}](4,6)
	\Edges[style={thick}](5,6)
	\Edges[style={thick}](5,7)
	\Edges[style={thick}](6,8)
	\Edges[style={thick}](7,8)
	\Edges[style={thick}](7,9)
	\Edges[style={thick}](8,10)
	\Edges[style={thick}](8,11)
	\Edges[style={thick}](9,11)
	\Edges[style={thick}](10,12)
	\Edges[style={thick}](11,12)
	\Edges[style={thick}](12,13)
	\Edges[style={thick}](13,14)
	\Edges[style={thick}](14,15)
	\Edges[style={thick}](1,-4)
	\Edges[style={thick}](2,-4)
	\Edges[style={thick}](2,-5)
	\Edges[style={thick}](3,-5)
	\Edges[style={thick}](-4,-6)
	\Edges[style={thick}](-5,-6)
	\Edges[style={thick}](-5,-7)
	\Edges[style={thick}](-6,-8)
	\Edges[style={thick}](-7,-8)
	\Edges[style={thick}](-7,-9)
	\Edges[style={thick}](-8,-10)
	\Edges[style={thick}](-8,-11)
	\Edges[style={thick}](-9,-11)
	\Edges[style={thick}](-10,-12)
	\Edges[style={thick}](-11,-12)
	\Edges[style={thick}](-12,-13)
	\Edges[style={thick}](-13,-14)
	\Edges[style={thick}](-14,-15)
	\node at (1.5,-4.75) {$\Lambda_{G_{\omega}(\mathbb{O}^3,\mathbb{O}^6)}$};
\end{tikzpicture}
\end{center}
\caption{The minuscule posets.} \label{fig:minuscule_posets}
\end{figure}

Technically what we defined in the previous paragraph are the \emph{connected} minuscule posets, and arbitrary minuscule posets are disjoint unions of these connected minuscule posets. But because we have $J(P + Q) = J(P)\times J(Q)$, and everything about distributive lattices that we are interested in studying (e.g., the CDE property, the behavior of rowmotion, et cetera) decomposes in the obvious way with respect to Cartesian products, we will restrict our attention in this paper to connected posets.

The minuscule posets enjoy a number of remarkable algebraic and combinatorial properties. For instance, they are conjectured to be the only examples of \emph{Gaussian posets} (meaning that their $P$-partition size generating function satisfies an algebraic equation of a specific form reminiscent of the $q$-binomials); see the seminal paper~\cite{proctor1984bruhat}. Moreover, it is often possible to establish these properties of minuscule posets in a uniform manner. By \emph{uniform} we mean in a manner which avoids the classification of minuscule posets and instead relies purely on Lie- and root system-theoretic concepts. 

Let us also mention that the $P$-partitions for minuscule posets $P$ are significant in Standard Monomial Theory~\cite{seshardri1978geometry}; and very recently these $P$-partitions have also been connected to the the theory of quiver representations~\cite{garver2018minuscule}. 

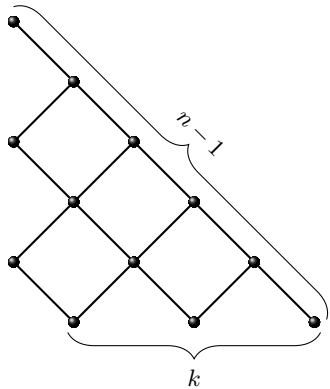
\begin{figure}
\begin{center}
\begin{tikzpicture}[scale=0.8]
	\SetFancyGraph
	\Vertex[NoLabel,x=0,y=0]{1}
	\Vertex[NoLabel,x=2,y=0]{2}
	\Vertex[NoLabel,x=4,y=0]{3}
	\Vertex[NoLabel,x=-1,y=1]{4}
	\Vertex[NoLabel,x=1,y=1]{5}
	\Vertex[NoLabel,x=3,y=1]{6}
	\Vertex[NoLabel,x=0,y=2]{7}
	\Vertex[NoLabel,x=2,y=2]{8}
	\Vertex[NoLabel,x=-1,y=3]{9}
	\Vertex[NoLabel,x=1,y=3]{10}
	\Vertex[NoLabel,x=0,y=4]{11}
	\Vertex[NoLabel,x=-1,y=5]{12}
	\Edges[style={thick}](1,4)
	\Edges[style={thick}](1,5)
	\Edges[style={thick}](2,5)
	\Edges[style={thick}](2,6)
	\Edges[style={thick}](3,6)
	\Edges[style={thick}](4,7)
	\Edges[style={thick}](5,7)
	\Edges[style={thick}](5,8)
	\Edges[style={thick}](6,8)
	\Edges[style={thick}](7,9)
	\Edges[style={thick}](7,10)
	\Edges[style={thick}](8,10)
	\Edges[style={thick}](9,11)
	\Edges[style={thick}](10,11)
	\Edges[style={thick}](11,12)
	\draw [decorate,decoration={brace,amplitude=10pt},yshift=7pt] (-1,5) -- (4.2,-0.2) node [black,midway,yshift=0.4cm,xshift=0.4cm]  {\rotatebox{-45}{\footnotesize $n-1$}};
	\draw [decorate,decoration={brace,amplitude=10pt,mirror},yshift=-2pt] (-0.1,-0.1) -- (4.1,-0.1) node [black,midway,yshift=-0.6cm]  {\footnotesize $k$};
\end{tikzpicture}
\end{center}
\caption{The trapezoid $T_{k,n}$.} \label{fig:trap}
\end{figure}

We are almost ready to define the minuscule doppelg\"{a}nger pairs. We just need to formally introduce our main antagonist in this paper: the trapezoid poset $T_{k,n}$. It is depicted in Figure~\ref{fig:trap}. Note that $T_{k,n}$ is defined for $1\leq k \leq n/2$. And observe that $T_{k,n}$ is an induced subposet (in fact, an order filter) of $\Phi^+(B_{n-k})$, and that $T_{k,2k} \simeq \Phi^+(B_k)$, but that in general $T_{k,n}$ is not a root poset (or a minuscule poset). 

\begin{thm} \label{thm:minuscule_doppelgangers}
Let $(P,Q) \in \{(\Lambda_{\mathrm{Gr}(k,n)},T_{k,n}),(\Lambda_{\mathrm{OG}(6,12)},\Phi^+(H_3)),(\Lambda_{\mathbb{Q}^{2n}},\Phi^+(I_2(2n)))\}$. Then $P$ and $Q$ are doppelg\"{a}ngers.
\end{thm}

We call the pair $(P,Q)$ appearing in Theorem~\ref{thm:minuscule_doppelgangers} a \emph{minuscule doppelg\"{a}nger pair}. 

As for the history of Theorem~\ref{thm:minuscule_doppelgangers}: that $\Lambda_{\mathrm{Gr}(k,n)}$ and $T_{k,n}$ are doppelg\"{a}ngers was first proved by Proctor~\cite{proctor1983shifted}; that the other two pairs are doppelg\"{a}ngers is a simple exercise (see~\cite{williams2013cataland}). More recently, Hamaker, Patrias, Pechenik, and Williams~\cite{hamaker2018doppelgangers} gave a uniform bijective proof of Theorem~\ref{thm:minuscule_doppelgangers} based on the $K$-theory of minuscule flag varieties. Their bijections employ the $K$-theoretic jeu-de-taquin developed by Thomas and Yong~\cite{thomas2009combinatorial, thomas2009jeu}.

It is easy to verify that $\Lambda_{\mathbb{Q}^{2n}}$ and $\Phi^+(I_2(2n))$ have isomorphic comparability graphs, that $\Lambda_{\mathrm{Gr}(k,n)}$ and $T_{k,n}$ have isomorphic comparability graphs if and only if $k\leq 2$, and that $\Lambda_{\mathrm{OG}(6,12)}$ and $\Phi^+(H_3)$ do not have isomorphic comparability graphs. In the remainder of the paper we will put forward some conjectures which suggest that even though the minuscule doppelg\"{a}nger pairs do not in general have isomorphic comparability graphs, in some respects they behave ``as if'' they do.

\begin{remark} \label{rem:a_root_poset}
The Type A root poset does not appear in Theorem~\ref{thm:minuscule_doppelgangers}. Nevertheless, there is a close connection between the $P$-partitions of the Type A root poset and of the shifted staircase, even though these posets are not literally doppelg\"{a}ngers; one might call them ``lookalikes.'' See~\cite[\S11.1]{hamaker2018doppelgangers}, which cites~\cite{proctor1990new, sheats1999symplectic, purbhoo2018marvellous}, for a precise account of this connection. It would be interesting to extend (in some modified form) the conjectures about minuscule doppelg\"{a}nger pairs we present in the next two sections to include this ``lookalike'' pair as well.
\end{remark}

\begin{remark}
A distinctive feature of all of the posets we are most interested in here (i.e., minuscule posets, root posets of coincidental type, and the trapezoid) is that they have product formulas for their order polynomials (see Conjectures~\ref{conj:minuscule_row_cyc_siev} and~\ref{conj:root_poset_cyc_siev} for $q$-analogs of these product formulas).
\end{remark}

\section{The CDE property} \label{sec:cde}

In this section we study the minuscule  doppelg\"{a}nger pairs from the perspective of the coincidental down-degree expectations (CDE) property. 

\subsection{The CDE property for distributive lattices} \label{sec:cde_dist}

Here we quickly review the coincidental down-degree expectations (CDE) property for distributive lattices. See~\cite{reiner2018poset} and~\cite{hopkins2017cde} for more details, including motivation for the study of this property.

Let $P$ be a poset. We use $\mathrm{uni}_P$ to denote the uniform distribution on $P$, and we use $\mathrm{maxchain}_P$ to denote distribution where each $p\in P$ occurs proportional to the number of maximal chains containing $p$. Recall that the \emph{down-degree} of $p\in P$, denoted $\mathrm{ddeg}(p)$, is the number of elements of $P$ that $p$ covers. For a statistic $f$ on a combinatorial set $X$ and a discrete probability distribution $\mu$ on $X$, we use $\mathbb{E}(\mu;f)$ to denote the expected value of $f$ with respect to $\mu$.

\begin{definition}
$P$ has the \emph{coincidental down-degree expectations (CDE)} property if $\mathbb{E}(\mathrm{maxchain}_P;\mathrm{ddeg})=\mathbb{E}(\mathrm{uni}_P;\mathrm{ddeg})$. We also say $P$ is \emph{CDE} in this case.
\end{definition}

Note that $\mathbb{E}(\mathrm{uni}_P;\mathrm{ddeg})$ is just the number of edges of the Hasse diagram of $P$ divided by the number of elements of $P$; we call this quantity the \emph{edge density} of the poset.

For graded posets, there are also multichain distributions which interpolate between the uniform and maxchain distributions. Namely, let $\mathrm{mchain}_{P}(\ell)$ denote the probability distribution on $P$ defined as follows: choose a multichain $p_0 \leq p_1 \leq \cdots \leq p_{\ell-1}$ of $P$ of length $(\ell-1)$ uniformly at random; choose $i \in \{0,1,\ldots,\ell-1\}$ uniformly at random; and then select element $p_i \in P$. Clearly, $\mathrm{mchain}_{P}(1) = \mathrm{uni}_{P}$, and if $P$ is graded then $\mathrm{lim}_{\ell\to\infty} \mathrm{mchain}_{P}(\ell) = \mathrm{maxchain}_P$.

\begin{definition}
$P$ is \emph{multichain CDE (mCDE)} if $\mathbb{E}(\mathrm{mchain}_P(\ell);\mathrm{ddeg})=\mathbb{E}(\mathrm{uni}_P;\mathrm{ddeg})$ for all $\ell\geq 1$.
\end{definition}

As we just explained, if $P$ is graded then $P$ being mCDE implies it is CDE.

Now we focus on the special case of a distributive lattice; that is, we will consider the CDE property not for $P$ but for $J(P)$. 

Let's make a few immediate comments about the CDE property for $J(P)$. First of all, for $I \in J(P)$ we have that $\mathrm{ddeg}(I)=\#\mathrm{max}(I)$, the number of maximal elements in~$I$. In this context we also refer to $\mathrm{ddeg}$ as the \emph{antichain cardinality statistic} on $J(P)$ because $I \mapsto \mathrm{max}(I)$ is a canonical bijection between $J(P)$ and the set of antichains of~$P$. Hence the edge density of $J(P)$ is just the average size of an antichain of $P$. Next, observe that $\mathrm{mchain}_{J(P)}(\ell)$ can be described as follows: choose a $P$-partition $T\in \mathrm{PP}^{\ell}(P)$ uniformly at random; choose $i \in \{0,1,\ldots,\ell-1\}$ uniformly at random; and then select the element $T^{-1}(\{0,1,\ldots,i\})$. (The distribution $\mathrm{maxchain}_{J(P)}$ can similarly be described in terms of linear extensions of $P$.) Finally, note that since $J(P)$ is necessarily graded, $J(P)$ being mCDE implies it is CDE.

To proceed further we need to introduce toggling and the ``toggle perspective.'' This perspective was first developed in~\cite{chan2017expected}, while the ``toggle'' terminology comes from Striker and Williams~\cite{striker2012promotion}. For $p \in P$ we define the \emph{toggle at $p$} $\tau_p\colon J(P)\to J(P)$ as
\[\tau_p(I) \coloneqq  \begin{cases} I \cup \{p\} &\textrm{if $p\notin I$ and $I\cup\{p\}$ is an order ideal of $P$}; \\ I\setminus \{p\} &\textrm{if $p\in I$ and $I\setminus \{p\}$ is an order ideal of $P$}; \\ I &\textrm{otherwise}. \end{cases}\]
The $\tau_p$ are involutions. We define the \emph{toggleability statistics} $\mathcal{T}_{p^+}, \mathcal{T}_{p^-},\mathcal{T}_p\colon J(P) \to \mathbb{Z}$, which record whether we can toggle $p$ into or out of an order ideal, by 
\begin{align*}
\mathcal{T}_{p^+}(I) &\coloneqq \begin{cases} 1 & \textrm{ if $I \subsetneq \tau_p(I)$}; \\ 0 &\textrm{otherwise}; \end{cases} \\
\mathcal{T}_{p^-}(I) &\coloneqq \begin{cases} 1 & \textrm{ if $\tau_p(I) \subsetneq I$}; \\ 0 &\textrm{otherwise}; \end{cases} \\
\mathcal{T}_p(I) &\coloneqq  \mathcal{T}_{p^+}(I) - \mathcal{T}_{p^-}(I).
\end{align*}

With these toggleability statistics we can introduce the notion of ``toggle-symmetric'' distributions on $J(P)$. 

\begin{definition}
A probability distribution $\mu$ on $J(P)$ is \emph{toggle-symmetric} if for all $p \in P$ we have $\mathbb{E}(\mu;\mathcal{T}_p)=0$.
\end{definition}

In other words, a distribution is toggle-symmetric if we are equally likely to be able toggle out as toggle in every element. In~\cite{chan2017expected} (see also~\cite{reiner2018poset, hopkins2017cde}), it was shown that the distributions $\mathrm{mchain}_{J(P)}(\ell)$ are toggle-symmetric. This motivates the following definition:

\begin{definition}
$J(P)$ is \emph{toggle CDE (tCDE)} if $\mathbb{E}(\mu;\mathrm{ddeg}) = \mathbb{E}(\mathrm{uni}_P;\mathrm{ddeg})$ for any toggle-symmetric distribution~$\mu$ on $J(P)$.
\end{definition}

Since the multichain distributions $\mathrm{mchain}_{J(P)}(\ell)$ are toggle-symmetric, we have:
\[ \textrm{$J(P)$ is tCDE $\Rightarrow$ $J(P)$ is mCDE $\Rightarrow$ $J(P)$ is CDE}.\]

Let us also recall a few other basic facts about these CDE properties for distributive lattices; see~\cite[\S2]{hopkins2017cde} for proofs:

\begin{prop} \label{prop:tcde_basics}
Let $P$ and $Q$ be posets. Then:
\begin{itemize}
\item $J(P)$ is CDE and $J(Q)$ is CDE $\Rightarrow$ $J(P + Q) = J(P)\times(Q)$ is CDE;
\item $J(P)$ is mCDE and $J(Q)$ is mCDE $\Rightarrow$ $J(P + Q) = J(P)\times(Q)$ is mCDE;
\item $J(P)$ is tCDE and $J(Q)$ is tCDE $\Rightarrow$ $J(P + Q) = J(P)\times(Q)$ is tCDE.
\end{itemize}
Furthermore,
\begin{itemize}
\item $J(P)$ is CDE $\Leftrightarrow$ $J(P^*)=J(P)^*$ is CDE;
\item $J(P)$ is mCDE $\Leftrightarrow$ $J(P^*)=J(P)^*$ is mCDE;
\item $J(P)$ is tCDE $\Leftrightarrow$ $J(P^*)=J(P)^*$ is tCDE.
\end{itemize}
Finally, if $P$ is graded and $J(P)$ is tCDE, then the edge density of $J(P)$ is
\[\mathbb{E}(\mathrm{uni}_{J(P)};\mathrm{ddeg})=\frac{\#P}{r(P)+2}.\]
\end{prop}

To end this subsection, let us formally establish something which was implicitly observed in previous work: that the question of whether a distributive lattice is tCDE is purely a linear algebraic question.

\begin{prop} \label{prop:tcde_eq}
Let $P$ be a poset. Then $J(P)$ is tCDE with edge density $\delta \in \mathbb{Q}$ if and only if we have the following equality of functions $J(P) \to \mathbb{Q}$:
\[ \mathrm{ddeg} + \sum_{p\in P}c_p \mathcal{T}_p  = \delta,\]
for some coefficients $c_p \in \mathbb{Q}$ for $p\in P$.
\end{prop}
\begin{proof}
If we have an equality of functions $\mathrm{ddeg} + \sum_{p\in P}c_p \mathcal{T}_p  = \delta$, then for a toggle-symmetric distribution $\mu$ on $J(P)$ we can take expectations with respect to $\mu$ of both sides of this equality to see that $\mathbb{E}(\mathrm{ddeg};\mu) = \delta$. So one direction of this ``if and only if'' is clear.

Now let $\mathbb{R}^{J(P)}$ denote the vector space of formal $\mathbb{R}$-linear combinations of elements of $J(P)$, and extend any function $f\colon J(P)\to \mathbb{R}$ to a linear function $f\colon \mathbb{R}^{J(P)} \to \mathbb{R}$ in the obvious way. Let~$V\subseteq \mathbb{R}^{J(P)}$ be the linear subspace with $ \mathcal{T}_p =0$ for all $p \in P$. And let $H\subseteq V$ be the affine hyperplane in $V$ with sum of standard basis coordinates equal to one. View the set $T$ of toggle-symmetric distributions on $J(P)$ inside of $H$ in the obvious way: it is the polytope in $H$ defined by requiring each standard basis coordinate to be between~$0$ and $1$. 

The set $T$ contains the uniform distribution $\mathrm{uni}_{J(P)}$. Since no coordinate of~$\mathrm{uni}_{J(P)}$ is equal to either zero or one, if $v$ is any vector in $H-J(P)_{\mathrm{uni}}$ then we can find some~$\varepsilon > 0$ such that $\mathrm{uni}_{J(P)}+\varepsilon v$ remains in $T$. This implies that the affine span of~$T$ is all of $H$, and thus that the linear span of $T$ is all of $V$. Therefore, the space of linear functions vanishing on all of $T$ is the same as the space of linear functions vanishing on all of $V$, and so is spanned by the $\mathcal{T}_p$. But if $J(P)$ is tCDE with edge density $\delta$ then the linear function $\delta-\mathrm{ddeg}$ is identically zero on all of $T$, which means that we indeed get an equality of functions of the form claimed by the proposition.
\end{proof}

Note that $\mathrm{ddeg}(I) = \sum_{p\in P}\mathcal{T}_{p^-}(I)$ for all~$I \in J(P)$. This way of writing the down-degree function is very useful for studying the tCDE property of distributive lattices because with this description of $\mathrm{ddeg}$ in mind, Proposition~\ref{prop:tcde_eq} tells us that the tCDE property is equivalent to the toggleability statistics $\mathcal{T}_{p^+},\mathcal{T}_{p^-}$ for $p\in P$ satisfying a linear equation of a certain form. Finding such a linear equation is essentially the only the general tool we have for proving that a distributive lattice is CDE.

\subsection{The CDE property for minuscule posets, root posets, and minuscule doppelg\"{a}nger pairs}

Now let us discuss the CDE property for the distributive lattices associated to minuscule posets, root posets, and minuscule doppelg\"{a}nger pairs.

First minuscule posets:

\begin{thm} \label{thm:minuscule_cde}
Let $P$ be a minuscule poset. Then $J(P)$ is tCDE.
\end{thm}

Theorem~\ref{thm:minuscule_cde} was first proved in a case-by-case manner in~\cite{hopkins2017cde}, also using results from~\cite{chan2017expected}. Shortly thereafter, Rush~\cite{rush2016minuscule} gave a uniform proof of Theorem~\ref{thm:minuscule_cde}. Rush's arguments were based on his earlier work with coauthors~\cite{rush2013orbits, rush2015orbits} on understanding toggling in order ideals of minuscule posets in terms of Stembridge's theory of minuscule heaps~\cite{stembridge1996fully}. All of this is just to say that we have a very satisfactory account of why Theorem~\ref{thm:minuscule_cde} is true.

Next we consider root posets. Not all distributive lattices associated to root posets are CDE: as mentioned in~\cite[Remark 2.25]{reiner2018poset}, $J(\Phi^+(D_4))$ is not CDE. However, it turns out that the distributive lattices associated to root posets \emph{of coincidental type} are CDE. In fact, most of the work needed to show this has already been done elsewhere; we only need to handle the non-crystallographic types $H_3$ and $I_2(m)$.

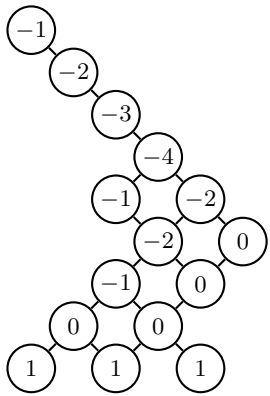
\begin{figure}
\begin{tikzpicture}[scale=0.75]
	\node[circle, draw=black, inner sep=0pt, thick, minimum size=18pt] (12) at (3,6) {\footnotesize $1$};
	\node[circle, draw=black, inner sep=0pt, thick, minimum size=18pt] (14) at (1.5,6) {\footnotesize $1$};
	\node[circle, draw=black, inner sep=0pt, thick, minimum size=18pt] (15) at (2.25,6.75) {\footnotesize $0$};
	\node[circle, draw=black, inner sep=0pt, thick, minimum size=18pt] (16) at (3,7.5) {\footnotesize $0$};
	\node[circle, draw=black, inner sep=0pt, thick, minimum size=18pt] (17) at (3.75,8.25) {\footnotesize $0$};
	\node[circle, draw=black, inner sep=0pt, thick, minimum size=18pt] (18) at (0,6) {\footnotesize $1$};
	\node[circle, draw=black, inner sep=0pt, thick, minimum size=18pt] (19) at (0.75,6.75) {\footnotesize $0$};
	\node[circle, draw=black, inner sep=0pt, thick, minimum size=18pt] (20) at (1.5,7.5) {\footnotesize $-1$};
	\node[circle, draw=black, inner sep=0pt, thick, minimum size=18pt] (21) at (2.25,8.25) {\footnotesize $-2$};
	\node[circle, draw=black, inner sep=0pt, thick, minimum size=18pt] (22) at (3.0,9) {\footnotesize $-2$};
	\node[circle, draw=black, inner sep=0pt, thick, minimum size=18pt] (23) at (1.5,9) {\footnotesize $-1$};
	\node[circle, draw=black, inner sep=0pt, thick, minimum size=18pt] (24) at (2.25,9.75) {\footnotesize $-4$};
	\node[circle, draw=black, inner sep=0pt, thick, minimum size=18pt] (25) at (1.5,10.5) {\footnotesize $-3$};
	\node[circle, draw=black, inner sep=0pt, thick, minimum size=18pt] (26) at (0.75,11.25) {\footnotesize $-2$};
	\node[circle, draw=black, inner sep=0pt, thick, minimum size=18pt] (27) at (0,12) {\footnotesize $-1$};
	\draw[thick] (12) -- (15);
	\draw[thick] (14) -- (15) -- (16) -- (17);
	\draw[thick] (14) -- (19);
	\draw[thick] (15) -- (20);
	\draw[thick] (16) -- (21);
	\draw[thick] (17) -- (22);
	\draw[thick] (18) -- (19) -- (20) -- (21) -- (22);
	\draw[thick] (21) -- (23);
	\draw[thick] (22) -- (24);
	\draw[thick] (23) -- (24);
	\draw[thick] (24) -- (25) -- (26) -- (27);
\end{tikzpicture}
\caption{The tCDE equation coefficients $c_p$ for $\Phi^{+}(H_3)$.} \label{fig:h3_coeffs}
\end{figure}

\begin{thm} \label{thm:root_poset_cde}
Let $P$ be a root poset of coincidental type. Then $J(P)$ is tCDE.
\end{thm}
\begin{proof}
We do a case-by-case check:
\begin{itemize}
\item For $P=\Phi^{+}(A_n)$, $P$ corresponds (up to duality- see Proposition~\ref{prop:tcde_basics}) to a staircase partition and so $J(P)$ is tCDE by~\cite[Theorem 3.4]{chan2017expected}.
\item For $P=\Phi^{+}(B_n)\simeq \Phi^+(C_n)$, $P$ corresponds (up to duality- see Proposition~\ref{prop:tcde_basics}) to a shifted double staircase and so $J(P)$ is tCDE by~\cite[Theorem 4.2]{hopkins2017cde}.
\item For $P=\Phi^{+}(H_3)$, if we let the coefficients $c_p$ for $p\in P$ be as in Figure~\ref{fig:h3_coeffs}, then one can check that we have the following equality of functions $J(P)\to\mathbb{Z}$:
\[2\cdot\mathrm{ddeg} +  \sum_{p\in P}c_p\mathcal{T}_p= 3,\]
and so $J(P)$ is tCDE by Proposition~\ref{prop:tcde_eq}.
\item For $P=\Phi^{+}(I_2(m))$, if we let $a$ and $b$ be the two minimal elements of $P$, then one can check that we have the following equality of functions $J(P)\to\mathbb{Z}$:
\[2\cdot\mathrm{ddeg} + \mathcal{T}_a + \mathcal{T}_b= 2,\]
and so $J(P)$ is tCDE by Proposition~\ref{prop:tcde_eq}. \qedhere
\end{itemize} 
\end{proof}

In contrast to Theorem~\ref{thm:minuscule_cde}, it is much less clear what a uniform proof of Theorem~\ref{thm:root_poset_cde} could look like. This is because, while there are various more-or-less ``uniform'' descriptions of the coincidental types (again, see~\cite[Theorem 14]{miller2015foulkes}), these criteria are quite hard to apply in practice. And moreover, recall that the root posets for the non-crystallographic root systems of coincidental types were defined in an ad hoc manner.

\begin{remark} \label{rem:pan}
Panyushev~\cite[Corollary~3.4]{panyushev2006poset} proved that for any crystallographic root system $\Phi$, the edge density of $J(\Phi^+)$ is $\frac{n}{2}$, where $n$ is the number of simple roots of $\Phi$. He was motivated to consider this edge density because of a conjectural duality property for ad-nilpotent ideals of Borel subalgebras of simple Lie algebras he proposed in~\cite{panyushev2004adnilpotent}. Note that $n/2=\#\Phi^+/h$ where $h := r(\Phi^+)+2$ is the Coxeter number of~$\Phi$ (see~\cite{kostant1959principal}). So Theorem~\ref{thm:root_poset_cde}, together with Proposition~\ref{prop:tcde_basics}, recaptures Panyushev's edge density result for the coincidental types.
\end{remark}

Finally, let's move on to the minuscule doppelg\"{a}nger pairs. Reiner-Tenner-Yong~\cite[Conjecture 2.24]{reiner2018poset} conjectured that $J(T_{k,n})$ is CDE. (Technically they conjectured this for the dual of $J(T_{k,n})$ but it amounts to the same thing by Proposition~\ref{prop:tcde_basics}.) Unfortunately, as was already observed in~\cite[Example 4.7]{hopkins2017cde}, $J(T_{2,5})$ is not tCDE. Indeed, it appears that $J(T_{k,n})$ is not tCDE except for the trivial case $k=1$ and the case $n=2k$ where $T_{k,2k}\simeq\Phi^+(B_k)$. Nevertheless, we make the following strengthening of Reiner-Tenner-Yong's conjecture:

\begin{conj} \label{conj:doppelganger_mcde}
Let $(P,Q) \in \{(\Lambda_{\mathrm{Gr}(k,n)},T_{k,n}),(\Lambda_{\mathrm{OG}(6,12)},\Phi^+(H_3)),(\Lambda_{\mathbb{Q}^{2n}},\Phi^+(I_2(2n)))\}$ be a minuscule doppelg\"{a}nger pair. Then $J(P)$ and $J(Q)$ are both mCDE with edge density
\[\frac{\#P}{r(P)+2} = \frac{\#Q}{r(Q)+2}.\]
\end{conj}

For a minuscule doppelg\"{a}nger pair $(P,Q)$ we always have $\#P=\#Q$ and $r(P)=r(Q)$ (see Proposition~\ref{prop:doppelganger_basics}), from which the equality $\frac{\#P}{r(P)+2} = \frac{\#Q}{r(Q)+2}$ in Conjecture~\ref{conj:doppelganger_mcde} follows. Also, Proposition~\ref{prop:ddeg_gf_1} below will imply that at least the edge density of the posets is as claimed. So the real content of this conjecture is the assertion that both of these posets are mCDE.

The only poset appearing in Conjecture~\ref{conj:doppelganger_mcde} whose distributive lattice of order ideals is not tCDE by one of Theorems~\ref{thm:minuscule_cde} or~\ref{thm:root_poset_cde} is $T_{k,n}$ for $k\neq 1, n/2$. For $k=2$, it was shown in~\cite[Proposition~4.8]{hopkins2017cde} that $J(T_{k,n})$ is mCDE. Actually, in just a moment we will explain another proof (in fact, two other proofs) that $J(T_{k,n})$ is mCDE when $k=2$, based on the fact that $\mathrm{com}(T_{k,n})\simeq \mathrm{com}(\Lambda_{\mathrm{Gr}(k,n)})$ in this case. What remains to prove is that $J(T_{k,n})$ is mCDE when $k\neq 1,2,n/2$. We have verified by computer that this holds for the first two cases not covered by the aforementioned considerations: $T_{3,7}$ and~$T_{3,8}$.

In the remainder of this section we will propose a couple of different strengthenings of Conjecture~\ref{conj:doppelganger_mcde} which all would hold under the assumption that $\mathrm{com}(P)\simeq\mathrm{com}(Q)$. In this way, the strengthened conjectures suggest that minuscule doppelg\"{a}nger pairs behave ``as if'' they had isomorphic comparability graphs.

\begin{remark} \label{rem:other_tcde}
Minuscule posets and root posets of coincidental type are not the only posets $P$ for which $J(P)$ is tCDE; for instance, this also happens when:
\begin{itemize}
\item $P$ is the poset corresponding to a ``balanced skew shape'' (see~\cite{chan2017expected});
\item $P$ is the poset corresponding to a ``shifted-balanced shape'' (see~\cite{hopkins2017cde}).
\end{itemize}
It is worthwhile to note, however, that all of these posets do have a $2$-dimensional grid-like structure. Indeed, it seems that an equality as in Proposition~\ref{prop:tcde_eq} can only happen as a result of the identity $\mathrm{min}(a,b) = a+b-\mathrm{max}(a,b)$ (see Lemma~\ref{lem:bi_toggle_eqs}).
\end{remark}

\subsection{Down-degree generating functions} \label{sec:ddeg_gf}

Let $P$ be a poset. For $T \in \mathrm{PP}^{\ell}(P)$, let us define the \emph{down-degree} of the $P$-partition $T$ to be
\[\mathrm{ddeg}(T) \coloneqq  \sum_{i=0}^{\ell-1} \mathrm{ddeg}(T^{-1}\{0,1,\ldots,i\}).\]
With this notation, it is clear that
\[\mathbb{E}(\mathrm{mchain}_{J(P)}(\ell);\mathrm{ddeg}) = \frac{\sum_{T \in \mathrm{PP}^{\ell}(P)}\mathrm{ddeg}(T)}{\ell\cdot\#\mathrm{PP}^{\ell}(P)}.\]
Since for a minuscule doppelg\"{a}nger pair $(P,Q)$ we have $\#\mathrm{PP}^{\ell}(P)=\#\mathrm{PP}^{\ell}(Q)$, Conjecture~\ref{conj:doppelganger_mcde} is therefore equivalent to
\[\sum_{T \in \mathrm{PP}^{\ell}(P)}\mathrm{ddeg}(T)=\sum_{T \in \mathrm{PP}^{\ell}(Q)}\mathrm{ddeg}(T),\]
for all $\ell \geq 1$. The first strengthening of Conjecture~\ref{conj:doppelganger_mcde} we offer asserts that for a minuscule doppelg\"{a}nger pair $(P,Q)$  we don't just have an equality of the sum of the $\mathrm{ddeg}$ statistic among $P$- and $Q$-partitions of height $\ell$; we have an equality of the \emph{generating functions} for this statistic.

\begin{conj} \label{conj:ddeg_gf}
Let $(P,Q) \in \{(\Lambda_{\mathrm{Gr}(k,n)},T_{k,n}),(\Lambda_{\mathrm{OG}(6,12)},\Phi^+(H_3)),(\Lambda_{\mathbb{Q}^{2n}},\Phi^+(I_2(2n)))\}$ be a minuscule doppelg\"{a}nger pair. Then for any $\ell \geq 1$ we have
\[ \sum_{T \in \mathrm{PP}^{\ell}(P)}q^{\mathrm{ddeg}(T)}=\sum_{T \in \mathrm{PP}^{\ell}(Q)}q^{\mathrm{ddeg}(T)}. \]
\end{conj}

\begin{remark} \label{rem:ppart_wt_gf}
As far as we know, there is no prior work on the generating function for $P$-partitions by down-degree. On the other hand, it is very common to study the generating function for these partitions by size. For $T \in \mathrm{PP}^{\ell}(P)$, define the \emph{size} of~$T$, denoted $|T|$, to be $|T| \coloneqq \sum_{p \in P} T(p)$. (To compare size to down-degree, it is useful to note that $|T| = \ell \cdot \#P - \sum_{i=0}^{\ell-1}\#T^{-1}(\{0,1,\ldots,i\})$.) Define the size generating function $F(P,\ell) \coloneqq  \sum_{T\in \mathrm{PP}^{\ell}(P)} q^{|T|}$. It is a celebrated theorem of MacMahon~\cite{macmahon2004combinatory} that
\[F([a]\times[b],\ell) = \prod_{i=1}^{a}\prod_{j=1}^{b}\frac{(1-q^{\ell+i+j-1})}{(1-q^{i+j-1})}.\]
In fact, for any minuscule poset $P$, the size generating function of $P$-partitions of height~$\ell$ satisfies a similar product formula:
\[F(P,\ell) = \prod_{p\in P}\frac{(1-q^{r(p)+\ell+1})}{(1-q^{r(p)+1})}.\]
That $F(P,\ell)$ has this form is what it means to say minuscule posets are Gaussian. This product formula is a $q$-version of the Weyl dimension formula for the irreducible $\mathfrak{g}$-representation $V(\ell \omega)$ with highest weight $\ell \omega$, where $\mathfrak{g}$ is the Lie algebra and $\omega$ the minuscule weight corresponding to our minuscule poset~$P$. This formula for $F(P,\ell)$ was first established by Proctor~\cite{proctor1984bruhat} using results from Standard Monomial Theory~\cite{seshardri1978geometry}; see also the presentation of Stembridge~\cite{stembridge1994minuscule}. 

But note that $F(P,\ell)$ does not play nicely with the doppelg\"{a}nger philosophy: already for $(P,Q)=(\Lambda_{\mathrm{Gr}(2,4)},T_{2,4})$ and $\ell=1$ we have $F(P,\ell)\neq F(Q,\ell)$. And also note that for a minuscule poset $P$, the down-degree generating function for $P$-partitions does not seem to satisfy any product formula. For instance, we have
\[\sum_{T \in \mathrm{PP}^{4}(\Lambda_{\mathrm{Gr}(2,4)})}q^{\mathrm{ddeg}(T)} = q^8 + 4q^7 + 10q^6 + 20q^5 + 35q^4 + 20q^3 + 10q^2 + 4q + 1,\]
which is an irreducible polynomial in $q$.
\end{remark}

\begin{remark}
For $T \in \mathrm{PP}^{\ell}(P)$, define $\mathrm{acard}(T) := \sum_{i=0}^{\ell-1} \#\mathrm{max}(I_{i}\setminus I_{i-1})$ where $I_i := T^{-1}(\{0,1,\ldots,i\})$ for $i=0,\ldots,\ell-1$, and $I_{-1} := \varnothing$. The quantity $\mathrm{acard}(T)$ is the size of the antichain of $P\times [\ell]$ naturally corresponding to the $P$-partition $T$. Observe that if $\ell=1$ then $\mathrm{acard}(T)=\mathrm{ddeg}(T)$, but otherwise these statistics are not the same in general. The generating function $\sum_{T \in \mathrm{PP}^{\ell}(P)} q^{\mathrm{acard}(T)}$ for a minuscule poset $P$ was considered recently in~\cite{ding2019antichain}. However, note that, like size, the statistic $\mathrm{acard}$ does not play nicely with the doppelg\"{a}nger philosophy: as observed by Stembridge in~\cite[\S6]{stembridge1986trapezoidal}, the $\mathrm{acard}$-generating functions for the $P$-partitions of the rectangle and the trapezoid need not be the same.
\end{remark}

Before discussing what is known about Conjecture~\ref{conj:ddeg_gf}, let us explain how it is an instance of the ``minuscule doppelg\"{a}nger pairs pretend to have isomorphic comparability graphs'' refrain we've been harping on.

\begin{prop} \label{prop:com_ddeg_gf}
Let $P$ and $Q$ be posets with $\mathrm{com}(P)\simeq\mathrm{com}(Q)$. Then for any $\ell \geq 1$ we have
\[ \sum_{T \in \mathrm{PP}^{\ell}(P)}q^{\mathrm{ddeg}(T)}=\sum_{T \in \mathrm{PP}^{\ell}(Q)}q^{\mathrm{ddeg}(T)}. \]
\end{prop}
\begin{proof}
Recall Stanley's transfer map $\phi\colon \mathcal{O}(P)\to\mathcal{C}(P)$ from Section~\ref{sec:doppelganger_defs}. It is routine to check that for $T \in \mathrm{PP}^{\ell}(P)$ we have $\frac{1}{\ell}\mathrm{ddeg}(T) = \sum_{p\in P} \phi(\frac{1}{\ell}T)(p)$. (Indeed, $\phi$ can be seen as a piecewise-linear generalization of the bijection $I\mapsto \mathrm{max}(I)$ between order ideals and antichains.) Hence,
\[ \sum_{T \in \mathrm{PP}^{\ell}(P)}q^{\mathrm{ddeg}(T)}= \sum_{f \in \frac{1}{\ell}\mathbb{Z}^{P}\cap\mathcal{C}(P)} q^{\ell \cdot \sum_{p\in P}f(p)}.\] 
But the set $\frac{1}{\ell}\mathbb{Z}^{P}\cap\mathcal{C}(P)$ depends only on the comparability graph of $P$, thus proving the proposition.
\end{proof}

As the following example demonstrates, this equality of down-degree generating functions does not automatically happen for all doppelg\"{a}ngers.

\begin{figure}
\begin{tikzpicture}[scale=1]
	\SetFancyGraph
	\Vertex[NoLabel,x=0,y=1]{1}
	\Vertex[NoLabel,x=0,y=2]{2}
	\Vertex[NoLabel,x=0,y=3]{3}
	\Vertex[NoLabel,x=1,y=1]{4}
	\Vertex[NoLabel,x=1,y=2]{5}
	\Vertex[NoLabel,x=1,y=3]{6}
	\Edges[style={thick}](1,2)
	\Edges[style={thick}](2,3)
	\Edges[style={thick}](4,5)
	\Edges[style={thick}](5,6)
	\node at (0.5,0.25) {$P$};
\end{tikzpicture} \qquad \vline \qquad \begin{tikzpicture}[scale=1]
	\SetFancyGraph
	\Vertex[NoLabel,x=0.5,y=1]{1}
	\Vertex[NoLabel,x=0,y=2]{2}
	\Vertex[NoLabel,x=1,y=2]{3}
	\Vertex[NoLabel,x=0,y=3]{4}
	\Vertex[NoLabel,x=1,y=3]{5}
	\Vertex[NoLabel,x=2,y=3]{6}
	\Edges[style={thick}](1,2)
	\Edges[style={thick}](1,3)
	\Edges[style={thick}](2,4)
	\Edges[style={thick}](3,5)
	\Edges[style={thick}](3,6)
	\node at (0.5,0.25) {$Q$};
\end{tikzpicture}
\caption{A pair of doppelg\"{a}ngers which don't behave like they have isomorphic comparability graphs.} \label{fig:not_com}
\end{figure}
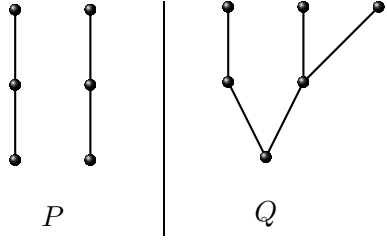

\begin{example} \label{ex:not_com}
Let $P$ and $Q$ be as in Figure~\ref{fig:not_com}. Then it can be verified that $P$ and $Q$ are doppelg\"{a}ngers with
\[\Omega_P(\ell) = \frac{1}{36}(\ell+1)^2(\ell+2)^2(\ell+3)^2=\Omega_Q(\ell).\]
(This example comes from~\cite{browning2017doppelgangers}.) But
\[ \sum_{I \in J(P)}q^{\mathrm{ddeg}(I)}= 9q^2 + 6q+1 \neq 2q^3 + 7q^2+6q+1=\sum_{I \in J(Q)}q^{\mathrm{ddeg}(I)}.\]
Indeed, we even have $\sum_{I \in J(P)}\mathrm{ddeg}(I) = 24\neq 26=\sum_{I \in J(Q)}\mathrm{ddeg}(I)$.
\end{example}

Note that a corollary to Proposition~\ref{prop:com_ddeg_gf} is that having a CDE (or mCDE) distributive lattice of order ideals is a comparability invariant:

\begin{cor}
Let $P$ and $Q$ be posets with $\mathrm{com}(P)\simeq\mathrm{com}(Q)$. Then
\begin{itemize}
\item $J(P)$ is CDE if and only if $J(Q)$ is CDE;
\item $J(P)$ is mCDE if and only if $J(Q)$ is mCDE;
\end{itemize}
\end{cor}
\begin{proof}
It follows from Theorem~\ref{thm:com_graph_dop} and Proposition~\ref{prop:com_ddeg_gf} that for $P$ and $Q$ posets with isomorphic comparability graphs we have
\[\frac{\sum_{T \in \mathrm{PP}^{\ell}(P)}\mathrm{ddeg}(T)}{\ell\cdot\#\mathrm{PP}^{\ell}(P)} = \frac{\sum_{T \in \mathrm{PP}^{\ell}(Q)}\mathrm{ddeg}(T)}{\ell\cdot\#\mathrm{PP}^{\ell}(Q)}.\]
But recall from above that
\[\mathbb{E}(\mathrm{mchain}_{J(P)}(\ell);\mathrm{ddeg}) = \frac{\sum_{T \in \mathrm{PP}^{\ell}(P)}\mathrm{ddeg}(T)}{\ell\cdot\#\mathrm{PP}^{\ell}(P)}.\]
So if $P$ and $Q$ have isomorphic comparability graphs, their distributive lattices of order ideals have the same multichain down-degree expectations, proving the corollary.
\end{proof}

However, having a tCDE distributive lattice of order ideals is not a comparability invariant: with $(P,Q)=(\Lambda_{\mathrm{Gr}(2,5)},T_{2,5})$, $P$ and $Q$ have isomorphic comparability graphs; but $J(P)$ is tCDE (for instance by Theorem~\ref{thm:minuscule_cde}) and, as mentioned, $J(Q)$ is not tCDE (see~\cite[Example 4.7]{hopkins2017cde}).

Now let us discuss what is known about Conjecture~\ref{conj:ddeg_gf}. An ideal proof of Conjecture~\ref{conj:ddeg_gf} would be a bijection between $\mathrm{PP}^{\ell}(P)$ and  $\mathrm{PP}^{\ell}(Q)$ which preserves the down-degree statistic. One might hope that the jeu-de-taquin style bijection of Hamaker et al.~\cite{hamaker2018doppelgangers} achieves this. Unfortunately, a quick check reveals that their bijection fails to preserve down-degree already for $\ell=1$ and $(P,Q) = (\Lambda_{\mathrm{Gr}(2,4)},T_{2,4})$; see~\cite[Example 1.6]{hamaker2018doppelgangers}. Nevertheless, for the case $\ell=1$ there actually \emph{is} a bijection already in the literature which does exactly what we want. Namely, we have:

\begin{prop} \label{prop:ddeg_gf_1}
Let $(P,Q) \in \{(\Lambda_{\mathrm{Gr}(k,n)},T_{k,n}),(\Lambda_{\mathrm{OG}(6,12)},\Phi^+(H_3)),(\Lambda_{\mathbb{Q}^{2n}},\Phi^+(I_2(2n)))\}$ be a minuscule doppelg\"{a}nger pair. Then 
\[ \sum_{I\in J(P)}q^{\mathrm{ddeg}(I)}=\sum_{I \in J(Q)}q^{\mathrm{ddeg}(I)}. \]
\end{prop}
\begin{proof}
The cases $(\Lambda_{\mathrm{OG}(6,12)},\Phi^+(H_3))$ and $(\Lambda_{\mathbb{Q}^{2n}},\Phi^+(I_2(2n)))$ are simple exercises. For $(P,Q) = (\Lambda_{\mathrm{Gr}(k,n)},T_{k,n})$, we want a bijection between $J(P)$ and $J(Q)$ which preserves down-degree. Via the bijection $I \mapsto \mathrm{max}(I)$, this is the same as a bijection between the antichains of $P$ and of $Q$ which preserves cardinality. Precisely such a bijection was constructed by Stembridge in~\cite{stembridge1986trapezoidal}.
\end{proof}

 \begin{remark}
Some other $\mathrm{ddeg}$-preserving bijections between $J(\Lambda_{\mathrm{Gr}(k,n)})$ and $J(T_{k,n})$ can be extracted from the papers of Elizalde~\cite{elizalde2015bijections} and Reiner~\cite{reiner1997noncrossing}.
\end{remark}

Of course, some cases of Conjecture~\ref{conj:ddeg_gf} are covered by Proposition~\ref{prop:com_ddeg_gf}. Among those cases not covered, we have checked by computer that Conjecture~\ref{conj:ddeg_gf} is true for
\[(P,Q) \in \{(\Lambda_{\mathrm{Gr}(3,7)},T_{3,7}),(\Lambda_{\mathrm{Gr}(3,8)},T_{3,8}),(\Lambda_{\mathrm{Gr}(4,8)},T_{4,8}),(\Lambda_{\mathrm{OG}(6,12)},\Phi^+(H_3))\}\]
 with $\ell=2,3,4$.
 
We have reviewed everything known about Conjecture~\ref{conj:ddeg_gf}. We have no serious, specific proposal of what a proof of Conjecture~\ref{conj:ddeg_gf} might look like, beyond the vague idea that perhaps one could construct a piecewise-linear bijection between $\mathcal{C}(P)$ and~$\mathcal{C}(Q)$ which maps $\frac{1}{\ell}\mathbb{Z}^P$ to $\frac{1}{\ell}\mathbb{Z}^Q$ and preserves the sum of coordinates.

In the next subsection we will present a different strengthening of Conjecture~\ref{conj:doppelganger_mcde} for which we can offer a more plausible program of attack.

\subsection{Set-valued \texorpdfstring{$P$}{P}-partitions}

Let $P$ be a poset. A \emph{set-valued $P$-partition of height~$\ell$} is a map $T\colon P \to \{X\subseteq \{0,1,\ldots,\ell\}\colon X\neq\varnothing\}$ from $P$ to the set of nonempty subsets of $\{0,1,\ldots,\ell\}$ such that $\mathrm{max}(T(p)) \leq \mathrm{min}(T(q))$ whenever $p\leq q \in P$. The \emph{excess} of a set-valued $P$-partition $T$ is $\sum_{p\in P}( \#T(p)-1)$. We denote the set of set-valued $P$-partitions of height $\ell$ and excess $e$ by $\mathrm{PP}^{\ell}_e(P)$.

 Set-valued $P$-partitions were introduced by Lam and Pylyavskyy~\cite[\S5.3]{lam2007combinatorial} in the course of their development of a $K$-theoretic generalization of the ring of quasisymmetric functions. The notion of set-valued $P$-partitions is derived from that of set-valued tableaux, which ultimately traces back to the work of Buch~\cite{buch2002littlewood} on the $K$-theory of Grassmannians. Observe that $\mathrm{PP}^{\ell}_0(P)=\mathrm{PP}^{\ell}(P)$ is just the set of ordinary $P$-partitions of height~$\ell$. Also note that $\mathrm{PP}^{\ell}_1(P)$ is what would be called in the language of Reiner-Tenner-Yong~\cite{reiner2018poset} the set of ``barely set-valued'' $P$-partitions of height~$\ell$.

Let us explain the connection of set-valued $P$-partitions to the CDE property. It goes through these ``barely set-valued'' $P$-partitions. The point is that there is a bijection
\[\mathrm{PP}^{\ell}_1(P) \to \{(T,i,p)\colon T\in\mathrm{PP}^{\ell}(P), i\in \{0,1,\ldots,\ell-1\},p\in \mathrm{max}(T^{-1}(\{0,1,\ldots,i\}))\} \]
given by $T\mapsto (T',i,p)$ where $p$ is the unique $p\in P$ with $\#T(p)=2$,  $i=\mathrm{max}(T(p))-1$, and $T'$ is obtained from $T$ by ``deleting'' the value $i+1$ at $p$. (This is a simple variant of~\cite[Proposition 3.5]{reiner2018poset}.) But clearly
\[\#\{(T,i,p)\colon T\in\mathrm{PP}^{\ell}(P), i\in \{0,\ldots,\ell-1\},p\in \mathrm{max}(T^{-1}(\{0,\ldots,i\}))\} = \hspace{-0.5cm} \sum_{T\in\mathrm{PP}^{\ell}(P)} \hspace{-0.5cm} \mathrm{ddeg}(T).\]
And recall from above that
\[\mathbb{E}(\mathrm{mchain}_{J(P)}(\ell);\mathrm{ddeg}) = \frac{\sum_{T \in \mathrm{PP}^{\ell}(P)}\mathrm{ddeg}(T)}{\ell\cdot\#\mathrm{PP}^{\ell}(P)}.\]
So in fact
\[\mathbb{E}(\mathrm{mchain}_{J(P)}(\ell);\mathrm{ddeg}) =\frac{\#\mathrm{PP}^{\ell}_1(P)}{\ell\cdot\#\mathrm{PP}^{\ell}(P)}. \]

This discussion shows that Conjecture~\ref{conj:doppelganger_mcde} is equivalent to the assertion that for $(P,Q)$ a minuscule doppelg\"{a}nger pair we have $\#\mathrm{PP}^{\ell}_1(P) = \#\mathrm{PP}^{\ell}_1(Q)$. The strengthening of Conjecture~\ref{conj:doppelganger_mcde} we propose is that in fact these posets have the same number of set-valued $P$-partitions of arbitrary excess.

\begin{conj} \label{conj:set_valued}
Let $(P,Q) \in \{(\Lambda_{\mathrm{Gr}(k,n)},T_{k,n}),(\Lambda_{\mathrm{OG}(6,12)},\Phi^+(H_3)),(\Lambda_{\mathbb{Q}^{2n}},\Phi^+(I_2(2n)))\}$ be a minuscule doppelg\"{a}nger pair. Then $\#\mathrm{PP}^{\ell}_e(P) = \#\mathrm{PP}^{\ell}_e(Q)$ for any $\ell \geq 1$, $e \geq 0$.
\end{conj}

Conjecture~\ref{conj:set_valued} is another instance of minuscule  doppelg\"{a}nger pair pretending to have isomorphic comparability graphs, as the following proposition shows.

\begin{prop} \label{prop:set_valued}
Let $P$ and $Q$ be posets with $\mathrm{com}(P)\simeq\mathrm{com}(Q)$. Then $\#\mathrm{PP}^{\ell}_e(P) = \#\mathrm{PP}^{\ell}_e(Q)$ for any $\ell \geq 1$, $e \geq 0$.
\end{prop}

\begin{proof}
By Lemma~\ref{lem:com_graph_criterion} we can assume that $Q$ is obtained from $P$ by dualizing an autonomous subset $A\subseteq P$. Then define a bijection from $\mathrm{PP}^{\ell}_e(P)$ to $\mathrm{PP}^{\ell}_e(Q)$ as follows. Let $T \in \mathrm{PP}^{\ell}_e(P)$. Set $\alpha \coloneqq  \mathrm{min}\{\mathrm{min}(T(p))\colon p \in A\}$ and $\omega \coloneqq   \mathrm{max}\{\mathrm{max}(T(p))\colon p \in A\}$. Define $T'$ by
\[ T'(q) \coloneqq  \begin{cases} T(q) &\textrm{if $q \notin A$}; \\ \{\alpha+\omega-x\colon x \in T(q)\} &\textrm{if $q \in A$}.\end{cases}\]
It is easy to see that $T' \in \mathrm{PP}^{\ell}_e(Q)$ and that the map $T\mapsto T'$ is indeed a bijection from $\mathrm{PP}^{\ell}_e(P)$ to $\mathrm{PP}^{\ell}_e(Q)$.
\end{proof}

Again, the equality of the number of set-valued $P$-partitions is not something that happens automatically for doppelg\"{a}ngers: let $P$ and $Q$ be as in Example~\ref{ex:not_com}; then the discussion in that example shows that~$\#\mathrm{PP}^{1}_1(P) = 24\neq 26=\#\mathrm{PP}^{1}_1(Q)$.

Now let us review what is known about Conjecture~\ref{conj:set_valued}:
\begin{itemize}
\item of course if $\mathrm{com}(P)\simeq \mathrm{com}(Q)$ then  Conjecture~\ref{conj:set_valued} is true by Proposition~\ref{prop:set_valued};
\item the case $e=0$ (for all values of $\ell$) is just the assertion that $P$ and $Q$ are doppelg\"{a}ngers and so is true;
\item as explained above, the case $e=1$ (for all values of $\ell$) is equivalent  to Conjecture~\ref{conj:doppelganger_mcde}, and so is true for $(P,Q)\in \{(\Lambda_{\mathrm{OG}(6,12)},\Phi^+(H_3)),(\Lambda_{\mathbb{Q}^{2n}},\Phi^+(I_2(2n)))\}$ and $(P,Q) =(\Lambda_{\mathrm{Gr}(k,n)},T_{k,n})$ with $k=1,2,n/2$;
\item the case $\ell=1$ (for all values of $e$) is easily seen to be equivalent to Proposition~\ref{prop:ddeg_gf_1} and so is true;
\item for $(P,Q)=(\Lambda_{\mathrm{Gr}(k,n)},T_{k,n})$ with $n\leq 8$ and $k\leq n/2$, and for $(P,Q)=(\Lambda_{\mathrm{OG}(6,12)},\Phi^+(H_3))$, we have verified by computer that Conjecture~\ref{conj:set_valued} is true for $\ell=2,3$ (and all values of $e$).
\end{itemize}

Unlike with Conjecture~\ref{conj:ddeg_gf}, for Conjecture~\ref{conj:set_valued} we can suggest a more specific, plausible strategy of proof: extend the $K$-theoretic jeu-de-taquin bijection of Hamaker et al.~\cite{hamaker2018doppelgangers} to the set-valued setting. Something along these lines is discussed in~\cite[\S6]{monical2018crystal}. A subtle point is that one will have to be careful to specify the order in which the jeu-de-taquin moves are carried out because of a lack of confluence of jeu-de-taquin in the $K$-theoretic setting.

\begin{remark}
For $P$ a poset and $T$ a set-valued $P$-partition, define the partition $\lambda_T$ to be the weakly decreasing rearrangement of the quantities $\#T(p)$ for $p\in P$. The bijection in the proof of Proposition~\ref{prop:set_valued} preserves this partition $\lambda_T$, so if two posets have isomorphic comparability graphs then the number of set-valued $P$-partitions of height $\ell$ with given $\lambda_T$ is the same for these two posets. But it is too much to hope for this to also be the case for minuscule doppelg\"{a}nger pairs. For instance, with $(P,Q) = (\Lambda_{\mathrm{Gr}(3,6)},T_{3,6})$, the number of set-valued $P$-partitions of height~$2$ with $\lambda_T=(2,2,2,2,2,1,1,1,1)$ is $2$, while the number of set-valued $Q$-partitions of height~$2$ with $\lambda_T=(2,2,2,2,2,1,1,1,1)$ is $3$. 

So the idea that minuscule doppelg\"{a}nger pairs behave the same as posets with isomorphic comparability graphs should not be taken too literally (although we do think it is an interesting heuristic).
\end{remark}

\section{Rowmotion} \label{sec:rowmotion}

In this section we further explore the idea of minuscule doppelg\"{a}nger pairs pretending to have isomorphic comparability graphs by considering the rowmotion operator.

\subsection{Rowmotion acting on order ideals}

Rowmotion is a certain invertible operator acting on the set of order ideals of any finite poset. Let $P$ be a poset. Recall that for any subset $A\subseteq P$ we use $\mathrm{max}(A)$ to denote the set of maximal elements of $A$ and $\mathrm{min}(A)$ to denote the set of minimal elements. For $I\in J(P)$, we define 
\[\mathrm{row}(I) \coloneqq  \{y\in P\colon y \leq x \textrm{ for some $x \in \mathrm{min}(P \setminus I)$}\}.\] 
It is clear that $\mathrm{row}(I) \in J(P)$ and so indeed we get an operator $\mathrm{row}\colon J(P)\to J(P)$, which we call \emph{rowmotion}. To see that this operator is invertible, we can realize it as the composition of three bijections. To that end, consider the maps:
\begin{align*}
I &\mapsto P\setminus I; \\
F &\mapsto \mathrm{min}(F); \\
A &\mapsto \{x \in P\colon x \leq y \textrm{ for some $y \in A$}\}.
\end{align*}
The first map is a bijection between order ideals of $P$ and order filters of $P$; the second is a bijection between order filters of $P$ and antichains of $P$; and the third is a bijection between antichains of $P$ and order ideals of $P$. (It is easy to see that each map is indeed a bijection.) Rowmotion is the composition of these three maps. Note that we can equivalently define rowmotion by ``$\mathrm{row}(I) \in J(P)$ is the unique order ideal of $P$ with $\mathrm{max}(\mathrm{row}(I))=\mathrm{min}(P \setminus I)$.''

There is another description of rowmotion due to Cameron and Fon-Der-Flaass~\cite{cameron1995orbits}. Recall the toggles $\tau_p\colon J(P) \to J(P)$ defined in Section~\ref{sec:cde_dist}. Let $p_1,p_2,\ldots,p_n$ be any linear extension of $P$. Cameron and Fon-Der-Flaass~\cite{cameron1995orbits} showed that
\[\mathrm{row} = \tau_{p_1} \circ \tau_{p_2} \circ \cdots \circ \tau_{p_n}.\]
The fact that $\tau_p$ and $\tau_q$ commute unless there is a cover relation between $p$ and $q$ means that this product of toggles is indeed independent of the choice of linear extension. This description of rowmotion explains the name ``rowmotion'': we view rowmotion as toggling the ``rows'' (i.e., the \emph{ranks}) of the poset $P$ in order from top to bottom. (The name ``rowmotion'' is due to Striker and Williams~\cite{striker2012promotion}.) Since the toggles are involutions, the above description also makes it clear that rowmotion is invertible.

\begin{example}
Let $P = [2]\times [2]$. Then the orbits of rowmotion acting on $J(P)$ are
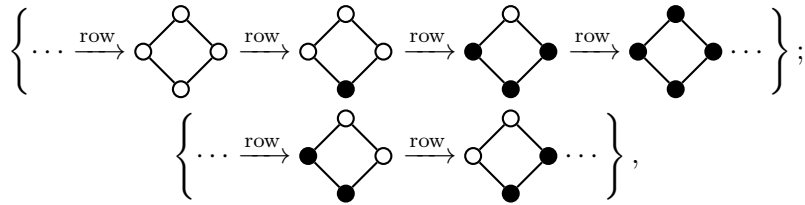
\begin{gather*}
\left\{ \cdots \xrightarrow{\mathrm{row}} {\parbox{0.5in}{\begin{tikzpicture}[scale=0.5]
	\node[circle, draw=black, fill=white, inner sep=0pt, thick, minimum size=6pt] (1) at (0,0) {};
	\node[circle, draw=black, fill=white, inner sep=0pt, thick, minimum size=6pt] (2) at (-1,1) {};
	\node[circle, draw=black, fill=white, inner sep=0pt, thick, minimum size=6pt] (3) at (1,1) {};
	\node[circle, draw=black, fill=white, inner sep=0pt, thick, minimum size=6pt] (4) at (0,2) {};
	\draw[thick] (4)--(3)--(1)--(2)--(4);
\end{tikzpicture} }} \xrightarrow{\mathrm{row}} {\parbox{0.5in}{\begin{tikzpicture}[scale=0.5]
	\node[circle, draw=black, fill=black, inner sep=0pt, thick, minimum size=6pt] (1) at (0,0) {};
	\node[circle, draw=black, fill=white, inner sep=0pt, thick, minimum size=6pt] (2) at (-1,1) {};
	\node[circle, draw=black, fill=white, inner sep=0pt, thick, minimum size=6pt] (3) at (1,1) {};
	\node[circle, draw=black, fill=white, inner sep=0pt, thick, minimum size=6pt] (4) at (0,2) {};
	\draw[thick] (4)--(3)--(1)--(2)--(4);
\end{tikzpicture} }}  \xrightarrow{\mathrm{row}} {\parbox{0.5in}{\begin{tikzpicture}[scale=0.5]
	\node[circle, draw=black, fill=black, inner sep=0pt, thick, minimum size=6pt] (1) at (0,0) {};
	\node[circle, draw=black, fill=black, inner sep=0pt, thick, minimum size=6pt] (2) at (-1,1) {};
	\node[circle, draw=black, fill=black, inner sep=0pt, thick, minimum size=6pt] (3) at (1,1) {};
	\node[circle, draw=black, fill=white, inner sep=0pt, thick, minimum size=6pt] (4) at (0,2) {};
	\draw[thick] (4)--(3)--(1)--(2)--(4);
\end{tikzpicture} }}  \xrightarrow{\mathrm{row}} {\parbox{0.5in}{\begin{tikzpicture}[scale=0.5]
	\node[circle, draw=black, fill=black, inner sep=0pt, thick, minimum size=6pt] (1) at (0,0) {};
	\node[circle, draw=black, fill=black, inner sep=0pt, thick, minimum size=6pt] (2) at (-1,1) {};
	\node[circle, draw=black, fill=black, inner sep=0pt, thick, minimum size=6pt] (3) at (1,1) {};
	\node[circle, draw=black, fill=black, inner sep=0pt, thick, minimum size=6pt] (4) at (0,2) {};
	\draw[thick] (4)--(3)--(1)--(2)--(4);
\end{tikzpicture} }} \cdots  \right\}; \\
\left \{ \cdots \xrightarrow{\mathrm{row}} {\parbox{0.5in}{\begin{tikzpicture}[scale=0.5]
	\node[circle, draw=black, fill=black, inner sep=0pt, thick, minimum size=6pt] (1) at (0,0) {};
	\node[circle, draw=black, fill=black, inner sep=0pt, thick, minimum size=6pt] (2) at (-1,1) {};
	\node[circle, draw=black, fill=white, inner sep=0pt, thick, minimum size=6pt] (3) at (1,1) {};
	\node[circle, draw=black, fill=white, inner sep=0pt, thick, minimum size=6pt] (4) at (0,2) {};
	\draw[thick] (4)--(3)--(1)--(2)--(4);
\end{tikzpicture} }} \xrightarrow{\mathrm{row}} {\parbox{0.5in}{\begin{tikzpicture}[scale=0.5]
	\node[circle, draw=black, fill=black, inner sep=0pt, thick, minimum size=6pt] (1) at (0,0) {};
	\node[circle, draw=black, fill=white, inner sep=0pt, thick, minimum size=6pt] (2) at (-1,1) {};
	\node[circle, draw=black, fill=black, inner sep=0pt, thick, minimum size=6pt] (3) at (1,1) {};
	\node[circle, draw=black, fill=white, inner sep=0pt, thick, minimum size=6pt] (4) at (0,2) {};
	\draw[thick] (4)--(3)--(1)--(2)--(4);
\end{tikzpicture} }}  \cdots  \right \}, 
\end{gather*}
where we fill in the elements of $I\in J(P)$. Observe that the order of $\mathrm{row}$ is $4$, and that the average of $\mathrm{ddeg}$ along each $\mathrm{row}$-orbit is $1$.
\end{example}

Rowmotion has been the subject of a lot of research over the past forty-plus years, with a renewed interest especially in the last ten years~\cite{brouwer1974period, fonderflaass1993orbits, cameron1995orbits, panyushev2009orbits, armstrong2013uniform, striker2012promotion, rush2013orbits}. The single poset on which the action of rowmotion has been studied the most is the rectangle $P=[a]\times [b]$. Initially the main interest was in computing the order of rowmotion and more generally in understanding its orbit structure. Brouwer and Schrijver~\cite{brouwer1974period} proved that the order of $\mathrm{row}$ acting on $J([a]\times [b])$ is $a+b$. Since $\#J([a]\times [b])=\binom{a+b}{b}$, this order is much smaller than what one would expect from a random invertible operator, implying that rowmotion acting on the rectangle is well-behaved. Later, Fon-Der-Flaass~\cite{fonderflaass1993orbits} determined exactly which divisors of $a+b$ appear as the sizes of orbits of $\mathrm{row}$ acting on $J([a]\times [b])$. Finally, Stanley~\cite[\S2]{stanley2009promotion} (see also~\cite[\S3.1]{striker2012promotion}) essentially completely resolved the problem of understanding the orbit structure by demonstrating that $J([a]\times [b])$ under rowmotion is in equivariant bijection with binary words (with $a$~1's and $b$~0's) under rotation.\footnote{The correspondence between rectangle poset order ideals under rowmotion and binary words under rotation was also independently discovered by Thomas; see~\cite[\S3.3.2]{propp2015homomesy}, where this correspondence is termed the ``Stanley-Thomas word.''}

Binary words under rotation have a particularly simple orbit structure. One way to compactly and precisely describe the orbit structure of binary words under rotation is via the \emph{cyclic sieving phenomenon} of Reiner, Stanton, and White~\cite{reiner2004cyclic} (indeed, binary words under rotation was one of the first examples of cyclic sieving considered by Reiner-Stanton-White). So to proceed further in our discussion of rowmotion, let's review the cyclic sieving phenomenon.

\begin{definition}
Let $X$ be a finite set with the action of a cyclic group $C=\langle c \rangle\simeq \mathbb{Z}/N\mathbb{Z}$. Let $X(q) \in \mathbb{N}[q]$ be a polynomial in $q$ with nonnegative integer coefficients. We say the triple $(X,C,X(q))$ \emph{exhibits the cyclic sieving phenomenon} if for all $k\geq 0$ we have $X(\zeta^k)=\#\{x\in X\colon c^k(x)=x\}$, where $\zeta = e^{2\pi i/N}$ is a primitive $N$th root of unity.
\end{definition}

Observe that if $(X,\mathbb{Z}/N\mathbb{Z},X(q))$ exhibits the cyclic sieving phenomenon then the entire orbit structure of $\mathbb{Z}/N\mathbb{Z}$ acting on $X$ is determined by $X(q)$ together with knowledge of the number~$N$. Thus if $X(q)$ has a simple form (e.g., if it has a nice product formula), then the orbit structure of this cyclic action must also be ``simple'' in some sense. When $\Phi\colon X \to X$ is an invertible operator, we use $\langle X \rangle$ to denote the free cyclic group action on $X$ generated by $\Phi$; but in general we can considering cyclic sieving for non-free actions as well.

As discussed above, it follows from Stanley's equivariant bijection and the seminal work of Reiner-Stanton-White that rowmotion acting on $J([a]\times [b])$ exhibits cyclic sieving for a polynomial with a simple product formula (the \emph{$q$-binomial coefficient}). Extending this result, Rush and Shi~\cite{rush2013orbits} later proved that rowmotion acting on $J(P)$ exhibits cyclic sieving for any minuscule poset $P$:

\begin{thm} \label{thm:minuscule_row_cyc_siev}
Let $P$ be a minuscule poset. Recall the size generating function $F(P,\ell)$ of $P$-partitions of height~$\ell$ defined in Remark~\ref{rem:ppart_wt_gf}:
\[F(P,\ell) =\prod_{p\in P} \frac{(1-q^{r(p)+\ell+1})}{(1-q^{r(p)+1})}.\]
Then $\mathrm{row}$ acting on $J(P)$ has order $h$ where $h \coloneqq  r(P)+2$ is Coxeter number of $P$. And $(J(P),\langle \mathrm{row}\rangle,F(P,1))$ exhibits cyclic sieving.
\end{thm}

Prior to the work of Rush-Shi, Panyushev~\cite{panyushev2009orbits} studied rowmotion acting on $J(P)$ where $P=\Phi^+$ is a root poset. He made a number of influential conjectures about this action. In addressing Panyushev's conjectures (as well as related conjectures of Bessis and Reiner~\cite{bessis2011cyclic}), Armstrong, Stump and Thomas~\cite{armstrong2013uniform} showed that rowmotion acting on $J(\Phi^+)$ also exhibits cyclic sieving:

\begin{thm} \label{thm:root_row_cyc_siev}
Let $P=\Phi^+$ be a crystallographic root poset. Define the \emph{$q$-$\Phi$-Catalan number} by
\[\mathrm{Cat}(\Phi;q) \coloneqq  \prod_{i=1}^{n} \frac{(1-q^{h+d_i})}{(1-q^{d_i})}\]
where $d_1,\ldots,d_n$ are the degrees of $\Phi$ and $h$ is its Coxeter number. Then $\mathrm{row}$ acting on $J(\Phi^{+})$ has order dividing $2h$. And if we let $C=\langle c\rangle\simeq\mathbb{Z}/2h\mathbb{Z}$ act on $J(P)$ by $c(I) :=\mathrm{row}(I)$, then $(J(\Phi^+),C,\mathrm{Cat}(\Phi;q))$ exhibits cyclic sieving.
\end{thm}

Though not obvious, we do have that $\mathrm{Cat}(\Phi;q)\in \mathbb{N}[q]$ (see~\cite{armstrong2013uniform} or~\cite[\S4]{bessis2011cyclic}).

Rowmotion acting on $J(\Phi^+)$ has order equal to either $h$ or $2h$: $\mathrm{row}^h$ acts as $-w_0$ where $w_0$ is the \emph{longest element} of the Weyl group $W$ of $\Phi$ (see~\cite{armstrong2013uniform}). This means that rowmotion acting on $J(\Phi^+)$ has order $2h$ for $\Phi=A_n$ $(n\geq 2)$, $D_{2n+1}$, $E_6$, and order $h$ otherwise.

Armstrong-Stump-Thomas's interest in Panyushev's action was its connection to Coxeter-Catalan combinatorics and in particular the relationship between noncrossing and nonnesting Coxeter-Catalan objects. (See~\cite{armstrong2009generalized} for background on Coxeter-Catalan combinatorics.) Note that Theorem~\ref{thm:root_row_cyc_siev} applies to \emph{all} crystallographic root posets, of coincidental type or not. (And actually it is easily checked that the conclusion of the theorem holds for $\Phi^+(H_3)$ and $\Phi^+(I_2(m))$ as well~\cite{cuntz2015root}.) Nevertheless, in the next subsection we will see a way in which the coincidental types behave better with respect to a certain natural generalization of rowmotion acting on order ideals.

In recent years, especially in the context of ``dynamical algebraic combinatorics''~\cite{roby2016dynamical, striker2017dynamical}, other aspects of rowmotion beyond its orbit structure have been investigated. There has been a particular focus on exhibiting \emph{homomesies} for rowmotion. So let's review the homomesy paradigm of Propp and Roby~\cite{propp2015homomesy}.

\begin{definition}
Let $X$ be a finite set, $\Psi\colon X \to X$ an invertible operator on $X$, and $f\colon X \to \mathbb{R}$ some statistic on $X$. Then we say that the tripe $(X,\Psi,f)$ \emph{exhibits homomesy} if the average of $f$ is the same along every $\Psi$-orbit of $X$. In this we also say that $f$ is \emph{homomesic} with respect to the action of $\Psi$ on $X$; and we say $f$ is \emph{$c$-mesic} if the average of $f$ along every $\Psi$-orbit is $c\in \mathbb{R}$.
\end{definition}

Homomesies of a dynamical system are in some sense ``dual'' to invariant quantities of the system (see~\cite[\S2.4]{propp2015homomesy}).

The main motivating example for Propp-Roby's introduction of the homomesy paradigm was an instance of homomesy conjectured by Panyushev~\cite{panyushev2009orbits} and proved by  Armstrong-Stump-Thomas~\cite{armstrong2013uniform}:

\begin{thm} \label{thm:root_row_homo}
Let $P=\Phi^+$ be a crystallographic root poset. Let $h \coloneqq  r(P)+2$ be its Coxeter number. Then the antichain cardinality statistic $I\mapsto \#\mathrm{max}(I)$ is $\#P/h$-mesic with respect to the action of $\mathrm{row}$ on $J(\Phi^+)$.
\end{thm}

Again, it is easily checked that the conclusion of Theorem~\ref{thm:root_row_homo} holds also for $\Phi^+(H_3)$ and $\Phi^+(I_2(m))$~\cite{cuntz2015root}. Also recall, as mentioned above in Remark~\ref{rem:pan}, that $\#\Phi^+/h=n/2$ where $n$ is the number of simple roots of~$\Phi$ (see~\cite{kostant1959principal}).

Following Panyushev and Armstrong-Stump-Thomas, Propp and Roby~\cite{propp2015homomesy} investigated homomesy for rowmotion acting on $J([a]\times [b])$. They exhibited a number of homomesies for this action. For instance, they showed that the antichain cardinality statistic is homomesic for this action (as are certain statistics which refine antichain cardinality). And they also showed that the \emph{order ideal cardinality} statistic $I\mapsto \#I$ is homomesic for this action (as are certain statistics which refine order ideal cardinality). Rush and Wang~\cite{rush2015orbits} extended (most of) Propp and Roby's homomesy results from the rectangle to all minuscule posets. In particular, they showed:

\begin{thm} \label{thm:minuscule_row_homo}
Let $P$ be a minuscule poset. Let $h \coloneqq  r(P)+2$ be its Coxeter number. Then the antichain cardinality statistic $I\mapsto \#\mathrm{max}(I)$ is $\#P/h$-mesic with respect to the action of $\mathrm{row}$ on $J(P)$.
\end{thm}

We will largely restrict our attention to the antichain cardinality statistic homomesy. The reason we focus on antichain cardinality is because there is a close connection between the CDE property and the antichain cardinality homomesy for rowmotion. This is not too surprising since, as mentioned earlier, the antichain cardinality statistic is also equal to down-degree in the lattice of order ideals. The precise connection, which follows from an observation of Striker~\cite{striker2015toggle} that for any $p\in P$, $\mathcal{T}_p$ is $0$-mesic with respect to the action of rowmotion, is the following (see, e.g.,~\cite[Corollary 7.9]{hopkins2017cde}):

\begin{lemma} \label{lem:cde_homo}
Let $P$ be a poset for which $J(P)$ is tCDE with edge density~$\delta$. Then the the antichain cardinality statistic is $\delta$-mesic with respect to the action of $\mathrm{row}$ on $J(P)$.
\end{lemma}

Observe that Lemma~\ref{lem:cde_homo} together with Theorem~\ref{thm:minuscule_cde} recovers all of Theorem~\ref{thm:minuscule_row_homo}; and together with Theorem~\ref{thm:root_poset_cde} this lemma recovers the part of Theorem~\ref{thm:root_row_homo} concerning root posets of coincidental type.

All of the above was review. Now let us consider rowmotion in the context of minuscule doppelg\"{a}ngers. We conjecture the following:

\begin{conj} \label{conj:min_dop_row}
Let $(P,Q) \in \{(\Lambda_{\mathrm{Gr}(k,n)},T_{k,n}),(\Lambda_{\mathrm{OG}(6,12)},\Phi^+(H_3)),(\Lambda_{\mathbb{Q}^{2n}},\Phi^+(I_2(2n)))\}$ be a minuscule doppelg\"{a}nger pair. Then there is a bijection $\varphi$ between the $\mathrm{row}$-orbits of $J(P)$ and the $\mathrm{row}$-orbits of $J(Q)$ such that for any $\mathrm{row}$-orbit $\mathcal{O}\subseteq J(P)$ we have
\begin{enumerate}
\item $\#\mathcal{O}=\#\varphi(\mathcal{O})$; \label{cond:minuscule_row_sizes}
\item $\sum_{I \in \mathcal{O}} \mathrm{ddeg}(I) = \sum_{I\in\varphi(\mathcal{O})} \mathrm{ddeg}(I)$. \label{cond:minuscule_row_ddeg}
\end{enumerate}
\end{conj}

In short, Conjecture~\ref{conj:min_dop_row} says that rowmotion behaves the same way on the two members of a minuscule doppelg\"{a}nger pair. Of course, in light of Theorems~\ref{thm:minuscule_row_cyc_siev} and~\ref{thm:minuscule_row_homo}, we could replace condition~\eqref{cond:minuscule_row_sizes} with the assertion that both members of a minuscule doppelg\"{a}nger pair exhibit cyclic sieving for the appropriate polynomial, and condition~\eqref{cond:minuscule_row_ddeg} with the assertion assertion that both members exhibit the antichain cardinality homomesy for rowmotion. But we chose to phrase the conjecture in this way to highlight that it is another instance of the minuscule doppelg\"{a}nger pairs behaving ``as if'' they had isomorphic comparability graphs. That is to say:

\begin{prop} \label{prop:com_graph_row}
Let $P$ and $Q$ be posets with $\mathrm{com}(P)\simeq\mathrm{com}(Q)$. Then there is a bijection $\varphi$ between the $\mathrm{row}$-orbits of $J(P)$ and the $\mathrm{row}$-orbits of $J(Q)$ such that for any $\mathrm{row}$-orbit $\mathcal{O}\subseteq J(P)$ we have
\begin{enumerate}
\item $\#\mathcal{O}=\#\varphi(\mathcal{O})$; \label{cond:com_graph_row_sizes}
\item $\sum_{I \in \mathcal{O}} \mathrm{ddeg}(I) = \sum_{I\in\varphi(\mathcal{O})} \mathrm{ddeg}(I)$. \label{cond:com_graph_row_ddeg}
\end{enumerate}
\end{prop}
\begin{proof}
Since we will be considering so many different rowmotion operators, let's use the notation $\mathrm{row}_P\colon J(P)\to J(P)$ to denote rowmotion acting on order ideals of~$P$. 

Let $P$ and $Q$ be posets with $\mathrm{com}(P)\simeq\mathrm{com}(Q)$. By Lemma~\ref{lem:com_graph_criterion} we can assume that~$Q$ is obtained from $P$ by dualizing an autonomous subset $A\subseteq P$. By abuse of notation, we use $A$ to denote the induced subposet of $P$ formed by the elements in~$A$. Thus the induced subposet of $Q$ formed by the elements of $A$ is $A^*$. Let's also define the subsets $U, L, N\subseteq P$ by
\begin{align*}
U &\coloneqq  \{p\in P\colon p > a \textrm{ for all $a \in A$}\}; \\
L &\coloneqq  \{p\in P\colon p < a \textrm{ for all $a \in A$}\}; \\
N &\coloneqq  \{p\in P\colon \textrm{$p$ is incomparable to $a$ for all $a \in A$}\}.
\end{align*}
Since $A$ is autonomous, $P$ is the disjoint union of $A$, $U$, $L$, and $N$.

First let's consider rowmotion acting on $J(A)$ and $J(A^*)$. The complementation map $c\colon J(A)\to J(A^*)$ defined by $c(I) \coloneqq  A\setminus I$ is a bijection, and moreover satisfies $c(\mathrm{row}_A(I)) = \mathrm{row}_{A^*}^{-1}(c(I))$ for all $I \in J(A)$. Also, we have $\mathrm{row}_A(A) = \varnothing$ and $\mathrm{row}_{A^*}(A) = \varnothing$. The previous two sentences imply that we can find a bijection $\psi\colon J(A) \to J(A^*)$ which satisfies all of the following properties:
\begin{itemize}
\item $\psi(\mathrm{row}_A(I)) = \mathrm{row}_{A^*}(\psi(I))$ for all $I\in J(A)$;
\item $\psi(I)$ and $c(I)$ belong to the same $\mathrm{row}_{A^*}$-orbit for all $I \in J(A)$;
\item $\psi(\varnothing) = \varnothing$ and $\psi(A)=A$.
\end{itemize}

Now we extend $\psi$ to a bijection $\widetilde{\psi}\colon J(P)\to J(Q)$ by setting $\widetilde{\psi}(I) \coloneqq  (I \setminus A) \cup \psi(I\cap A)$ for any $I\in J(P)$. Note that for any $I\in J(P)$, $I \cap U \neq \varnothing$ implies $I\cap A = A$, and $I \cap L \neq L$ implies $I\cap A = \varnothing$;  and the same is true for any $I\in J(Q)$. Therefore, the fact that $\psi(\varnothing) = \varnothing$ and $\psi(A)=A$ means that $\widetilde{\psi}$ really is a bijection from $J(P)$ to~$J(Q)$. 

We claim moreover that $\widetilde{\psi}(\mathrm{row}_{P}(I)) = \mathrm{row}_{Q}(\widetilde{\psi}(I))$ for all $I\in J(P)$. Indeed, the fact that $\psi(\varnothing) = \varnothing$ and $\psi(A)=A$ means that certainly $\widetilde{\psi}(\mathrm{row}_{P}(I))\cap X = \mathrm{row}_{Q}(\widetilde{\psi}(I))\cap X$ for $X=U$, $L$, or $N$ for all $I\in J(P)$. And then observe that, for $I \in J(P)$:
\begin{itemize}
\item if $\mathrm{row}_{P}(I)\cap U \neq \varnothing$, then $\mathrm{row}_{P}(I)\cap A = A$;
\item if $I \cap L \neq L$, then $\mathrm{row}_{P}(I)\cap A = \varnothing$;
\item if  $\mathrm{row}_{P}(I)\cap U = \varnothing$ and $I\cap L = L$, then $\mathrm{row}_P(I)\cap A = \mathrm{row}_A(I\cap A)$;
\end{itemize}
Together these imply that $\widetilde{\psi}(\mathrm{row}_{P}(I))\cap A = \mathrm{row}_{Q}(\widetilde{\psi}(I))\cap A$. So indeed the bijection $\widetilde{\psi}\colon J(P)\to J(Q)$ commutes with rowmotion. Then by setting $\varphi(\mathcal{O}) \coloneqq  \{\widetilde{\psi}(I)\colon I \in \mathcal{O}\}$ for any $\mathrm{row}_P$-orbit $\mathcal{O}\subseteq J(P)$, we obtain a bijection of rowmotion orbits satisfying~\eqref{cond:com_graph_row_sizes}. 

Finally, let's show~\eqref{cond:com_graph_row_ddeg} is also satisfied. So let $\mathcal{O}\subseteq J(P)$ be a $\mathrm{row}_P$-orbit. In fact, we claim something stronger than~\eqref{cond:com_graph_row_ddeg} holds: namely, that $\sum_{I \in \mathcal{O}} \mathcal{T}_{p^-}(I) = \sum_{I\in\varphi(\mathcal{O})} \mathcal{T}_{p^-}(I)$ for any $p\in P$. For any $p \notin A$, it is easy to see that we have $\mathcal{T}_{p^-}(I)= \mathcal{T}_{p^-}(\widetilde{\psi}(I))$ for all~$I\in J(P)$, which immediately gives us what we want. So from now on suppose that~$p \in A$. As we traverse the orbit $\mathcal{O}$, the intersection of our order ideal $I\in J(P)$ with $A$ looks like the sequence:
\[\ldots,I', \mathrm{row}_{A}(I'), \mathrm{row}^2_{A}(I'), \ldots, \mathrm{row}^{-1}_A(I'), \overline{A}, \overline{A},\ldots, \overline{A}, \overline{\varnothing}, \overline{\varnothing}, \ldots, \overline{\varnothing}, \ldots \]
and repeats that pattern, where we have overlined the $I\cap A$ for which either $I\cap U\neq \varnothing$ or $I\cap L\neq L$. Note that we might not have any overlined $I\cap A$'s, or conversely we could have a constant pattern $\overline{A},\overline{A},\ldots$ or $\overline{\varnothing},\overline{\varnothing},\ldots$; but importantly the not-overlined $I\cap A$'s do decompose into full $\mathrm{row}_A$-orbits. For any overlined $I\cap A$, $\mathcal{T}_{p^-}(I)= \mathcal{T}_{p^-}(\widetilde{\psi}(I))=0$, so we can safely ignore these. Now consider a not-overlined $\mathrm{row}_A$-orbit $\mathcal{O}' = \{I', \mathrm{row}_{A}(I'), \mathrm{row}^2_{A}(I'), \ldots, \mathrm{row}^{-1}_A(I')\}$ from the above sequence. By the second condition we imposed on $\psi$ above, we have $\{\psi(I')\colon I' \in \mathcal{O}'\}=\{c(I')\colon I'\in \mathcal{O}'\}$. Observe that we can toggle $p$ out of $I' \in J(A)$ if and only if we can toggle $p$ into $c(I')\in J(A^*)$. Moreover, for $I\in\mathcal{O}$ for which $I\cap A$ is not overlined, we can toggle $p$ out of $I\in J(P)$ if and only if we can toggle $p$ out of $I\cap A\in J(P)$; and similarly with toggling~$p$ in for $\widetilde{\psi}(I)\in J(Q)$ and $\widetilde{\psi}(I)\cap A\in J(A^*)$. So by grouping together the not-overlined $\mathrm{row}_A$-orbits in the above sequence we see that $\sum_{I \in \mathcal{O}} \mathcal{T}_{p^-}(I) = \sum_{I\in\varphi(\mathcal{O})} \mathcal{T}_{p^+}(I)$. But we also know that $\mathcal{T}_p$ averages to zero along any rowmotion orbit, which gives us $\sum_{I \in \mathcal{O}} \mathcal{T}_{p^-}(I) = \sum_{I\in\varphi(\mathcal{O})} \mathcal{T}_{p^-}(I)$, as desired.
\end{proof}

\begin{cor}
Let $P$ and $Q$ be posets with $\mathrm{com}(P)\simeq\mathrm{com}(Q)$. Then $\mathrm{ddeg}$ is homomesic with respect to the action of $\mathrm{row}$ on $J(P)$ if and only if $\mathrm{ddeg}$ is homomesic with respect to the action of $\mathrm{row}$ on $J(Q)$.
\end{cor}
\begin{proof}
This is immediate from Proposition~\ref{prop:com_graph_row}.
\end{proof}

To our knowledge, no one had previously thought to study the behavior of rowmotion on posets with isomorphic comparability graphs. So Proposition~\ref{prop:com_graph_row} is an instance in which the philosophy behind Problem~\ref{prob:main} actually aided in the discovery of comparability invariants.

Again, the conclusion of Proposition~\ref{prop:com_graph_row} fails badly under the weaker assumption that $P$ and $Q$ are arbitrary doppelg\"{a}ngers: with $P$ and $Q$ as in Example~\ref{ex:not_com}, the order of $\mathrm{row}\colon J(P)\to J(P)$ is $4$ while the order of $\mathrm{row}\colon J(Q)\to J(Q)$ is $12$.

Let us also briefly comment on another impossible strengthening of Proposition~\ref{prop:com_graph_row}. As explained in the proof of Proposition~\ref{prop:com_graph_row}, there exists a bijection between the $\mathrm{row}$-orbits of $J(P)$ and  $\mathrm{row}$-orbits of $J(P^{*})$ which preserves the multiset $\{\mathrm{ddeg}(I)\colon I\in \mathcal{O}\}$ of the down-degree statistic values along each orbit. But in general there is no such bijection between the rowmotion orbits for two posets with isomorphic comparability graphs, as Example~\ref{ex:com_graph_row_ddeg} below will show. That is, we cannot strengthen Proposition~\ref{prop:com_graph_row} by requiring the equality of multisets~$\{\mathrm{ddeg}(I)\colon I \in \mathcal{O}\}= \{\mathrm{ddeg}(I)\colon I \in \varphi(\mathcal{O})\}$. (Note, however, that the multisets $\{\mathrm{ddeg}(I)\colon I \in J(P)\}$ and $\{\mathrm{ddeg}(I)\colon I \in J(Q)\}$ actually \emph{are} equal: this is the content of Proposition~\ref{prop:com_ddeg_gf}.)

\begin{figure}
\begin{tikzpicture}[scale=1]
	\SetFancyGraph
	\Vertex[NoLabel,x=0,y=0]{1}
	\Vertex[NoLabel,x=1,y=0]{2}
	\Vertex[NoLabel,x=0,y=1]{3}
	\Vertex[NoLabel,x=1,y=1]{4}
	\Vertex[NoLabel,x=0.5,y=2]{5}
	\Vertex[NoLabel,x=1.5,y=2]{6}
	\Vertex[NoLabel,x=0.25,y=3]{7}
	\Edges[style={thick}](1,3)
	\Edges[style={thick}](2,3)
	\Edges[style={thick}](2,4)
	\Edges[style={thick}](3,7)
	\Edges[style={thick}](4,5)
	\Edges[style={thick}](4,6)
	\Edges[style={thick}](5,7)
	\Edges[style={thick}](6,7)
	\node at (0.5,-0.75) {$P$};
\end{tikzpicture} \qquad \vline \qquad \begin{tikzpicture}[scale=1]
	\SetFancyGraph
	\Vertex[NoLabel,x=0,y=0]{1}
	\Vertex[NoLabel,x=1,y=0]{2}
	\Vertex[NoLabel,x=0,y=1]{3}
	\Vertex[NoLabel,x=1,y=2.5]{4}
	\Vertex[NoLabel,x=0.5,y=1.5]{5}
	\Vertex[NoLabel,x=1.5,y=1.5]{6}
	\Vertex[NoLabel,x=0.25,y=3]{7}
	\Edges[style={thick}](1,3)
	\Edges[style={thick}](2,3)
	\Edges[style={thick}](2,5)
	\Edges[style={thick}](2,6)
	\Edges[style={thick}](3,7)
	\Edges[style={thick}](4,5)
	\Edges[style={thick}](4,6)
	\Edges[style={thick}](4,7)
	\node at (0.5,-0.75) {$Q$};
\end{tikzpicture}
\caption{Posets with isomorphic comparability graphs but different down-degree multisets for their rowmotion orbits.} \label{fig:com_graph_row_ddeg}
\end{figure}
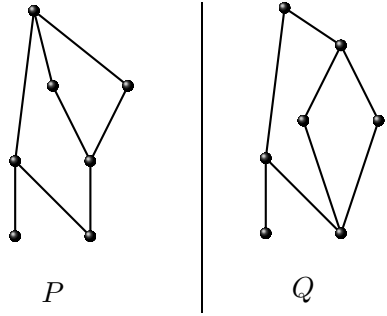

\begin{example} \label{ex:com_graph_row_ddeg}
Let $P$ and $Q$ be as in Figure~\ref{fig:com_graph_row_ddeg}. Then $P$ and $Q$ have isomorphic comparability graphs, but $\{ \{\mathrm{ddeg}(I)\colon I \in \mathcal{O}\}: \textrm{$\mathcal{O}$ a $\mathrm{row}$-orbit of $J(P)$} \}$ is:
\[ \{\{1, 0, 2, 2, 2, 1, 1, 2, 3\}, \{3, 1, 1\}, \{2, 1, 2, 2, 1, 2\}\}\]
while $\{\{\mathrm{ddeg}(I)\colon I \in \mathcal{O}\}: \textrm{$\mathcal{O}$ a $\mathrm{row}$-orbit of $J(Q)$} \} $ is:
\[\{\{1, 0, 2, 3, 1, 1, 1, 3, 2\}, \{2, 1, 2\}, \{2, 1, 2, 2, 1, 2\}\} .\]
These sets of multisets are evidently different.
\end{example}

Now let's return to a discussion of Conjecture~\ref{conj:min_dop_row}. First of all, by combining various results mentioned above in this section, we can prove some cases of this conjecture.

\begin{prop} \label{prop:min_dop_row}
Conjecture~\ref{conj:min_dop_row} is true for $(P,Q)=(\Lambda_{\mathrm{OG}(6,12)},\Phi^+(H_3))$, for $(P,Q)=(\Lambda_{\mathbb{Q}^{2n}},\Phi^+(I_2(2n)))$, and for $(P,Q) =(\Lambda_{\mathrm{Gr}(k,n)},T_{k,n})$ with $k=1,2,n/2$.
\end{prop}

\begin{proof}
For $(P,Q)=(\Lambda_{\mathbb{Q}^{2n}},\Phi^+(I_2(2n)))$ or $(P,Q) =(\Lambda_{\mathrm{Gr}(k,n)},T_{k,n})$ with $k=1,2$, the conjecture follows from Proposition~\ref{prop:com_graph_row} because in this case $\mathrm{com}(P)\simeq \mathrm{com}(Q)$.

For $(P,Q)=(\Lambda_{\mathrm{OG}(6,12)},\Phi^+(H_3))$ we can just check by hand or by computer.

Now consider $(P,Q) =(\Lambda_{\mathrm{Gr}(k,2k)},T_{k,2k})$. Recall that $Q=T_{k,2k}\simeq \Phi^+(B_k)$. By Theorems~\ref{thm:minuscule_row_homo} and~\ref{thm:root_row_homo}, $J(P)$ and $J(Q)$ both exhibit the antichain cardinality homomesy for rowmotion, so any bijection $\varphi$ satisfying condition~\eqref{cond:minuscule_row_sizes} of Conjecture~\ref{conj:min_dop_row} will satisfy condition~\eqref{cond:minuscule_row_ddeg} as well. (Recall that $J(P)$ and $J(Q)$ have the same edge density; see Proposition~\ref{prop:ddeg_gf_1}.) Therefore, we just need to show that $J(P)$ and $J(Q)$ have the same rowmotion orbit structure. This follows from the cyclic sieving results mentioned above (Theorems~\ref{thm:minuscule_row_cyc_siev} and~\ref{thm:root_row_cyc_siev}) together with the observation that
\[ \left[ F(\Lambda_{\mathrm{Gr}(k,2k)},1)\right]_{q\mapsto q^2}= \left [ \prod_{i=1}^{k}\prod_{j=1}^{k}  \frac{(1-q^{i+j})}{ (1-q^{i+j-1})} \right ]_{q\mapsto q^2} \hspace{-0.25cm} = \prod_{i=1}^{k}\frac{(1-q^{2k+2i})}{(1-q^{2i})} = \mathrm{Cat}(B_k;q),\]
for the relevant cyclic sieving polynomials.
\end{proof}

\begin{remark} \label{rem:cyc_siev_polys}
In the proof of Proposition~\ref{prop:min_dop_row} we saw an equality of the cyclic sieving polynomials appearing in Theorems~\ref{thm:minuscule_row_cyc_siev} and~\ref{thm:root_row_cyc_siev} when $(P,Q)=(\Lambda_{\mathrm{Gr}(k,2k)},\Phi^+(B_k))$. It can also be checked that for $(P,Q)=(\Lambda_{\mathrm{OG}(6,12)},\Phi^+(H_3))$ we have
\[\left [\prod_{i=1}^{5}\prod_{j=i}^{5} \frac{ (1-q^{i+j}) }{ (1-q^{i+j-1})}  \right ]_{q\mapsto q^2} = \frac{(1-q^{12})(1-q^{16})(1-q^{20})}{(1-q^{2})(1-q^{6})(1-q^{10})},\]
for the relevant cyclic sieving polynomials; and similarly for  $(P,Q)=(\Lambda_{\mathbb{Q}^{2n}},\Phi^+(I_2(2n)))$ we have
\[ \left [ \frac{(1-q^2)}{(1-q)} \cdot \prod_{i=1}^{2n-1} \frac{(1-q^{i+1})}{(1-q^{i})} \right ]_{q\mapsto q^2} = \frac{(1-q^{2n+2})(1-q^{4n})}{(1-q^2)(1-q^{2n})}.\]
Note that, even with the knowledge that rowmotion has the same orbit structure for $J(P)$ and $J(Q)$, the cyclic sieving phenomenon by itself does not automatically guarantee that these polynomials are equal; it only guarantees that they agree after modding out by $(1-q^{2h})$, where~$h$ is the relevant Coxeter number. It would be interesting to explain why we have these exact equalities of polynomials.
\end{remark}

Among the cases not covered by Proposition~\ref{prop:min_dop_row}, we have verified Conjecture~\ref{conj:min_dop_row} by computer for $(P,Q) =(\Lambda_{\mathrm{Gr}(k,n)},T_{k,n})$ for all $n \leq 10$ and $2 < k < n/2$.

\begin{remark} \label{rem:reu}
After the first version of this paper was made available online, much of Conjecture~\ref{conj:min_dop_row} was resolved affirmatively by Dao, Wellman, Yost-Wolff, and Zhang~\cite{dao2019trapezoid}. Namely, they established that the bijection of Hamaker et al.~\cite{hamaker2018doppelgangers} between $J(P)$ and~$J(Q)$ commutes with rowmotion. And they also showed that for $k\leq 3$, $J(T_{k,n})$ satisfies a condition similar to but weaker than the tCDE property, but which is still strong enough to imply the rowmotion antichain cardinality homomesy.
\end{remark}

In the remainder of this section we will present some strengthenings and variants of Conjecture~\ref{conj:min_dop_row} based on the \emph{piecewise-linear} and \emph{birational} liftings of rowmotion introduced by Einstein and Propp~\cite{einstein2013combinatorial, einstein2014piecewise}.

 \begin{remark}
\emph{Promotion} acting on the linear extensions of a poset $P$ is another operator with interesting dynamical properties (see~\cite{stanley2009promotion} for background on promotion); it is spiritually very similar to rowmotion acting on order ideals (see~\cite{striker2012promotion}). Promotion of linear extensions has good behavior for minuscule posets $P$: it is known that $\#P$ applications of promotion is the identity; furthermore, Rhoades~\cite[Theorem 1.3]{rhoades2010cyclic} has obtained a cyclic sieving result in the case of $P$ being a rectangle; and Purbhoo~\cite[Theorem 5.1]{purbhoo2018marvellous} extended Rhoades's work to the case of $P$ a shifted staircase. It makes sense to ask whether the minuscule doppelg\"{a}nger philosophy applies to promotion as well. Actually, there are known results for promotion, rather than mere conjectures. First of all, it is easy to establish (using Lemma~\ref{lem:com_graph_criterion}) that if $P$ and $Q$ are posets with isomorphic comparability graphs then they have the same orbit structure for promotion acting on their linear extensions. Furthermore, if $(P,Q)$ is a minuscule doppelg\"{a}nger pair, then $P$ and $Q$ also have the same orbit structure for promotion of linear extensions: for $(P,Q)=(\Lambda_{\mathbb{Q}^{2n}},\Phi^+(I_2(2n)))$, this is trivial; for $(P,Q)=(\Lambda_{\mathrm{OG}(6,12)},\Phi^+(H_3))$, this amounts to a finite check that can be carried out via computer; and for $(P,Q) =(\Lambda_{\mathrm{Gr}(k,n)},T_{k,n})$, this follows from the fact that a bijection due to Haiman~\cite[Proposition 8.11]{haiman1989mixed}~\cite[p.~112]{haiman1992dual} between the linear extensions of $P$ and $Q$ commutes with promotion. (The main bijection of Hamaker et~al.~\cite{hamaker2018doppelgangers} is an extension of Haiman's bijection from linear extensions to $P$-partitions, so compare this result of Haiman to the result of~\cite{dao2019trapezoid} mentioned in Remark~\ref{rem:reu}.)
\end{remark}

\subsection{Piecewise-linear rowmotion and rowmotion acting on \texorpdfstring{$P$}{P}-partitions}

We now describe the piecewise-linear lifting of rowmotion. We will review the history of this extension of rowmotion in the next subsection where we consider birational rowmotion.

Recall from Section~\ref{sec:doppelganger_defs} that $\mathbb{R}^{P}$ denotes the real vector space of functions $P\to \mathbb{R}$. Let $\widehat{P}$ denote the poset obtained from $P$ by adding a new minimal element $\widehat{0}$ and a new maximal element~$\widehat{1}$. We view any $f\in \mathbb{R}^{P}$ also being an element of $\mathbb{R}^{\widehat{P}}$ with the convention that we always have $f(\, \widehat{0} \,)=0$ and $f(\, \widehat{1} \,)=1$. For $p\in P$, we define the \emph{piecewise-linear toggle at $p$} $\tau_{p}^{\mathrm{PL}}\colon \mathbb{R}^{P}\to \mathbb{R}^{P}$ by
\[\tau_{p}^{\mathrm{PL}}(f)(q) \coloneqq  \begin{cases} f(q) &\textrm{if $p \neq q$}; \\ \mathrm{min}\{f(x)\colon\textrm{$x\in \widehat{P}$ covers $p$}\} +  \mathrm{max}\{f(x)\colon\textrm{$p$ covers $x\in \widehat{P}$} \}  - f(p) &\textrm{if $p=q$}.\end{cases}\]
It is clear that each $\tau_{p}^{\mathrm{PL}}$ is an involution.

For $I\in J(P)$, let $f_I \in \mathbb{R}^P$ denote the indicator function of its complement~$P\setminus I$. (In other words, $f_I$ is just the realization of the order ideal $I$ as a $P$-partition of height~$1$.) Then we have~$\tau^{\mathrm{PL}}_p(f_I) = f_{\tau_p(I)}$ for all $I\in J(P)$ and $p\in P$. So these piecewise-linear toggles do generalize the combinatorial toggles. 

Following  Cameron and Fon-Der-Flaass's description of combinatorial rowmotion as a composition of toggles, we then define \emph{piecewise-linear rowmotion} $\mathrm{row}^{\mathrm{PL}}\colon \mathbb{R}^{P}\to \mathbb{R}^{P}$ by
\[\mathrm{row}^{\mathrm{PL}} \coloneqq  \tau_{p_1}^{\mathrm{PL}} \circ \tau_{p_2}^{\mathrm{PL}} \circ \cdots \circ \tau_{p_n}^{\mathrm{PL}},  \]
where $p_1, p_2, \ldots, p_n$ is any linear extension of $P$. (It can again be checked that $\tau_{p}^{\mathrm{PL}}$ and $\tau_{q}^{\mathrm{PL}}$ commute unless there is a covering relation between $p$ and $q$, so that this composition is well-defined.)

\begin{remark}
Recall that we originally defined combinatorial rowmotion as a composition of three bijections (from order ideals to order filters to antichains back to order ideals). It is possible to generalize this definition to the piecewise-linear setting (and indeed the birational setting), essentially by using Stanley's transfer map~$\phi$ from the order polytope to chain polytope; see~\cite{einstein2013combinatorial, einstein2014piecewise, joseph2017antichain} for details.
\end{remark}

 As we have seen above, $P$-partitions are particularly well-behaved for the posets we care about (minuscule posets and root posets of coincidental type). Therefore, in this subsection we will actually primarily focus on piecewise-linear rowmotion acting on the rational points in the order polytope, i.e., on $P$-partitions.

The piecewise-linear toggles $\tau_{p}^{\mathrm{PL}}$ preserve the order polytope $\mathcal{O}(P) \subseteq \mathbb{R}^{P}$. Thus piecewise-linear rowmotion preserves the order polytope as well. It's also easy to see that these toggles preserve $\frac{1}{\ell}\mathbb{Z}^{P}$ for any $\ell \geq 1$, and thus rowmotion does as well. But recall that there is a canonical identification $\mathrm{PP}^{\ell}(P)\simeq\mathcal{O}(P)\cap \frac{1}{\ell}\mathbb{Z}^{P}$ given by $T\mapsto \frac{1}{\ell}T$. Thus by pulling back the action of $\mathrm{row}^{\mathrm{PL}}$ on $\mathcal{O}(P)\cap \frac{1}{\ell}\mathbb{Z}^{P}$, we get an action on $P$-partitions of height $\ell$, which we simply denote by~$\mathrm{row}\colon \mathrm{PP}^{\ell}(P)\to\mathrm{PP}^{\ell}(P)$. Note that $\mathrm{row}\colon\mathrm{PP}^{1}(P)\to\mathrm{PP}^{1}(P)$ is exactly the same as $\mathrm{row}\colon J(P)\to J(P)$.

\begin{example}
Let $P = [2]\times [2]$. Let's consider $\mathrm{row}$ acting on $\mathrm{PP}^3(P)$. We depict a $P$-partition $T\in\mathrm{PP}^{\ell}(P)$ by writing $T(p)$ on $p \in P$; and actually for such a $P$-partition we find it convenient to depict all of $\widehat{P}$ with the convention $T(\,\widehat{0}\,) = 0$ and $T(\,\widehat{1}\,) = \ell$. For example, for one choice of $T \in \mathrm{PP}^3(P)$ we could compute $\mathrm{row}(T)$ as a composition of piecewise-linear toggles as follows:
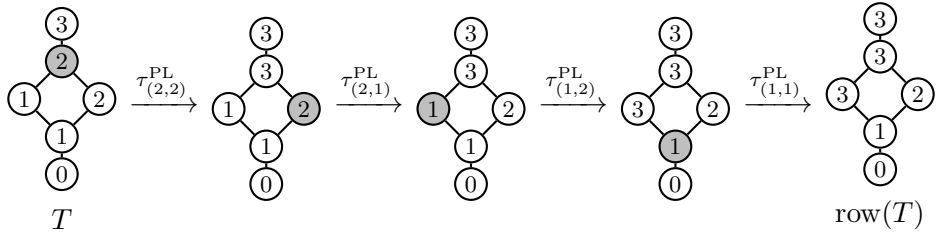
\[ \begin{tikzpicture}[scale=0.5]
	\node[circle, draw=black, fill=white, inner sep=0pt, thick, minimum size=12pt] (0) at (0,-1) {\footnotesize 0};
	\node[circle, draw=black, fill=white, inner sep=0pt, thick, minimum size=12pt] (1) at (0,0) {\footnotesize 1};
	\node[circle, draw=black, fill=white, inner sep=0pt, thick, minimum size=12pt] (2) at (-1,1) {\footnotesize 1};
	\node[circle, draw=black, fill=white, inner sep=0pt, thick, minimum size=12pt] (3) at (1,1) {\footnotesize 2};
	\node[circle, draw=black, fill=gray!50, inner sep=0pt, thick, minimum size=12pt] (4) at (0,2) {\footnotesize 2};
	\node[circle, draw=black, fill=white, inner sep=0pt, thick, minimum size=12pt] (5) at (0,3) {\footnotesize 3};
	\draw[thick] (4)--(3)--(1)--(2)--(4);
	\draw[thick] (0)--(1);
	\draw[thick] (4)--(5);
	\node at (0,-2.2) {$T$};
\end{tikzpicture} \parbox{0.5in}{\begin{center}\vspace{-1.5in}$\xrightarrow{\tau_{(2,2)}^{\mathrm{PL}}}$\end{center}} \begin{tikzpicture}[scale=0.5]
	\node[circle, draw=black, fill=white, inner sep=0pt, thick, minimum size=12pt] (0) at (0,-1) {\footnotesize 0};
	\node[circle, draw=black, fill=white, inner sep=0pt, thick, minimum size=12pt] (1) at (0,0) {\footnotesize 1};
	\node[circle, draw=black, fill=white, inner sep=0pt, thick, minimum size=12pt] (2) at (-1,1) {\footnotesize 1};
	\node[circle, draw=black, fill=gray!50, inner sep=0pt, thick, minimum size=12pt] (3) at (1,1) {\footnotesize 2};
	\node[circle, draw=black, fill=white, inner sep=0pt, thick, minimum size=12pt] (4) at (0,2) {\footnotesize 3};
	\node[circle, draw=black, fill=white, inner sep=0pt, thick, minimum size=12pt] (5) at (0,3) {\footnotesize 3};
	\draw[thick] (4)--(3)--(1)--(2)--(4);
	\draw[thick] (0)--(1);
	\draw[thick] (4)--(5);
	\node at (0,-2.2) {};
\end{tikzpicture} \parbox{0.5in}{\begin{center}\vspace{-1.5in}$\xrightarrow{\tau_{(2,1)}^{\mathrm{PL}}}$\end{center}} \begin{tikzpicture}[scale=0.5]
	\node[circle, draw=black, fill=white, inner sep=0pt, thick, minimum size=12pt] (0) at (0,-1) {\footnotesize 0};
	\node[circle, draw=black, fill=white, inner sep=0pt, thick, minimum size=12pt] (1) at (0,0) {\footnotesize 1};
	\node[circle, draw=black, fill=gray!50, inner sep=0pt, thick, minimum size=12pt] (2) at (-1,1) {\footnotesize 1};
	\node[circle, draw=black, fill=white, inner sep=0pt, thick, minimum size=12pt] (3) at (1,1) {\footnotesize 2};
	\node[circle, draw=black, fill=white, inner sep=0pt, thick, minimum size=12pt] (4) at (0,2) {\footnotesize 3};
	\node[circle, draw=black, fill=white, inner sep=0pt, thick, minimum size=12pt] (5) at (0,3) {\footnotesize 3};
	\draw[thick] (4)--(3)--(1)--(2)--(4);
	\draw[thick] (0)--(1);
	\draw[thick] (4)--(5);
	\node at (0,-2.2) {};
\end{tikzpicture} \parbox{0.5in}{\begin{center}\vspace{-1.5in}$\xrightarrow{\tau_{(1,2)}^{\mathrm{PL}}}$\end{center}} \begin{tikzpicture}[scale=0.5]
	\node[circle, draw=black, fill=white, inner sep=0pt, thick, minimum size=12pt] (0) at (0,-1) {\footnotesize 0};
	\node[circle, draw=black, fill=gray!50, inner sep=0pt, thick, minimum size=12pt] (1) at (0,0) {\footnotesize 1};
	\node[circle, draw=black, fill=white, inner sep=0pt, thick, minimum size=12pt] (2) at (-1,1) {\footnotesize 3};
	\node[circle, draw=black, fill=white, inner sep=0pt, thick, minimum size=12pt] (3) at (1,1) {\footnotesize 2};
	\node[circle, draw=black, fill=white, inner sep=0pt, thick, minimum size=12pt] (4) at (0,2) {\footnotesize 3};
	\node[circle, draw=black, fill=white, inner sep=0pt, thick, minimum size=12pt] (5) at (0,3) {\footnotesize 3};
	\draw[thick] (4)--(3)--(1)--(2)--(4);
	\draw[thick] (0)--(1);
	\draw[thick] (4)--(5);
	\node at (0,-2.2) {};
\end{tikzpicture} \parbox{0.5in}{\begin{center}\vspace{-1.5in}$\xrightarrow{\tau_{(1,1)}^{\mathrm{PL}}}$\end{center}} \begin{tikzpicture}[scale=0.5]
	\node[circle, draw=black, fill=white, inner sep=0pt, thick, minimum size=12pt] (0) at (0,-1) {\footnotesize 0};
	\node[circle, draw=black, fill=white, inner sep=0pt, thick, minimum size=12pt] (1) at (0,0) {\footnotesize 1};
	\node[circle, draw=black, fill=white, inner sep=0pt, thick, minimum size=12pt] (2) at (-1,1) {\footnotesize 3};
	\node[circle, draw=black, fill=white, inner sep=0pt, thick, minimum size=12pt] (3) at (1,1) {\footnotesize 2};
	\node[circle, draw=black, fill=white, inner sep=0pt, thick, minimum size=12pt] (4) at (0,2) {\footnotesize 3};
	\node[circle, draw=black, fill=white, inner sep=0pt, thick, minimum size=12pt] (5) at (0,3) {\footnotesize 3};
	\draw[thick] (4)--(3)--(1)--(2)--(4);
	\draw[thick] (0)--(1);
	\draw[thick] (4)--(5);
	\node at (0,-2.2) {$\mathrm{row}(T)$};
\end{tikzpicture} \]
We could further compute that the $\mathrm{row}$-orbit to which $T$ belongs is
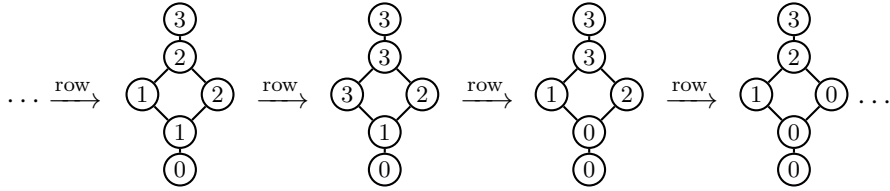
\[ \parbox{0.75in}{\begin{center}\vspace{-0.9in}$\cdots \xrightarrow{ \mathrm{row}}$\end{center}} \begin{tikzpicture}[scale=0.5]
	\node[circle, draw=black, fill=white, inner sep=0pt, thick, minimum size=12pt] (0) at (0,-1) {\footnotesize 0};
	\node[circle, draw=black, fill=white, inner sep=0pt, thick, minimum size=12pt] (1) at (0,0) {\footnotesize 1};
	\node[circle, draw=black, fill=white, inner sep=0pt, thick, minimum size=12pt] (2) at (-1,1) {\footnotesize 1};
	\node[circle, draw=black, fill=white, inner sep=0pt, thick, minimum size=12pt] (3) at (1,1) {\footnotesize 2};
	\node[circle, draw=black, fill=white, inner sep=0pt, thick, minimum size=12pt] (4) at (0,2) {\footnotesize 2};
	\node[circle, draw=black, fill=white, inner sep=0pt, thick, minimum size=12pt] (5) at (0,3) {\footnotesize 3};
	\draw[thick] (4)--(3)--(1)--(2)--(4);
	\draw[thick] (0)--(1);
	\draw[thick] (4)--(5);
\end{tikzpicture} \parbox{0.5in}{\begin{center}\vspace{-0.9in}$\xrightarrow{ \mathrm{row}}$\end{center}}  \begin{tikzpicture}[scale=0.5]
	\node[circle, draw=black, fill=white, inner sep=0pt, thick, minimum size=12pt] (0) at (0,-1) {\footnotesize 0};
	\node[circle, draw=black, fill=white, inner sep=0pt, thick, minimum size=12pt] (1) at (0,0) {\footnotesize 1};
	\node[circle, draw=black, fill=white, inner sep=0pt, thick, minimum size=12pt] (2) at (-1,1) {\footnotesize 3};
	\node[circle, draw=black, fill=white, inner sep=0pt, thick, minimum size=12pt] (3) at (1,1) {\footnotesize 2};
	\node[circle, draw=black, fill=white, inner sep=0pt, thick, minimum size=12pt] (4) at (0,2) {\footnotesize 3};
	\node[circle, draw=black, fill=white, inner sep=0pt, thick, minimum size=12pt] (5) at (0,3) {\footnotesize 3};
	\draw[thick] (4)--(3)--(1)--(2)--(4);
	\draw[thick] (0)--(1);
	\draw[thick] (4)--(5);
\end{tikzpicture} \parbox{0.5in}{\begin{center}\vspace{-0.9in}$\xrightarrow{ \mathrm{row}}$\end{center}}  \begin{tikzpicture}[scale=0.5]
	\node[circle, draw=black, fill=white, inner sep=0pt, thick, minimum size=12pt] (0) at (0,-1) {\footnotesize 0};
	\node[circle, draw=black, fill=white, inner sep=0pt, thick, minimum size=12pt] (1) at (0,0) {\footnotesize 0};
	\node[circle, draw=black, fill=white, inner sep=0pt, thick, minimum size=12pt] (2) at (-1,1) {\footnotesize 1};
	\node[circle, draw=black, fill=white, inner sep=0pt, thick, minimum size=12pt] (3) at (1,1) {\footnotesize 2};
	\node[circle, draw=black, fill=white, inner sep=0pt, thick, minimum size=12pt] (4) at (0,2) {\footnotesize 3};
	\node[circle, draw=black, fill=white, inner sep=0pt, thick, minimum size=12pt] (5) at (0,3) {\footnotesize 3};
	\draw[thick] (4)--(3)--(1)--(2)--(4);
	\draw[thick] (0)--(1);
	\draw[thick] (4)--(5);
\end{tikzpicture} \parbox{0.5in}{\begin{center}\vspace{-0.9in}$\xrightarrow{ \mathrm{row}}$\end{center}}  \begin{tikzpicture}[scale=0.5]
	\node[circle, draw=black, fill=white, inner sep=0pt, thick, minimum size=12pt] (0) at (0,-1) {\footnotesize 0};
	\node[circle, draw=black, fill=white, inner sep=0pt, thick, minimum size=12pt] (1) at (0,0) {\footnotesize 0};
	\node[circle, draw=black, fill=white, inner sep=0pt, thick, minimum size=12pt] (2) at (-1,1) {\footnotesize 1};
	\node[circle, draw=black, fill=white, inner sep=0pt, thick, minimum size=12pt] (3) at (1,1) {\footnotesize 0};
	\node[circle, draw=black, fill=white, inner sep=0pt, thick, minimum size=12pt] (4) at (0,2) {\footnotesize 2};
	\node[circle, draw=black, fill=white, inner sep=0pt, thick, minimum size=12pt] (5) at (0,3) {\footnotesize 3};
	\draw[thick] (4)--(3)--(1)--(2)--(4);
	\draw[thick] (0)--(1);
	\draw[thick] (4)--(5);
\end{tikzpicture}  \parbox{0.3in}{\begin{center}\vspace{-0.8in}$\cdots$\end{center}} \]
Observe that the size of this orbit is $4$, and that the average of $\mathrm{ddeg}$ along this orbit is $\frac{1}{4}(2+2+4+4) = 3$.
\end{example}

\begin{remark}
Recall that for $I\in J(P)$, $\mathrm{row}(I)$ is the unique order ideal of $P$ with $\mathrm{max}(\mathrm{row}(I))=\mathrm{min}(P \setminus I)$. There is a generalization of this for rowmotion acting on $P$-partitions. Namely, for $T\in \mathrm{PP}^{\ell}(P)$, $\mathrm{row}(T)$ is the unique element of $ \mathrm{PP}^{\ell}(P)$ with
\[\biguplus_{i=0}^{\ell-1} \mathrm{max}(\mathrm{row}(T)^{-1}(\{0,1,\ldots,i\})) = \biguplus_{i=0}^{\ell-1} \mathrm{min}(P\setminus T^{-1}(\{0,1,\ldots,i\})),  \] 
where $\uplus$ denotes multiset sum.
\end{remark}

\begin{remark}
There is a natural bijection $\mathrm{PP}^{\ell}(P)\simeq J(P\times[\ell])$. But the rowmotion acting on $P$-partitions of height $\ell$ that we have just defined in terms of piecewise-linear toggles is {\bf not the same} as order ideal rowmotion of $J(P\times [\ell])$. For example, for the rectangle poset $P=[a]\times[b]$, the order of $\mathrm{row}$ acting on $\mathrm{PP}^{\ell}(P)$ is $a+b$ for all $\ell \geq 1$ (see Theorem~\ref{thm:minuscule_row_pp_order} below); while Cameron and Fon-Der-Flaass~\cite{cameron1995orbits} showed that $\mathrm{row}$ acting on $J(P\times [2])$ has order $a+b+1$. Order ideal rowmotion acting on $J(P\times [\ell])$ for a minuscule poset $P$ was studied by Rush and Shi~\cite{rush2013orbits} and later by Mandel and Pechenik~\cite{mandel2018orbits} in the context of cyclic sieving, but the results they obtained were somewhat sporadic. We believe rowmotion acting on $\mathrm{PP}^{\ell}(P)$ behaves better than rowmotion acting on $J(P\times [\ell])$: for example, below we conjecture a uniform cyclic sieving result for rowmotion acting on $\mathrm{PP}^{\ell}(P)$ when $P$ is a minuscule poset.
\end{remark}

Let's now consider rowmotion acting on $P$-partitions for a minuscule poset $P$. Very recently Garver, Patias, and Thomas~\cite{garver2018minuscule} proved in a uniform manner that rowmotion acting on $P$-partitions of height~$\ell$ has the ``correct'' order  for a minuscule poset $P$: 

\begin{thm} \label{thm:minuscule_row_pp_order}
Let $P$ be a minuscule poset and let $h \coloneqq  r(P)+2$ be its Coxeter number. Then for any $\ell \geq 1$, $\mathrm{row}\colon \mathrm{PP}^{\ell}(P)\to\mathrm{PP}^{\ell}(P)$ has order $h$.
\end{thm}

(Theorem~\ref{thm:minuscule_row_pp_order} was essentially established earlier in a case-by-case manner by Grinberg and Roby~\cite{grinberg2016birational1, grinberg2015birational2}, and will be established in a case-by-case manner in a paper of Okada~\cite{okada2019birational} in preparation; see the discussion after Conjecture~\ref{conj:root_bi_row_ord} below.)

\begin{remark} \label{rem:gpt_details}
Garver-Patrias-Thomas~\cite{garver2018minuscule} studied (piecewise-linear) \emph{promotion} acting on $\mathrm{PP}^{\ell}(P)$, which is a slightly different operator than (piecewise-linear) rowmotion acting on $\mathrm{PP}^{\ell}(P)$: it is also a composition of all of the toggles $\tau_{p}^{\mathrm{PL}}$, but in a different order than for rowmotion. Nevertheless, it follows from the arguments of Striker and Williams~\cite{striker2012promotion} that promotion and rowmotion are conjugate elements of the ``toggle group'' (i.e., the group generated by the toggles) and so these two operators have the same orbit structure. Also, Garver-Patrias-Thomas~\cite{garver2018minuscule} claimed only that the order of promotion \emph{divides} $h$. But in fact it is easy to see that there is always at least one rowmotion orbit of size $h$: for $i=0,1,\ldots,r(P)+1$ we can define $T_i\in \mathrm{PP}^{\ell}(P)$ by
\[T_i(p) \coloneqq  \begin{cases} 0 &\textrm{ if $r(p)<i$}, \\ \ell &\textrm{otherwise};\end{cases}\]
and observe that $\{T_0,T_1,\ldots,T_{r(P)+1}\} \subseteq \mathrm{PP}^{\ell}(P)$ is a $\mathrm{row}$-orbit.
\end{remark}

Furthermore, we conjecture the following cyclic sieving result for rowmotion acting on $P$-partitions of a minuscule poset $P$, extending Rush and Shi's Theorem~\ref{thm:minuscule_row_cyc_siev} (which is the case $\ell=1$):

\begin{conj} \label{conj:minuscule_row_cyc_siev}
Let $P$ be a minuscule poset. Recall the size generating function $F(P,\ell)$ of $P$-partitions of height~$\ell$ defined in Remark~\ref{rem:ppart_wt_gf}:
\[F(P,\ell) = \prod_{p\in P}\frac{(1-q^{r(p)+\ell+1})}{(1-q^{r(p)+1})}.\]
Then for any $\ell \geq 1$, $(\mathrm{PP}^{\ell}(P),\langle \mathrm{row}\rangle,F(P,\ell))$ exhibits cyclic sieving.
\end{conj}

\begin{remark} \label{rem:cyc_siev_rect}
Conjecture~\ref{conj:minuscule_row_cyc_siev} is known to be true in Type A, i.e., in the case of the rectangle $P=[a]\times[b]$. In fact, there are two different proofs of this cyclic sieving result, as we now explain. 

The first proof is based on relating (piecewise-linear) rowmotion of $P$-partitions to  (jeu-de-taquin) promotion of semistandard Young tableaux. To relate rowmotion of $P$-partitions to promotion of semistandard tableaux, we first need to instead consider (piecewise-linear) \emph{promotion} of $P$-partitions: this is a slightly different operator than rowmotion, but, as discussed in Remark~\ref{rem:gpt_details} above, it follows from Striker-Williams~\cite{striker2012promotion} that these two operators have the same orbit structure. Then, as explained in~\cite[pp.~516-517]{einstein2014piecewise} or~\cite[Appendix A]{hopkins2019cyclic}, we observe that there is a simple bijection (based on Gelfand-Tsetlin patterns) from $\mathrm{PP}^{\ell}([a]\times[b])$ to the set of semistandard Young tableaux of shape~$a \times \ell$ with entries in $\{1,2,\ldots,a+b\}$ such that (piecewise-linear) promotion of $P$-partitions corresponds to (jeu-de-taquin) promotion of semistandard tableaux. And finally, Rhoades~\cite[Theorem 1.4]{rhoades2010cyclic} proved that promotion of semistandard tableaux of rectangular shape with bounded entries exhibits cyclic sieving for the appropriate polynomial. Rhoades's argument employed the dual canonical basis for representations of~$GL_n$~\cite{lusztig1990canonical, du1992canonical}.

Recently, Shen and Weng~\cite{shen2018cyclic} gave a different proof of Conjecture~\ref{conj:minuscule_row_cyc_siev} in the case of the rectangle. Shen and Weng worked directly with piecewise-linear toggles, although again they considered promotion of $P$-partitions, rather than rowmotion. Their approach was from the perspective of the ``cluster duality conjecture'' of Fock and Goncharov~\cite{fock2009cluster} and applied the breakthrough work of Gross, Hacking, Keel, and Kontsevich~\cite{gross2018canonical} on canonical bases for cluster algebras.
\end{remark}

\begin{remark}
Rhoades~\cite[Theorem 1.5]{rhoades2010cyclic} also established a content-refined version of cyclic sieving for rectangular semistandard Young tableaux. Namely, if $\mu$ is a vector of nonnegative integers invariant under $d$-fold rotation, then the $d$th power of promotion acts on the set of rectangular semistandard tableaux with content $\mu$. Rhoades showed that this action exhibits cyclic sieving (at least up to taking absolute value) with the sieving polynomial being a Kostka-Foulkes polynomial. (Fontaine and Kamnitzer~\cite[Theorem 5.2]{fontaine2014cyclic} found the correct power of $q$ to multiply the Kostka-Foulkes polynomial by to give cyclic sieving on the nose.)

There is possibly a refined version of Conjecture~\ref{conj:minuscule_row_cyc_siev} as well, as we know explain. Let $P$ be a minuscule poset corresponding to the minuscule weight $\omega$ of the irreducible crystallographic root system $\Phi$. Let $W$ be the Weyl group and $\Lambda$ the weight lattice of~$\Phi$. There is a canonical identification, which we call \emph{weight}, between the distributive lattice $J(P)$ and the Weyl group orbit of $\omega$: $\mathrm{wt}\colon J(P)\simeq W(\omega)$. Rush and Shi~\cite[Theorem 6.3]{rush2013orbits} showed that for each simple reflection $s_i \in W$ there is a corresponding subset $X_i \subseteq P$ so that $\mathrm{wt}((\prod_{p\in X_i} \tau_p) I) = s_i(\mathrm{wt}(I))$ for all $I\in J(P)$. We can extend the weight map to $\mathrm{wt}\colon \mathrm{PP}^{\ell}(P)\to \Lambda$ by defining $\mathrm{wt}(T) := \sum_{i=0}^{\ell-1}\mathrm{wt}(T^{-1}(\{0,1,\ldots,i\}))$. It is possible to extend Rush and Shi's result to show that $\mathrm{wt}((\prod_{p\in X_i} \tau^{\mathrm{PL}}_p) T) = s_i(\mathrm{wt}(T))$ for all $T\in \mathrm{PP}^{\ell}(P)$. This means that if $\mu \in \Lambda$ is a weight for which $c^d(\mu)=\mu$, where $c\in W$ is a Coxeter element, then the $d$th power of piecewise-linear promotion acts on the set of $T\in \mathrm{PP}^{\ell}(P)$ with $\mathrm{wt}(T)=\mu$. It seems highly plausible that this action exhibits cyclic sieving (at least up to taking absolute value) with the sieving polynomial being Lusztig's $q$-analog of weight multiplicity~\cite{lusztig1983singularities, joseph2000brylinski}. (Lusztig's $q$-analogs of weight multiplicity are the same as the Kostka-Foulkes polynomials in Type A.) We thank Joel Kamnitzer for suggesting that we consider the $q$-analog of weight multiplicity.
\end{remark}

\begin{remark} \label{rem:invariant_tensors}
Some of Rhoades's cyclic sieving results from~\cite{rhoades2010cyclic} have been extended to other minuscule types in the work of Fontaine and Kamnitzer~\cite{fontaine2014cyclic}, and of Westbury and his collaborators~\cite{westbury2016invariant, rubey2015symplectic, pfannerer2018promotion}. The basic idea is to study the Henriques-Kamnitzer~\cite{henriques2006crystals} cactus group action on the tensor product of crystals corresponding to minuscule representations. But, as far as we can tell, this crystal perspective does not seem to be directly related to Conjecture~\ref{conj:minuscule_row_cyc_siev}. However, there may be a connection between this crystal perspective and Conjecture~\ref{conj:root_poset_cyc_siev} below: that is because the dimension of the space of invariant tensors of $\otimes^{r}V$, where $V$ is the spin representation of the spin group $\mathrm{Spin}(2n+1)$, is equal to $\#\mathrm{PP}^{n}(\Phi^{+}(A_r))$ (see~\cite[Example 2.4]{pfannerer2018promotion}). We thank Bruce Westbury for explaining this possible connection to us.
\end{remark}

Among the cases not addressed by Remark~\ref{rem:cyc_siev_rect}, we have checked Conjecture~\ref{conj:minuscule_row_cyc_siev} by computer with $\ell=2,3,4$ for $P=\Lambda_{\mathrm{OG}(n,2n)}$ with $2\leq n \leq 7$, for $P=\Lambda_{\mathbb{Q}^{2n}}$ with $2\leq n \leq 6$, and for $P=\Lambda_{\mathbb{OP}^2}$ and $P=\Lambda_{G_{\omega}(\mathbb{O}^3,\mathbb{O}^6)}$.

Now let's consider rowmotion acting on $P$-partitions of height~$\ell$ for $P=\Phi^+$ a root poset. Unfortunately, already for $P=\Phi^+(D_4)$ and $\ell=2$ there is a $\mathrm{row}$-orbit of $\mathrm{PP}^{\ell}(P)$ of size $54$, which is a lot bigger than the Coxeter number~$h=6$ and dashes any hope of good behavior for rowmotion acting on the $P$-partitions of an arbitrary root poset. However, for root posets of coincidental type it does appear that we still have good behavior. For instance, results of Grinberg-Roby~\cite{grinberg2016birational1, grinberg2015birational2} imply that the order of rowmotion acting on $\mathrm{PP}^{\ell}(\Phi^+)$ for $\Phi$ of coincidental type divides $2h$:

\begin{thm}
Let $P = \Phi^+$ be a root poset of coincidental type. Let $h$ denote the Coxeter number of $\Phi$. Then for any $\ell \geq 1$, $\mathrm{row}$ acting on $\mathrm{PP}^{\ell}(P)$ has order
\[ \mathrm{ord}(\mathrm{row}) =\begin{cases} 2h &\Phi=A_n (n\geq 2), I_2(2m+1); \\ h &\Phi = B_n, H_3, I_2(2m).\end{cases}\]
\end{thm}
\begin{proof}
The map $T \mapsto \ell \cdot T$ gives a $\mathrm{row}$-equivariant embedding of $\mathrm{PP}^{1}(P)$ into $\mathrm{PP}^{\ell}(P)$. Since it is known that the order of $\mathrm{row}$ acting on $\mathrm{PP}^{1}(P)$ is the claimed value (see the discussion after Theorem~\ref{thm:root_row_cyc_siev} above), we need only prove that the order of $\mathrm{row}$ acting on $\mathrm{PP}^{\ell}(P)$ divides the claimed value.

For $\Phi$ other than $\Phi=B_n$ the birational analog of this result was proved by Grinberg-Roby~\cite{grinberg2016birational1, grinberg2015birational2} (see the discussion after Conjecture~\ref{conj:root_bi_row_ord} below). Via tropicalization it therefore holds at the piecewise-linear level as well. Meanwhile, the $P$-partitions in $\mathrm{PP}^{\ell}(\Phi^+(B_n))$ naturally correspond to the $P$-partitions in $\mathrm{PP}^{\ell}(\Phi^+(A_{2n-1}))$ which are invariant under the reflection across the vertical axis of symmetry of $\Phi^+(A_{2n-1})$. This embedding of $\mathrm{PP}^{\ell}(\Phi^+(B_n))$ into $\mathrm{PP}^{\ell}(\Phi^+(A_{2n-1}))$ is easily seen to be $\mathrm{row}$-equivariant. A theorem of Grinberg-Roby~\cite[Theorem 65]{grinberg2015birational2} says that $\mathrm{row}^{n}$ acting on $\mathrm{PP}^{\ell}(\Phi^+(A_{2n-1}))$ is the reflection across the vertical axis of symmetry of $\Phi^+(A_{2n-1})$. This gives the desired result for $\Phi=B_n$ as well.
\end{proof}

 Moreover, we conjecture the following cyclic sieving result:

\begin{conj} \label{conj:root_poset_cyc_siev}
Let $P = \Phi^+$ be a root poset of coincidental type.  Define the \emph{$q$-$\Phi$-multi-Catalan number} by
\[\mathrm{Cat}(\Phi,\ell;q) \coloneqq  \prod_{j=0}^{\ell-1} \prod_{i=1}^{n} \frac{(1-q^{h+d_i+2j})}{(1-q^{d_i+2j})}\]
where $d_1,\ldots,d_n$ are the degrees of $\Phi$ and $h$ is its Coxeter number. Then for any $\ell \geq 1$, $\mathrm{row}$ acting on $\mathrm{PP}^{\ell}(P)$ has order dividing $2h$. And if we let $C=\langle c\rangle\simeq\mathbb{Z}/2h\mathbb{Z}$ act on $\mathrm{PP}^{\ell}(P)$ by $c(I) :=\mathrm{row}(I)$, then $(\mathrm{PP}^{\ell}(P),C,\mathrm{Cat}(\Phi,\ell;q))$ exhibits cyclic sieving.
\end{conj}

Though not obvious, we do have that $\mathrm{Cat}(\Phi,\ell;q) \in \mathbb{N}[q]$. Indeed, in Type~A we have
\[ \mathrm{Cat}(A_n,\ell;q)= \prod_{1\leq i \leq j \leq n} \frac{(1-q^{2\ell+i+j})}{(1-q^{i+j})},\]
which is $q^{\ell \cdot \binom{n+1}{2}}$ times the expression denoted ``(CGI)'' by Proctor in~\cite{proctor1990new}. Proctor showed that this expression is a generating function for $P$-partitions in $\mathrm{PP}^{\ell}(\Phi^+(A_n))$ with respect to a certain statistic. Meanwhile, for Types $B_n$, $H_3$, and $I_2(2m)$ the polynomials $\mathrm{Cat}(\Phi,\ell;q)$ are basically the same as the size generating functions for the $P$-partitions of their minuscule doppelg\"{a}ngers (see Remark~\ref{rem:cyc_siev_conj_polys} below).

\begin{remark}
The multi-Catalan numbers are {\bf not} the same as the more well-known \emph{Fuss-Catalan numbers} (see~\cite[\S3.5]{armstrong2009generalized} and~\cite{fomin2005generalized} for the Fuss-Catalan numbers). These polynomials $\mathrm{Cat}(\Phi,\ell;q)$ first appeared in a paper of Ceballos, Labb\'{e}, and Stump~\cite[\S9]{ceballos2014subword} concerning multi-triangulations and the multi-cluster complex. In particular, Ceballos-Labb\'{e}-Stump conjectured for the coincidental types that $\mathrm{Cat}(\Phi,\ell;q)$ is a cyclic sieving polynomial for the action of the Auslander-Reiten translate (a.k.a.~the ``deformed'' or ``tropical'' Coxeter element) on the facets of the multi-cluster complex. In Type A, this action is rotation of the multi-triangulations of a polygon. Their conjecture remains open (to our knowledge), except for the case $\ell=1$ which is addressed by the work of Eu and Fu~\cite{eu2008cyclic}. That these same polynomials appear in Conjecture~\ref{conj:root_poset_cyc_siev} suggests that perhaps toggling of coincidental type root poset $P$-partitions could be related to cluster algebras and/or Auslander-Reiten quivers. However, note that the orders of rowmotion and of the Auslander-Reiten translate are not the same, even in the case $\ell=1$. At any rate, it is known that the multi-Catalan numbers do count the root poset $P$-partitions for the coincidental types: see~\cite[\S4.6.1]{williams2013cataland} and~\cite[Theorem 3]{hamaker2015subwords}.
\end{remark}

The case $\ell=1$ of Conjecture~\ref{conj:root_poset_cyc_siev} is of course a consequence of Theorem~\ref{thm:root_row_cyc_siev} because $\mathrm{Cat}(\Phi,1;q)=\mathrm{Cat}(\Phi;q)$. We have checked Conjecture~\ref{conj:root_poset_cyc_siev} by computer with $\ell=2,3,4$ for $\Phi=A_n$ with $n \leq 5$, for $\Phi=B_n$ with $n \leq 4$, for $\Phi=H_3$, and for $\Phi=I_2(m)$ with $3 \leq m \leq 6$.

So far we have discussed the order/orbit structure of rowmotion acting on $P$-partitions (i.e., we've discussed extensions of Theorems~\ref{thm:minuscule_row_cyc_siev} and~\ref{thm:root_row_cyc_siev}). But we can also consider homomesies for rowmotion acting on $P$-partitions (i.e., extensions of Theorems~\ref{thm:minuscule_row_homo} and~\ref{thm:root_row_homo}). It turns out that the CDE property (or rather, the tCDE property) continues to be extremely useful here, as we now explain.

The homomesy that we are interested in generalizing is the one involving antichain cardinality. So we need a piecewise-linear analog of the antichain cardinality statistic. First we define the \emph{piecewise-linear toggleability statistics} $\mathcal{T}^{\mathrm{PL}}_{p^+}, \mathcal{T}^{\mathrm{PL}}_{p^-},\mathcal{T}^{\mathrm{PL}}_p\colon \mathbb{R}^{P} \to \mathbb{R}$ by 
\begin{align*}
\mathcal{T}^{\mathrm{PL}}_{p^+}(f) &\coloneqq  f(p) - \mathrm{max}\{f(q)\colon\textrm{$p$ covers $q \in \widehat{P}$}\}; \\
\mathcal{T}^{\mathrm{PL}}_{p^-}(f) & \coloneqq  \mathrm{min}\{f(q)\colon\textrm{$q \in \widehat{P}$ covers $p$}\} - f(p); \\
\mathcal{T}^{\mathrm{PL}}_p(f) &\coloneqq  \mathcal{T}^{\mathrm{PL}}_{p^+}(f) - \mathcal{T}^{\mathrm{PL}}_{p^-}(f).
\end{align*} 
Of course these are defined so that $\mathcal{T}^{\mathrm{PL}}_{p^+}(f_I) = \mathcal{T}_{p^+}(I)$ and $\mathcal{T}^{\mathrm{PL}}_{p^-}(f_I) = \mathcal{T}_{p^-}(I)$ for all order ideals $I \in J(P)$. Then we can analogously define the \emph{piecewise-linear down-degree statistic} $\mathrm{ddeg}^{\mathrm{PL}}\colon \mathbb{R}^{P} \to \mathbb{R}$ by
\[ \mathrm{ddeg}^{\mathrm{PL}} \coloneqq  \sum_{p\in P}\mathcal{T}^{\mathrm{PL}}_{p^-}.\]
Since antichain cardinality and $\mathrm{ddeg}$ are the same for order ideals, $\mathrm{ddeg}^{\mathrm{PL}}$ will serve as our piecewise-linear analog of antichain cardinality. 

Observe that Stanley's transfer map~$\phi\colon\mathbb{R}^P \to \mathbb{R}^P$ is defined by $\phi(f)(p) \coloneqq  \mathcal{T}^{\mathrm{PL}}_{p^-}(f)$ for all $p\in P$ and $f \in \mathbb{R}^P$. Thus, $\mathrm{ddeg}^{\mathrm{PL}}(f) = \sum_{p\in P} \phi(f)(p)$ for any $f \in \mathbb{R}^P$. (Compare this to the natural piecewise-linear analog of the order ideal cardinality, which would be $f \mapsto \#P - \sum_{p\in P} f(p)$.) Also observe that for $T\in\mathrm{PP}^{\ell}(P)$, $\mathrm{ddeg}(T)=\ell\cdot\mathrm{ddeg}^{\mathrm{PL}}(\frac{1}{\ell}T)$.

Writing down-degree in terms of the toggleability statistics is extremely useful, because identities among the toggleability statistics lift from the combinatorial level to the whole order polytope, as the next lemma explains.

\begin{lemma} \label{lem:order_poly_toggle_eqs}
Suppose we have an equality of functions $J(P)\to \mathbb{R}$:
\[ \sum_{p\in P} c_{p^+}\mathcal{T}_{p^+} + c_{p^-}\mathcal{T}_{p^-}=\delta \]
where $c_{p^+},c_{p^-}$ for $p\in P$ and $\delta$ are constants in~$\mathbb{R}$. Then for any $f \in \mathcal{O}(P)$ we have
\[ \sum_{p\in P} c_{p^+}\mathcal{T}^{\mathrm{PL}}_{p^+}(f) + c_{p^-}\mathcal{T}^{\mathrm{PL}}_{p^-}(f) =\delta.\]
\end{lemma}
\begin{proof}
Consider the triangulation of the order polytope $\mathcal{O}(P)$ where the maximal simplices are given by $0\leq f(p_1) \leq f(p_2) \leq \cdots \leq f(p_n) \leq 1$ for all linear extensions $p_1, p_2, \ldots, p_n$ of~$P$. And fix one such maximal simplex~$\Delta$. It's clear that $\mathcal{T}^{\mathrm{PL}}_{p^-}$ and $\mathcal{T}^{\mathrm{PL}}_{p^+}$ for $p\in P$ are affine linear functions on~$\Delta$. Also, the vertices of~$\Delta$ are a subset of the vertices of $\mathcal{O}(P)$, i.e., are equal to $f_I$ for certain order ideals $I \in J(P)$. And by supposition we have
\[ \sum_{p\in P} c_{p^+}\mathcal{T}^{\mathrm{PL}}_{p^+}(f_I) + c_{p^-}\mathcal{T}^{\mathrm{PL}}_{p^-}(f_I) =\delta.\]
for any order ideal $I\in J(P)$. Thus, by linearity we have
\[ \sum_{p\in P} c_{p^+}\mathcal{T}^{\mathrm{PL}}_{p^+}(f) + c_{p^-}\mathcal{T}^{\mathrm{PL}}_{p^-}(f) =\delta.\]
for any $f\in \Delta$. But every $f  \in \mathcal{O}(P)$ is in one of these simplices~$\Delta$, so we are done.
\end{proof}

Furthermore, Striker's~\cite{striker2015toggle} key observation about $\mathcal{T}_p$ averaging to zero along rowmotion orbits continues to apply to piecewise-linear rowmotion.

\begin{lemma} \label{lem:pl_striker}
Let $f \in \mathcal{O}(P)$. Then for any $p\in P$ we have
\[\lim_{n \to \infty} \frac{1}{n} \sum_{i=0}^{n-1}\mathcal{T}_{p}^{\mathrm{PL}}( (\mathrm{row}^{\mathrm{PL}})^i(f)) = 0.\]
\end{lemma}
\begin{proof}
Let $p\in P$. It is immediate from the definitions of the piecewise-linear toggles and the toggleability statistics that for any $f \in \mathbb{R}^P$, we have $\mathcal{T}^{\mathrm{PL}}_{p^-}(\tau^{\mathrm{PL}}_p(f)) = \mathcal{T}^{\mathrm{PL}}_{p^+}(f)$. We claim that moreover for any $f \in \mathbb{R}^P$,  we have $\mathcal{T}^{\mathrm{PL}}_{p^-}(\mathrm{row}^{\mathrm{PL}}(f)) = \mathcal{T}^{\mathrm{PL}}_{p^+}(f)$. To see this, we need only need to recall that
\[\mathrm{row}^{\mathrm{PL}} = \tau_{p_1}^{\mathrm{PL}} \circ \tau_{p_2}^{\mathrm{PL}} \circ \cdots \circ \tau_{p_n}^{\mathrm{PL}},  \]
where $p_1,\ldots,p_n$ is some linear extension of $P$, and observe that: if $q \in P$ is greater than or incomparable to $p$, then $\mathcal{T}^{\mathrm{PL}}_{p^+}(\tau^{\mathrm{PL}}_q(f))=\mathcal{T}^{\mathrm{PL}}_{p^+}(f)$; and similarly, if $q \in P$ is less than or incomparable to $p$, then $\mathcal{T}^{\mathrm{PL}}_{p^-}(\tau^{\mathrm{PL}}_q(f))=\mathcal{T}^{\mathrm{PL}}_{p^-}(f)$.

The fact that $\mathcal{T}^{\mathrm{PL}}_{p^-}(\mathrm{row}^{\mathrm{PL}}(f)) = \mathcal{T}^{\mathrm{PL}}_{p^+}(f)$ means that for any $f\in \mathbb{R}^P$ we have
\[ \sum_{i=0}^{n}\mathcal{T}_{p}^{\mathrm{PL}}( (\mathrm{row}^{\mathrm{PL}})^i(f))=\mathcal{T}^{\mathrm{PL}}_{p^+}((\mathrm{row}^{\mathrm{PL}})^{n-1}(f))-\mathcal{T}^{\mathrm{PL}}_{p^-}(f) \]
due to a large number of cancellations. But since $\mathrm{row}^{\mathrm{PL}}$ preserves $\mathcal{O}(P)$, and $\mathcal{T}^\mathrm{PL}_{p^+}$ is bounded on $\mathcal{O}(P)$, we indeed have for any $f\in\mathcal{O}(P)$ that
\[\lim_{n \to \infty} \frac{1}{n} \sum_{i=0}^{n-1}\mathcal{T}_{p}^{\mathrm{PL}}( (\mathrm{row}^{\mathrm{PL}})^i(f)) =\lim_{n \to \infty} \frac{1}{n} \left ( \mathcal{T}^{\mathrm{PL}}_{p^+}((\mathrm{row}^{\mathrm{PL}})^{n-1}(f))-\mathcal{T}^{\mathrm{PL}}_{p^-}(f) \right) = 0, \]
as claimed.
\end{proof}

The above lemmas imply the analog of the antichain cardinality homomesy for piecewise-linear rowmotion acting on the order polytope for posets $P$ with $J(P)$ tCDE:

\begin{thm} \label{thm:row_ppart_homo}
Let $P$ be a poset for which $J(P)$ is tCDE with edge density~$\delta$. Then for any $f\in \mathcal{O}(P)$ we have
\[\lim_{n \to \infty} \frac{1}{n} \sum_{i=0}^{n-1}\mathrm{ddeg}^{\mathrm{PL}}( (\mathrm{row}^{\mathrm{PL}})^i(f)) = \delta.\]
\end{thm}
\begin{proof}
Let $P$ be a poset for which $J(P)$ is tCDE with edge density~$\delta$. By Proposition~\ref{prop:tcde_eq} we have an equality of functions~$J(P)\to\mathbb{R}$:
\[ \mathrm{ddeg} + \sum_{p\in P}c_p \mathcal{T}_p  = \delta,\]
for some coefficients $c_p \in \mathbb{Q}$. By Lemma~\ref{lem:order_poly_toggle_eqs} that means we have the following equality of functions $\mathcal{O}(P)\to\mathbb{R}$:
\[ \mathrm{ddeg}^{\mathrm{PL}} = \delta- \sum_{p\in P}c_p \mathcal{T}^{\mathrm{PL}}_p.\]
Thus for any $f \in \mathcal{O}(P)$ we have
\[\lim_{n \to \infty} \frac{1}{n} \sum_{i=0}^{n-1}\mathrm{ddeg}^{\mathrm{PL}}( (\mathrm{row}^{\mathrm{PL}})^i(f)) = \delta - \lim_{n \to \infty} \sum_{p \in P} \frac{c_p}{n} \sum_{i=0}^{n-1}\mathcal{T}_{p}^{\mathrm{PL}}( (\mathrm{row}^{\mathrm{PL}})^i(f)) = \delta,\]
by Lemma~\ref{lem:pl_striker}.
\end{proof}

\begin{remark} \label{rem:bounded}
We thank David Einstein for explaining the following to us. Suppose that $P$ is graded. Let $x_0 \in \mathcal{O}(P)$ be the point defined by $x_0(p) = \frac{r(p)+1}{r(P)+2}$ for all~$p \in P$. Note that $x_0$ is a fixed point of $\tau^{\mathrm{PL}}_p$ for every $p\in P$. (In fact, even without assuming $P$ is graded such an $x_0$ exists thanks to the Brouwer fixed-point theorem.) Moreover, it can be verified that for all $m \geq 0$ we have $\tau^{\mathrm{PL}}_p( m(\mathcal{O}(P)-x_0) + x_0)= m(\mathcal{O}(P)-x_0) + x_0$ for every $p\in P$. But every $f \in \mathbb{R}^P$ belongs to $m(\mathcal{O}(P)-x_0) + x_0$ for some $m\geq 0$ depending on $f$. Thus for any $f \in \mathbb{R}^P$ the sequence $f, \mathrm{row}^{\mathrm{PL}}(f), (\mathrm{row}^{\mathrm{PL}})^{2}(f), ...$ belongs to some bounded subset of $\mathbb{R}^P$, and so the conclusion of Lemma~\ref{lem:pl_striker} is true for all $f \in\mathbb{R}^P$. Combined with results we will prove in the next subsection, this will imply that the conclusion of Theorem~\ref{thm:row_ppart_homo} holds for all~$f \in\mathbb{R}^P$ for posets $P$ for which $J(P)$ is tCDE and which have a $2$-dimensional grid-like structure (see Lemma~\ref{lem:bi_toggle_eqs}).
\end{remark}

\begin{cor} \label{cor:row_ppart_homo}
Let $P$ be a poset for which $J(P)$ is tCDE with edge density~$\delta$. Then for any $\ell \geq 1$, $\mathrm{ddeg}$ is $\ell\cdot \delta$-mesic with respect to the action of $\mathrm{row}$ on $\mathrm{PP}^{\ell}(P)$.
\end{cor}
\begin{proof}
This is immediate from Theorem~\ref{thm:row_ppart_homo} and the observation that for any $P$-partition $T\in\mathrm{PP}^{\ell}(P)$ we have $\mathrm{ddeg}(T)=\ell\cdot\mathrm{ddeg}^{\mathrm{PL}}(\frac{1}{\ell}T)$.
\end{proof}

Of course, Theorem~\ref{thm:row_ppart_homo} and Corollary~\ref{cor:row_ppart_homo} apply when $P$ is a minuscule poset or a root poset of coincidental type. (They also apply to the other posets $P$ for which $J(P)$ is tCDE mentioned in Remark~\ref{rem:other_tcde}.)

So we've now seen (either conjecturally or provably) that for minuscule posets and root posets of coincidental type, the good behavior of rowmotion acting on order ideals (e.g., predictable and small order, cyclic sieving, and antichain cardinality homomesy) extends to rowmotion acting on $P$-partitions. Next, we want to examine how $P$-partition rowmotion behaves for the minuscule doppelg\"{a}nger pairs. Unsurprisingly, we make the following conjecture, directly generalizing  Conjecture~\ref{conj:min_dop_row}:

\begin{conj} \label{conj:min_dop_row_ppart}
Let $(P,Q) \in \{(\Lambda_{\mathrm{Gr}(k,n)},T_{k,n}),(\Lambda_{\mathrm{OG}(6,12)},\Phi^+(H_3)),(\Lambda_{\mathbb{Q}^{2n}},\Phi^+(I_2(2n)))\}$ be a minuscule doppelg\"{a}nger pair. Let $\ell \geq 1$. Then there is a bijection $\varphi$ between the $\mathrm{row}$-orbits of $\mathrm{PP}^{\ell}(P)$ and the $\mathrm{row}$-orbits of $\mathrm{PP}^{\ell}(Q)$ such that for any $\mathrm{row}$-orbit $\mathcal{O}\subseteq \mathrm{PP}^{\ell}(P)$ we have
\begin{enumerate}
\item $\#\mathcal{O}=\#\varphi(\mathcal{O})$;
\item $\sum_{T \in \mathcal{O}} \mathrm{ddeg}(T) = \sum_{T \in\varphi(\mathcal{O})} \mathrm{ddeg}(T)$.
\end{enumerate}
\end{conj}

We have checked Conjecture~\ref{conj:min_dop_row_ppart} by computer with $\ell=2,3,4$ for $(P,Q)=(\Lambda_{\mathrm{Gr}(k,n)},T_{k,n})$ with $n\leq 8$ and $k\leq n/2$, for $(P,Q)=(\Lambda_{\mathrm{OG}(6,12)},\Phi^+(H_3))$, and for $(\Lambda_{\mathbb{Q}^{2n}},\Phi^+(I_2(2n)))$ with $2\leq n \leq 8$.

In terms of proving Conjecture~\ref{conj:min_dop_row_ppart}, we note that the bijection of Hamaker et~al.~\cite{hamaker2018doppelgangers} {\bf does not} commute with rowmotion of $P$-partitions of height $\ell > 1$ (see~\cite{dao2019trapezoid}). Thus, we have no specific proposal for how one might prove Conjecture~\ref{conj:min_dop_row_ppart}.

\begin{remark} \label{rem:cyc_siev_conj_polys}
Conjecture~\ref{conj:min_dop_row_ppart} is consistent with Conjectures~\ref{conj:minuscule_row_cyc_siev} and~\ref{conj:root_poset_cyc_siev} in the sense that for $(P,Q) \in \{(\Lambda_{\mathrm{Gr}(k,2k)},\Phi^+(B_k)),(\Lambda_{\mathrm{OG}(6,12)},\Phi^+(H_3)),(\Lambda_{\mathbb{Q}^{2n}},\Phi^+(I_2(2n))\}$, we have
\[ \left [ F(P,\ell) \right ]_{q \mapsto q^2} = \mathrm{Cat}(\Phi,\ell;q)\]
for the appropriate cyclic sieving polynomials, just as in Remark~\ref{rem:cyc_siev_polys}.
\end{remark}

\begin{remark} \label{rem:ppart_homo_implies_mcde}
Note that Conjecture~\ref{conj:min_dop_row_ppart} implies Conjecture~\ref{conj:doppelganger_mcde}. This is because, as explained in Section~\ref{sec:ddeg_gf}, Conjecture~\ref{conj:doppelganger_mcde} is equivalent to
\[\sum_{T \in \mathrm{PP}^{\ell}(P)}\mathrm{ddeg}(T)=\sum_{T \in \mathrm{PP}^{\ell}(Q)}\mathrm{ddeg}(T),\]
and Conjecture~\ref{conj:min_dop_row_ppart} refines this equality to hold at the level of rowmotion orbits.
\end{remark}

We believe that Conjecture~\ref{conj:min_dop_row_ppart} is another instance of the minuscule doppelg\"{a}nger pairs pretending to have isomorphic comparability graphs, but this time we could not prove the corresponding statement even in this (presumably simpler) setting.

\begin{conj} \label{conj:com_graph_row_ppart}
Let $P$ and $Q$ be posets with $\mathrm{com}(P)\simeq\mathrm{com}(Q)$. Let $\ell \geq 1$. Then there is a bijection $\varphi$ between the $\mathrm{row}$-orbits of $\mathrm{PP}^{\ell}(P)$ and the $\mathrm{row}$-orbits of $\mathrm{PP}^{\ell}(Q)$ such that for any $\mathrm{row}$-orbit $\mathcal{O}\subseteq \mathrm{PP}^{\ell}(P)$ we have
\begin{enumerate}
\item $\#\mathcal{O}=\#\varphi(\mathcal{O})$;
\item $\sum_{T \in \mathcal{O}} \mathrm{ddeg}(T) = \sum_{T \in\varphi(\mathcal{O})} \mathrm{ddeg}(T)$.
\end{enumerate}
\end{conj}

Possibly the same general approach from the proof of Proposition~\ref{prop:com_graph_row} could be used to prove Conjecture~\ref{conj:com_graph_row_ppart} but the details seem more difficult. For this conjecture it might be useful to work with the ``antichain toggles'' (and their piecewise-linear analogs) described in~\cite{joseph2017antichain} (see also~\cite{joseph2019birationalantichain}). We note that antichain toggles are a special case of ``independent set toggles,'' which have also received a fair amount of attention~\cite{einstein2016noncrossing, striker2018rowmotion, joseph2018toggling}. Thomas and Williams~\cite{thomas2018independence} abstracted and generalized independent set toggles with their definition of ``independence posets,'' so it might also be worthwhile to understand Proposition~\ref{prop:com_graph_row} and Conjecture~\ref{conj:com_graph_row_ppart} in the context of independence posets. At any rate, we have checked Conjecture~\ref{conj:com_graph_row_ppart} by computer with $\ell=2,3,4$ for all connected posets on $8$ or fewer vertices.

\subsection{Birational rowmotion}

Now we consider birational rowmotion. Before we give any definitions, let's review the history. Birational rowmotion was introduced by Einstein and Propp~\cite{einstein2013combinatorial, einstein2014piecewise}. It is defined by ``detropicalizing'' the piecewise-linear operations we studied in the previous subsection. The idea of starting with some operation in algebraic combinatorics, expressing it in terms of sums and maxes, and then ``deptropicalizing'' to get a subtraction-free rational expression, goes back at least to~\cite{kirillov1995groups, kirillov2001tropical}. One reason for considering birational liftings of combinatorial operations is that they really are direct generalizations: results proved at the birational level tropicalize to the piecewise-linear level, and then specialize to the combinatorial level.

Birational rowmotion, especially in the case of the rectangle poset, was thoroughly investigated by Grinberg and Roby~\cite{grinberg2016birational1, grinberg2015birational2} (see also Musiker and Roby~\cite{musiker2018paths}). For the rectangle, birational rowmotion enjoys remarkable integrability properties related to cluster algebra \emph{Y-system dynamics} (a.k.a., \emph{Zamolodchikov periodicity})~\cite{volkov2007periodicity}. Birational rowmotion was also the principal example motivating the recent definition of \emph{R-systems} due to Galashin and Pylyavskyy~\cite{galashin2017rsystems}.

Before we introduce birational rowmotion, we need to slightly extend the description of piecewise-linear rowmotion from the previous subsection. Namely, we fix two parameters $\alpha^{\mathrm{PL}}, \omega^{\mathrm{PL}} \in \mathbb{R}$. We then view any $f \in \mathbb{R}^P$ as belonging also to $\mathbb{R}^{\widehat{P}}$ by the convention that $f(\, \widehat{0}\, )=\alpha^{\mathrm{PL}}$ and $f(\, \widehat{1}\, )=\omega^{\mathrm{PL}}$. Otherwise the piecewise-linear operations are exactly as we described above. The case we considered before is of course the specialization  $\alpha^{\mathrm{PL}}=0$ and $\omega^{\mathrm{PL}}=1$.

Now we ``detropicalize'' everything. This means we replace sums by products, maxes by sums, et cetera. What we will obtain are subtraction free rational expressions. From the perspective of algebraic geometry, it would probably be best to work with $\mathbb{C}$-valued functions on our poset $P$; but then we would have to consider the expressions we obtain via detropicalization as birational maps, which are not necessarily defined everywhere because of ``division by zero'' issues. An easy (but perhaps clumsy) way to get around this problem is to instead work with $\mathbb{R}_{>0}$-valued functions on $P$, for which all expressions will always be everywhere well-defined. We denote the set of such functions by $\mathbb{R}_{>0}^{P}$.

Fix parameters $\alpha^{\mathrm{B}}, \omega^{\mathrm{B}} \in \mathbb{R}_{>0}$. View any element $f\in\mathbb{R}_{>0}^{P}$ as also belonging to $\mathbb{R}_{>0}^{\widehat{P}}$ by the convention that $f(\, \widehat{0}\, )=\alpha^{\mathrm{B}}$ and $f(\, \widehat{1}\, )=\omega^{\mathrm{B}}$. Then we define for each $p\in P$ the \emph{birational toggle at $p$} $\tau_{p}^{\mathrm{B}}\colon \mathbb{R}_{>0}^{P}\to \mathbb{R}_{>0}^{P}$ by
\[\tau_{p}^{\mathrm{B}}(f)(q) \coloneqq  \begin{cases} f(q) &\textrm{if $p \neq q$}; \\ \displaystyle \frac{1}{f(p)} \cdot \frac{\displaystyle \sum_{\textrm{$p$ covers $x\in \widehat{P}$}} f(x)}{\displaystyle \sum_{\textrm{$x\in \widehat{P}$ covers $p$}} \frac{1}{f(x)}} &\textrm{if $p=q$}.\end{cases}\]
These toggles are once again involutions, and $\tau_{p}^{\mathrm{B}}$ commutes with $\tau_{q}^{\mathrm{B}}$ unless there is a covering relation between them. We define \emph{birational rowmotion} $\mathrm{row}^{\mathrm{B}}\colon \mathbb{R}_{>0}^{P}\to \mathbb{R}_{>0}^{P}$ by
\[\mathrm{row}^{\mathrm{B}} \coloneqq  \tau_{p_1}^{\mathrm{B}} \circ \tau_{p_2}^{\mathrm{B}} \circ \cdots \circ \tau_{p_n}^{\mathrm{B}},  \]
where $p_1,\ldots,p_n$ is any linear extension of $P$.

 Unsurprisingly, birational rowmotion behaves well for minuscule posets and root posets of coincidental type. We'll focus on the same two aspects of rowmotion for our posets of interest: order/orbit structure, and homomesies.

It is not so meaningful to talk about the orbit structure when we will have uncountably many orbits, but it can be very interesting to study the order of birational rowmotion because there is no a priori reason for birational rowmotion to have finite order and indeed for most posets it does not. (Actually, already piecewise-linear rowmotion acting on all of $\mathbb{R}^P$ will usually fail to have finite order.)

The following two conjectures about the order of birational rowmotion for minuscule posets and root posets of coincidental type were essentially stated in the work of Grinberg-Roby~\cite{grinberg2016birational1, grinberg2015birational2}.

 \begin{conj} \label{conj:min_bi_row_ord}
 Let $P$ be a minuscule poset, with $h \coloneqq  r(P)+2$ its Coxeter number. Then the order of  $\mathrm{row}^{\mathrm{B}}\colon \mathbb{R}_{>0}^{P}\to \mathbb{R}_{>0}^{P}$ is $h$.
 \end{conj}
 
  \begin{conj} \label{conj:root_bi_row_ord}
 Let $P$ be a root poset of coincidental type, with $h \coloneqq  r(P)+2$ its Coxeter number. Then the order of $\mathrm{row}^{\mathrm{B}}\colon \mathbb{R}_{>0}^{P}\to \mathbb{R}_{>0}^{P}$ divides $2h$.
 \end{conj}
 
Results of Grinberg-Roby~\cite{grinberg2016birational1, grinberg2015birational2} are almost enough to establish Conjecture~\ref{conj:min_bi_row_ord} in a case-by-case manner. Namely, they proved this conjecture for $P=\Lambda_{\mathrm{Gr}(k,n)}$~\cite[Theorem~30]{grinberg2015birational2}, for $P=\Lambda_{\mathrm{OG}(n,2n)}$~\cite[Theorem~58]{grinberg2015birational2}, for $P=\Lambda_{\mathbb{Q}^{2n}}$~\cite[\S10]{grinberg2016birational1}, and for $P=\Lambda_{\mathbb{OP}^2}$~\cite[\S13]{grinberg2015birational2}. For the only other minuscule poset $P=\Lambda_{G_{\omega}(\mathbb{O}^3,\mathbb{O}^6)}$ coming from Type $E_7$, all evidence suggests the conjecture is true in this case as well (see~\cite[\S13]{grinberg2015birational2}), but this poset was simply too big to be worked out completely by computer. In a paper in preparation, Okada~\cite{okada2019birational} establishes Conjecture~\ref{conj:min_bi_row_ord} by using the coordinate transformation and lattice paths interpretation of Musiker and Roby~\cite{musiker2018paths}. Of course, a uniform proof of Conjecture~\ref{conj:min_bi_row_ord} would be preferred.
 
Results of Grinberg-Roby~\cite{grinberg2016birational1, grinberg2015birational2} again establish much of Conjecture~\ref{conj:root_bi_row_ord}: as explained in~\cite[\S13]{grinberg2015birational2}, they proved that this conjecture holds for $P=\Phi^+(A_n)$, $P=\Phi^+(H_3)$, and $P=\Phi^+(I_2(m))$. For the remaining case $P=\Phi^+(B_n)$, it is explicitly stated as a conjecture in~\cite[Conjecture~73]{grinberg2015birational2} that birational rowmotion should have order~$h=2n$.  We note that for the root poset $P=\Phi^+(D_4)$, birational rowmotion does not (appear) to have finite order. 
 
Now let's consider homomesies for birational rowmotion. The birational analog of the order ideal cardinality homomesy (and its refinements) for the rectangle were established by Einstein and Propp~\cite{einstein2013combinatorial, einstein2014piecewise} (see also Musiker-Roby~\cite{musiker2018paths}). We are interested in the birational analog of the antichain cardinality homomesy, which had not been previously studied in any publicly available article (but see Remark~\ref{rem:darij} below). So let's define birational analogs of the toggleabilitiy and down-degree statistics. For $p\in P$ we define the \emph{birational toggleability statistics} $\mathcal{T}^{\mathrm{B}}_{p^+}, \mathcal{T}^{\mathrm{B}}_{p^-},\mathcal{T}^{\mathrm{B}}_p\colon \mathbb{R}_{>0}^{P} \to \mathbb{R}_{>0}$ by 
\begin{align*}
\mathcal{T}^{\mathrm{B}}_{p^+}(f) &\coloneqq  f(p) \cdot \frac{1}{\displaystyle \sum_{\textrm{$p$ covers $q \in \widehat{P}$}}f(q)} ; \\
\mathcal{T}^{\mathrm{B}}_{p^-}(f) & \coloneqq  \frac{1}{f(p)} \cdot \frac{1}{\displaystyle \sum_{\textrm{$q \in \widehat{P}$ covers $p$}} \frac{1}{f(q)}}; \\
\mathcal{T}^{\mathrm{B}}_p(f) &\coloneqq  \frac{\mathcal{T}^{\mathrm{PL}}_{p^+}(f)}{\mathcal{T}^{\mathrm{PL}}_{p^-}(f)}.
\end{align*} 
Then we define the \emph{birational down-degree statistic} $\mathrm{ddeg}^{\mathrm{B}}\colon \mathbb{R}_{>0}^{P} \to \mathbb{R}_{>0}$ by
\[ \mathrm{ddeg}^{\mathrm{B}} \coloneqq  \prod_{p\in P}\mathcal{T}^{\mathrm{B}}_{p^-}.\]

Again, writing down-degree in terms of toggleability statistics allows us to relate the tCDE property to birational antichain cardinality homomesy, as the following series of lemmas explain.

\begin{lemma} \label{lem:bi_toggle_eqs}
Let $P$ be a poset for which each element covers, and is covered by, at most two elements. Suppose we have an equality of functions $J(P)\to \mathbb{R}$:
\[ \sum_{p\in P} c_{p^+}\mathcal{T}_{p^+} + c_{p^-}\mathcal{T}_{p^-}=\delta \]
where $c_{p^+},c_{p^-}$ for $p\in P$ and $\delta$ are constants in~$\mathbb{Q}$. Then we have an analogous equality of functions $\mathbb{R}_{>0}^{P} \to \mathbb{R}_{>0}$:
\[ \prod_{p\in P} (\mathcal{T}^{\mathrm{B}}_{p^+})^{c_{p^+}} \cdot (\mathcal{T}^{\mathrm{B}}_{p^-})^{c_{p^-}} = \left( \frac{\omega^{\mathrm{B}}}{\alpha^{\mathrm{B}}}\right)^{\delta}. \]
(Here we use principal roots for non-integral powers of numbers in $\mathbb{R}_{>0}$.)
\end{lemma}
\begin{proof}
Let $P$ and $c_{p^+},c_{p^-},\delta \in \mathbb{Q}$ be as in the statement of the lemma. First we will prove the piecewise-linear version of the result that we want. By repeatedly applying the rule $\mathrm{min}(a,b)=a+b-\mathrm{max}(a,b)$, we can write an equality of functions $\mathbb{R}^P\to \mathbb{R}$:
\[ \sum_{p\in P} c_{p^+}\mathcal{T}^{\mathrm{PL}}_{p^+}(f) + c_{p^-}\mathcal{T}^{\mathrm{PL}}_{p^-}(f)= \sum_{p \in P}c'_p f(p) + \hspace{-0.6cm} \sum_{\substack{p,q \in P \\ \textrm{$p$ and $q$}\\ \textrm{incomparable}}} \hspace{-0.6cm} c'_{p,q} \, \mathrm{max}(f(p),f(q)) + c'_{\alpha} \, \alpha^{\mathrm{PL}} + c'_{\omega}\omega^{\mathrm{PL}}, \]
where $c'_{p}, c'_{p,q}, c'_{\alpha}, c'_{\omega}\in \mathbb{Q}$ are constants. (This is the step where we use the fact that each element covers, and is covered by, at most two elements in an essential way.) We want to show that $c'_{p}=0$ for all $p$, $c'_{p,q}=0$ for all $p,q$, and $c'_{\alpha}=-\delta$ and $c'_{\omega}=\delta$. 

First we address $c'_{\alpha}$ and $c'_{\omega}$. By considering the empty order ideal~$\varnothing \in J(P)$, we see that
\[  \sum_{p\in P} c_{p^+}\mathcal{T}_{p^+}(\varnothing) + c_{p^-}\mathcal{T}_{p^-}(\varnothing) =\delta \]
implies $\sum_{\textrm{$p$ minimal in $P$}} c_{p^+}=-\delta$, which in turn means $c'_{\alpha}=-\delta$. Similarly, by considering the full order ideal $P\in J(P)$, we see that
\[  \sum_{p\in P} c_{p^+}\mathcal{T}_{p^+}(P) + c_{p^-}\mathcal{T}_{p^-}(P) =\delta \]
implies $\sum_{\textrm{$p$ maximal in $P$}} c_{p^-}=\delta$, which in turn means $c'_{\omega}=\delta$. 

From now on we specialize $\omega^{\mathrm{PL}}\coloneqq 1$ and $\alpha^{\mathrm{PL}}\coloneqq 0$.

Next we will prove $c'_{p,q}=0$ for all $p,q$. So let $p$ and $q$ be incomparable elements of~$P$. Consider the subset $X\subseteq \mathcal{O}(P)$ of the order polytope of $P$ where:
\begin{itemize}
\item $f(x)$ can be anything between $0$ and $1$ if $x=p$ or $x=q$;
\item $f(x) = 0$ if $x < p$ or $x < q$;
\item $f(x) = 1$ if $x \not \leq p$ and $x \not \leq q$.
\end{itemize}
Then we have the equality of functions $X \to \mathbb{R}$:
\[ \sum_{x\in P} c_{x^+}\mathcal{T}^{\mathrm{PL}}_{x^+}(f) + c_{x^-}\mathcal{T}^{\mathrm{PL}}_{x^-}(f)= c'_{p,q}\, \mathrm{max}(f(p),f(q)) + c''_p \, f(p) + c''_q \, f(q) +\delta'\]
for some constants $c''_p, c''_q, \delta' \in \mathbb{Q}$. But by Lemma~\ref{lem:order_poly_toggle_eqs} we also have an equality of functions $\mathcal{O}(P) \to \mathbb{R}$:
\[\sum_{x\in P} c_{x^+}\mathcal{T}^{\mathrm{PL}}_{x^+}(f) + c_{x^-}\mathcal{T}^{\mathrm{PL}}_{x^-}(f)=\delta.\]
Then by considering in turn the subsets of $X$ where $f(p)\geq f(q)$ and where $f(q)\geq f(p)$, we see that we must have $c''_p= c''_q=c'_{p,q}=0$.

We can similarly prove $c'_p=0$ for all $p$. So let $p\in P$. And consider the subset $X\subseteq \mathcal{O}(P)$ of the order polytope of $P$ where:
\begin{itemize}
\item $f(x)$ can be anything between $0$ and $1$ if $x=p$;
\item $f(x) = 0$ if $x < p$ ;
\item $f(x) = 1$ if $x \not \leq p$.
\end{itemize}
Since we've already proved $c'_{p,q}=0$, we have the equality of functions $X\to \mathbb{R}$:
\[ \sum_{x\in P} c_{x^+}\mathcal{T}^{\mathrm{PL}}_{x^+}(f) + c_{x^-}\mathcal{T}^{\mathrm{PL}}_{x^-}(f)=c'_p \, f(p) +\delta'\]
for some constant $\delta' \in \mathbb{Q}$. Then Lemma~\ref{lem:order_poly_toggle_eqs} again easily implies that $c'_p=0$.

So we have now proven the piecewise-linear version of the result we want, namely, the equality of functions $\mathbb{R}^P\to \mathbb{R}$:
\[ \sum_{p\in P} c_{p^+}\mathcal{T}^{\mathrm{PL}}_{p^+} + c_{p^-}\mathcal{T}^{\mathrm{PL}}_{p^-}=\delta(\omega^{\mathrm{PL}}-\alpha^{\mathrm{PL}}).\]
Deducing the birational result is actually now very easy. By repeatedly applying the rule $\frac{1}{\frac{1}{a}+\frac{1}{b}}=ab \cdot \frac{1}{a+b}$, we can write an equality of functions $\mathbb{R}_{>0}^P\to \mathbb{R}_{>0}$:
\[  \prod_{p\in P} (\mathcal{T}^{\mathrm{B}}_{p^+}(f))^{c_{p^+}} \cdot (\mathcal{T}^{\mathrm{B}}_{p^-}(f))^{c_{p^-}} = \prod_{p\in P} f(p)^{c'_p} \cdot \hspace{-0.6cm}  \prod_{\substack{p,q \in P \\ \textrm{$p$ and $q$}\\ \textrm{incomparable}}} \hspace{-0.6cm} (f(p)+f(q))^{c'_{p,q}} \cdot (\alpha^{\mathrm{B}})^{c'_{\alpha}} \cdot (\omega^{\mathrm{B}})^{c'_{\omega}}, \]
where $c'_{p}, c'_{p,q}, c'_{\alpha}, c'_{\omega}\in \mathbb{Q}$ are constants. Importantly, these must the same constants as above because they are obtained in the same way. Therefore we get the equality of functions $\mathbb{R}_{>0}^{P} \to \mathbb{R}_{>0}$:
\[ \prod_{p\in P} (\mathcal{T}^{\mathrm{B}}_{p^+})^{c_{p^+}} \cdot (\mathcal{T}^{\mathrm{B}}_{p^-})^{c_{p^-}} = \left( \frac{\omega^{\mathrm{B}}}{\alpha^{\mathrm{B}}}\right)^{\delta},\]
as claimed.
\end{proof}

\begin{lemma} \label{lem:bi_striker}
Let $\mathcal{O}$ be a finite orbit of $\mathrm{row}^{\mathrm{B}}\colon \mathbb{R}^{P}_{>0} \to \mathbb{R}^{P}_{>0}$. Then for any $p \in P$ we have
\[\prod_{f \in \mathcal{O}} \mathcal{T}^{\mathrm{B}}_p(f) = 1.\]
\end{lemma}
\begin{proof}
The same argument as in the proof of Lemma~\ref{lem:pl_striker} proves that for any $f \in \mathbb{R}^{P}_{>0}$, $\mathcal{T}^{\mathrm{B}}_{p^-}(\mathrm{row}^{\mathrm{B}}(f)) = \mathcal{T}^{\mathrm{B}}_{p^+}(f)$. Thus multiplying $\mathcal{T}^{\mathrm{B}}_p(f)$ along a finite rowmotion orbit leads to all terms cancelling.
\end{proof}

\begin{thm} \label{thm:row_bi_homo}
Let $P$ be a poset for which each element covers, and is covered by, at most two elements, and for which $J(P)$ is tCDE with edge density~$\delta$. Let $\mathcal{O}$ be a finite orbit of $\mathrm{row}^{\mathrm{B}}\colon \mathbb{R}^{P}_{>0} \to \mathbb{R}^{P}_{>0}$. Then
\[ \prod_{f \in \mathcal{O}} \mathrm{ddeg}^{\mathrm{B}}(f) = \left( \frac{\omega^{\mathrm{B}}}{\alpha^{\mathrm{B}}}\right)^{\delta\cdot \#\mathcal{O}}.\]
\end{thm}
\begin{proof}

Let $P$ be as in the statement of the theorem. By Proposition~\ref{prop:tcde_eq} we have an equality of functions~$J(P)\to\mathbb{R}$:
\[ \mathrm{ddeg} + \sum_{p\in P}c_p \mathcal{T}_p  = \delta,\]
for some coefficients $c_p \in \mathbb{Q}$. By Lemma~\ref{lem:bi_toggle_eqs} that means we have the following equality of functions $\mathbb{R}^{P}_{>0}\to\mathbb{R}_{>0}$:
\[ \mathrm{ddeg}^{\mathrm{B}} =  \left( \frac{\omega^{\mathrm{B}}}{\alpha^{\mathrm{B}}}\right)^{\delta} \cdot \prod_{p\in P} (\mathcal{T}^{\mathrm{B}}_p)^{-c_p}.\]
Let $\mathcal{O}$ be a finite orbit of $\mathrm{row}^{\mathrm{B}}\colon \mathbb{R}^{P}_{>0} \to \mathbb{R}^{P}_{>0}$. Then by Lemma~\ref{lem:bi_striker} we have
\[ \prod_{f \in \mathcal{O}} \mathrm{ddeg}^{\mathrm{B}}(f) = \left( \frac{\omega^{\mathrm{B}}}{\alpha^{\mathrm{B}}}\right)^{\delta\cdot \#\mathcal{O}},\]
as claimed.
\end{proof}

Of course, Theorem~\ref{thm:row_bi_homo} applies when $P$ is a minuscule poset or a root poset of coincidental type. (It also applies to the other posets $P$ for which $J(P)$ is tCDE mentioned in Remark~\ref{rem:other_tcde}.)

\begin{remark} \label{rem:refined}
Let $P\coloneqq [a]\times [b]$ be the rectangle. For this poset, Propp and Roby~\cite{propp2015homomesy} proved some refinements of the antichain cardinality homomesy. These refinements also hold at the birational level, as we now explain. For fixed $1 \leq i \leq a$, let $X_i \subseteq P$ consist of the elements of the form $(i,j)$ for $1 \leq j \leq b$. Propp-Roby~\cite[Theorem 27]{propp2015homomesy} proved that for any $1\leq i \leq a$, the statistic $I \mapsto \#(\mathrm{max}(I) \cap X_i)$ is $\frac{b}{a+b}$-mesic with respect to the action of rowmotion acting on $J(P)$. This statistic can also be written as $\sum_{p \in X_i} \mathcal{T}_{p^-}$. As explained in~\cite[p.~23]{chan2017expected}, it is easy to see that for any $1 \leq i \leq a-1$, we have the following equality of functions $J(P) \to \mathbb{R}$:
\[\sum_{p \in X_i} \mathcal{T}_{p^-} = \sum_{p \in X_{i+1}} \mathcal{T}_{p^+}.\]
Lemma~\ref{lem:bi_toggle_eqs} thus gives the following equality of functions $\mathbb{R}^P_{>0}\to\mathbb{R}_{>0}$:
\[ \prod_{p \in X_i} \mathcal{T}^{\mathrm{B}}_{p^-} = \prod_{p \in X_{i+1}} \mathcal{T}^{\mathrm{B}}_{p^+}.\]
And Lemma~\ref{lem:bi_striker} then says that for any orbit $\mathcal{O}$ of $\mathrm{row}^{\mathrm{B}}\colon \mathbb{R}^P_{>0}\to \mathbb{R}^P_{>0}$, we have
\[ \prod_{f\in\mathcal{O}} \prod_{p \in X_i} \mathcal{T}^{\mathrm{B}}_{p^-}(f) =\prod_{f\in\mathcal{O}} \prod_{p \in X_{i+1}} \mathcal{T}^{\mathrm{B}}_{p^-}(f).\]
Together with Theorem~\ref{thm:row_bi_homo}, this means that for any $\mathrm{row}^{\mathrm{B}}$-orbit $\mathcal{O}$ and any $1\leq i \leq a$ we have
\[ \prod_{f\in\mathcal{O}} \prod_{p \in X_i} \mathcal{T}^{\mathrm{B}}_{p^-}(f) =  \left( \frac{\omega^{\mathrm{B}}}{\alpha^{\mathrm{B}}}\right)^{\frac{b}{a+b}\cdot \#\mathcal{O}},\]
which is the birational analog of the refined homomesy. Of course, the same works for the subsets $Y_j\subseteq P$ consisting of elements of the form $(i,j)$ for fixed $j$ as well.
\end{remark}

\begin{remark} \label{rem:darij}
For the specific case of the rectangle, the birational antichain cardinality homomesy (as well as the refined homomesies mentioned in Remark~\ref{rem:refined}) had been obtained independently in prior unpublished work of Darij Grinberg~\cite{grinberg2018homomesy}. Grinberg's arguments used the techniques of~\cite{grinberg2016birational1, grinberg2015birational2} and hence relied on the strong integrability properties exhibited by birational rowmotion of the rectangle poset. Also, after the first version of this paper was posted online, Joseph and Roby~\cite{joseph2020stanleythomas} extended ``Stanley--Thomas word'' approach of Propp-Roby~\cite{propp2015homomesy} from the combinatorial to the birational setting. They obtained these same rectangle homomesies in the process. Finally, in a paper in preparation, Okada~\cite{okada2019birational} establishes many homomesies for birational rowmotion of minuscule posets, including the antichain cardinality homomesy. He does this by adapting the coordinate system and lattice paths interpretation of Musiker-Roby~\cite{musiker2018paths} from the rectangle to the other minuscule posets. All of these other approaches require exact formulas for birational rowmotion of special posets. Our arguments don't rely on any such integrability assumptions and apply more generally.
\end{remark}

Finally, let's end our consideration of birational rowmotion by bringing the minuscule doppelg\"{a}ngers back into the picture. We offer the following variant of Conjecture~\ref{conj:min_dop_row} (note that it is a \emph{variant} and not a direct generalization):

\begin{conj} \label{conj:min_dop_bi_row}
Let $(P,Q) \in \{(\Lambda_{\mathrm{Gr}(k,n)},T_{k,n}),(\Lambda_{\mathrm{OG}(6,12)},\Phi^+(H_3)),(\Lambda_{\mathbb{Q}^{2n}},\Phi^+(I_2(2n)))\}$ be a minuscule doppelg\"{a}nger pair. Then,
\begin{itemize}
\item both $\mathrm{row}^{\mathrm{B}}\colon \mathbb{R}^{P}_{>0}\to \mathbb{R}^{P}_{>0}$ and $\mathrm{row}^{\mathrm{B}}\colon \mathbb{R}^{Q}_{>0}\to \mathbb{R}^{Q}_{>0}$ have order~$r(P)+2=r(Q)+2$;
\item for any $f \in \mathbb{R}^{P}_{>0}, g \in \mathbb{R}^{Q}_{>0}$, we have 
\[\prod_{i=0}^{r(P)+1} \mathrm{ddeg}^{\mathrm{B}}( (\mathrm{row}^{\mathrm{B}})^i(f)) = \left( \frac{\omega^{\mathrm{B}}}{\alpha^{\mathrm{B}}}\right)^{\#P} = \left( \frac{\omega^{\mathrm{B}}}{\alpha^{\mathrm{B}}}\right)^{\#Q} = \prod_{i=0}^{r(Q)+1} \mathrm{ddeg}^{\mathrm{B}}( (\mathrm{row}^{\mathrm{B}})^i(g)).\]
\end{itemize}
\end{conj}

Recall that the equalities $\#P=\#Q$ and $r(P)=r(Q)$ are forced by $P$ and $Q$ being doppelg\"{a}ngers (see Proposition~\ref{prop:doppelganger_basics}). The only poset appearing in Conjecture~\ref{conj:min_dop_bi_row} for which the results discussed above are not enough to resolve the conjecture is the trapezoid $T_{k,n}$. And in fact, Nathan Williams had previously conjectured (see~\cite[Conjecture 75]{grinberg2015birational2}) that birational rowmotion should have finite order for the trapezoid. So the only really ``new'' part of this conjecture is the assertion that $T_{k,n}$ exhibits the birational antichain cardinality homomesy. As for partial results: Grinberg and Roby~\cite[Theorem~74]{grinberg2015birational2} proved that birational rowmotion has order $r(P)+2$ for $P=T_{k,2k+1}$; and of course Theorem~\ref{thm:row_bi_homo} applies to $P=T_{k,2k}\simeq\Phi^+(B_k)$. We have checked Conjecture~\ref{conj:min_dop_bi_row} by computer for $T_{k,n}$ for $n\leq 7$ and $k \leq n/2$.

\begin{remark}
Even though Conjecture~\ref{conj:min_dop_bi_row} does not directly imply either Conjecture~\ref{conj:min_dop_row} or Conjecture~\ref{conj:min_dop_row_ppart} (because it says nothing about orbit structure), it actually does still imply Conjecture~\ref{conj:doppelganger_mcde}. Indeed, via tropicalization and by restricting to the rational points in the order polytope, the birational analog of the antichain cardinality homomesy implies that $\mathrm{ddeg}$ is homomesic with respect to the action of $\mathrm{row}$ on $\mathrm{PP}^{\ell}(P)$ and on $\mathrm{PP}^{\ell}(Q)$ for a minuscule doppelg\"{a}nger pair $(P,Q)$. This in turn implies
\[\sum_{T \in \mathrm{PP}^{\ell}(P)}\mathrm{ddeg}(T)=\sum_{T \in \mathrm{PP}^{\ell}(Q)}\mathrm{ddeg}(T),\]
which as explained in Remark~\ref{rem:ppart_homo_implies_mcde} above is equivalent to Conjecture~\ref{conj:doppelganger_mcde}.
\end{remark}

Something like Conjecture~\ref{conj:min_dop_bi_row} should be true for posets with isomorphic comparability graphs, but because most posets do not have finite birational rowmotion order it becomes slightly harder to formulate the corresponding statement in this case. It does appear that, e.g., if $P$ has finite birational rowmotion order and $Q$ satisfies $\mathrm{com}(P)\simeq \mathrm{com}(Q)$, then $Q$ has finite birational rowmotion order as well.

\bibliography{minuscule_cde}{}
\bibliographystyle{plain}

\end{document}